\documentclass{amsart}
\usepackage{moreverb,url}
\usepackage{amssymb,amsmath,amsfonts}
\usepackage{bm}
\usepackage[font=scriptsize, labelfont=bf]{caption}
\usepackage{float}
\usepackage{varwidth}
\usepackage{booktabs}
\usepackage{afterpage}
\usepackage{mathrsfs}
\usepackage{subcaption}
\usepackage[squaren]{SIunits}
\usepackage{verbatim}
\usepackage[square,numbers]{natbib}
\usepackage{stmaryrd}
\usepackage{fancyhdr}
\usepackage{syntax,etoolbox}
\usepackage{graphicx}
\usepackage[dvipsnames]{xcolor}
\usepackage{grffile}
\usepackage{hyperref}
\hypersetup{colorlinks=true,allcolors=blue}
\usepackage[square,numbers]{natbib}
\usepackage[left=1in, right=1in, bottom=1in,top=1in]{geometry}
\usepackage{syntax,etoolbox}
\usepackage{algorithmicx}
\usepackage{algorithm}
\usepackage{algpseudocode}
\usepackage{tikz}
\allowdisplaybreaks

\newtheorem{theorem}{Theorem}[section]

\newtheorem{innercustomgeneric}{\customgenericname}
\providecommand{\customgenericname}{}
\newcommand{\newcustomproblem}[2]{%
	\newenvironment{#1}[1]
	{%
		\renewcommand\customgenericname{#2}%
		\renewcommand\theinnercustomgeneric{##1}%
		\innercustomgeneric
	}
	{\endinnercustomgeneric}
}

\newcustomproblem{customprob}{Problem}

\newcommand*{\bqed}{\hfill\ensuremath{\blacksquare}}%

\newcommand{\vertiii}[1]{{\left\vert\kern-0.25ex\left\vert\kern-0.25ex\left\vert #1 
		\right\vert\kern-0.25ex\right\vert\kern-0.25ex\right\vert}}

\def\dd{\, \mathrm{d}}

\algblock{Input}{EndInput}

\algnotext{EndInput}

\algblock{Output}{EndOutput}

\algnotext{EndOutput}

\algblock{Iterate}{EndIterate}

\algnotext{EndIterate}

\makeatletter
\newcommand{\addresseshere}{%
	\enddoc@text\let\enddoc@text\relax
}
\makeatother

\begin{document}
	
	
	\title[Mixed finite element methods for fourth order obstacle problems in elasticity]{Geometrical mixed finite element methods for fourth order obstacle problems in linearised elasticity}
	

	\author[Paolo Piersanti]{Paolo Piersanti}
	\address{P.P. School of Science and Engineering, The Chinese University of Hong Kong (Shenzhen), 2001 Longxiang Blvd., Longgang District, Shenzhen, China}
	\email[Corresponding author]{ppiersanti@cuhk.edu.cn}
	
	\author[Tianyu Sun]{Tianyu Sun}
\address{T.S. Department of Mathematics and Institute for Scientific Computing and Applied Mathematics, Indiana University Bloomington, 729 East Third Street, Bloomington, Indiana, USA}
\email{ts19@iu.edu}

\today

\begin{abstract}
This paper introduces a mixed Finite Element Method aimed at approximating solutions to fourth-order variational problems with constraints.

We first address the biharmonic obstacle problem and propose an error convergence framework that offers an alternative to the established approaches by Ciarlet \& Raviart and Ciarlet \& Glowinski. Our method emphasises improved ease of numerical implementation, which may enhance practical usability.

Next, we investigate a two-dimensional variational problem involving linearly elastic shallow shells constrained within a specified half-space. We find that, if the middle surface of the shallow shell under consideration is non-flat, the Finite Element discretisation cannot rely solely on Courant triangles, indicating a variation in approach based on the geometric features of the problem. This constitutes a counterexample to the statement that \emph{solutions of fourth order linear problems can be approximated by solely resorting to Courant triangles if one considers the mixed formulation of the original problem}. We notice that this counterexample arises in connection with the lack of rigidity of linearly elastic shallow shells middle surface.

Our theoretical findings are supported by a series of supplementary numerical experiments, illustrating the practicality of the proposed methods.

\smallskip


\end{abstract}

\maketitle

\tableofcontents

\section{Introduction}
Fourth order variational problems appear in a manifold of applications like, for instance, elasticity~\cite{CGN+2021,GWX19,Lewicka2023}, Fluid-Structure-Interaction~\cite{GPP2024} and phase transition problems~\cite{MW2019,Miranville2019}. It is possible that the solutions of these models might have to satisfy certain constraints of analytical, physical or geometrical nature, thus rendering the problems under consideration obstacle problems.

In this paper, we are interested in the numerical approximation of obstacle problems in linearised elasticity modelled by means of fourth order differential operators. In particular, we are interested in the numerical approximation of the solutions of two-dimensional models describing the deflection of thin films subjected to remaining confined in a convex subset of the three-dimensional Euclidean space. In general, the models we will be considering are recovered upon completion of a rigorous asymptotic analysis departing from the standard energy of three-dimensional linearised elasticity.

\subsection*{Literature review}

The mathematical justification of obstacle problems in linearised elasticity via a rigorous asymptotic analysis was initiated in~\cite{Leger2008,Leger2010} in the case where the elastic structure under consideration is a linearly elastic shallow shell. However, the constraint considered in the latter two papers was only affecting the \emph{transverse} component of the displacement field. The mathematical justification of obstacle problems for linearly elastic shells and their numerical analysis was then continued in the papers~\cite{CiaPie2018b,MeiPie2024,Pie-2022-interior,Pie2023,Rodri2018,WSSY24,WSYZ24}.

For a different approach and different contact conditions from the one considered in this paper, we refer the reader to~\cite{CR19,CCRR21,Rodri2018,RC18,RC18-2}. In this direction, we also mention the recent papers on mathematical biology, which are modelled by means of contact mechanics~\cite{BDTT16,HT17}.

The discretisation of an \emph{unconstrained} fourth order biharmonic variational problem via a Finite Element Method of mixed type was proposed by Ciarlet \& Raviart~\cite{CiaRav74} and Ciarlet \& Glowinski~\cite{CiaGlo75}. The mixed Finite Element Method proposed in~\cite{CiaRav74} discretises the solution of the biharmonic equation by introducing a finite-dimensional space which is in general \emph{not} a subspace of the vector space where the mixed formulation is posed (cf., e.g., page~382 of~\cite{PGCFEM}) and, moreover, the \emph{primal variable} and the \emph{dual variable} appearing in the mixed formulation are, in general, not independent of each other (cf., e.g., the remark after Theorem~7.1.4 of~\cite{PGCFEM}). Besides, the actual computation of the pair solving the discrete mixed problem is carried out by means of the Uzawa Method (cf., e.g., the Introduction of Chapter~7 in~\cite{PGCFEM} for a detailed description).

It seems that the method proposed in~\cite{CiaRav74} is not applicable to more complex fourth order models where mixed second derivatives appear. Moreover, it is not clear if the implementation by Uzawa's Method discussed in~\cite{CiaGlo75} is applicable to  more complex fourth order models for which the energy minimisers are required to comply with a certain constraint.

The existing literature about the approximation of the solution of obstacle problems governed by fourth order operators is mainly based on Morley Finite Elements, Hsieh-Clough-Tocher Finite Elements, and Enriching operators~\cite{Brenner2013,CarKol2017,PS}. The applicability of the methods based on Enriching Operators seem, in particular, to be limited to scalar fourth order obstacle problems~\cite{Brenner2013} and vector-valued fourth order obstacle problems in elasticity where the constraint hinges on the sole transverse component of the solution~\cite{PS}. For the sake of completeness, we also mention the recent paper~\cite{LYC25}, where the authors studied a multi-scale Finite Element Method for a second order problem in elasticity, and the papers~\cite{DEP09,DEP10,DEP11,DEP11-2}, where the authors studied second order time-dependent Signorini problems with contact and collisions.

\subsection*{Novelties and main results}

The objective of this paper amounts to proposing new mixed Finite Element Methods for approximating the solution of more general constrained variational problems governed by fourth order differential operators. The \textbf{main novelties} introduced in this paper are the following:
\begin{itemize}
	\item[$(1)$] The first novelty presented in this paper amounts to proposing a new \emph{conforming} mixed Finite Element Method for approximating the solution of the biharmonic model via Courant triangles. This methodology both applies to the standard unconstrained problem considered in Chapter~7 in~\cite{PGCFEM} and to the corresponding constrained fourth order model. Besides, differently from~\cite{Brenner2013}, the techniques here presented do not hinge on the exploitation of Morley's triangles. After relaxing the original problem by introducing penalty terms associated with the violation of the constraint and with the relation between the primal variable and the dual variable, we derive error estimates for the approximate solutions and we investigate the robustness of the algorithm with respect to the parameters there considered. 
	
	\item[$(2)$] The second novelty we present in this paper concerns the discretisation of an obstacle problem for linearly elastic shallow shells, where the constraint amounts to restraining the set of admissible solutions to those displacement fields for which the deformed reference configuration of the shell remains confined in a prescribed convex subset of the three-dimensional Euclidean space. After relaxing the original problem by introducing penalty terms associated with the violation of the constraint and with the relation between the primal variable and the dual variable, we derive error estimates for the approximate solutions and we investigate the robustness of the algorithm with respect to the parameters there considered.
	
	In particular, we are able to show that if the middle surface of the linearly elastic shallow shell under consideration is flat, in the sense that the curvature is identically zero in the definition domain of the surface parametrisation, then it is possible to propose a mixed formulation based on Courant triangles akin to the one we considered for the biharmonic obstacle problem.
	
	\item[$(3)$] The third novelty is a counterexample showing that if the middle surface of the linearly elastic shallow shell under consideration is such that its curvature is not identically zero, then it is \emph{not} possible to carry out a discretisation via Courant triangles (or any other Finite Elements of Lagrange type). We observe in this fashion that this \emph{pathology} appears to be connected with the apparent lack of a \emph{rigidity theorem} for linearly elastic shallow shells.
\end{itemize}

Apart from the \emph{per se} theoretical interest, the nature of the \emph{counterexample} constituting the third novelty presented in this paper should hint at the importance $\mathcal{C}^1$ Finite Elements, which are available, to-date, only for a limited number of libraries\footnote{For a complete overview of the features of the most popular Finite Element Analysis libraries, see the website \href{https://en.wikipedia.org/wiki/List_of_finite_element_software_packages}{\texttt{https://en.wikipedia.org/wiki/List_of_finite_element_software_packages}}}.

\subsection*{Structure of the paper}

The paper is divided into four sections, including this one, in the remainder of which the basic notation will be recalled.
In section~\ref{sec1}, we present a new Finite Element Method for the prototypical biharmonic obstacle problem. The new method we are going to discuss also works as an alternative approach to the discretisation of the biharmonic equation proposed first by Ciarlet \& Raviart~\cite{CiaRav74}, and then complemented by Ciarlet \& Glowinski~\cite{CiaGlo75}.

In section~\ref{sec2}, we consider a fourth-order obstacle problem for linearly elastic shallow shells whose middle surface is, in general, not flat and that are  constrained to remain confined in a prescribed half-space. We duly exhibit the consequences of the aforementioned \emph{lack of rigidity} in the choice for the Finite Element Method for approximating the solution to this model, and we observe that rigidity is regained when the middle surface is flat.

In section~\ref{sec1-bis}, we present numerical results corroborating the theoretical results obtained in section~\ref{sec1}. Numerical experiments for the shallow shell model with flat surface presented in the form of supplementary material (Appendix~\ref{sec4}). The execution time of the second experiment for the model discussed in section~\ref{sec1} becomes extensively long when the mesh becomes finer. Therefore, performing the second experiment for the model presented in section~\ref{sec2} in the case of a general applied body force becomes challenging, as it is not in general straightforward to exhibit a manufactured solution and, moreover, the execution time increases.

For an overview about the classical notions of differential geometry used in this paper see, e.g., \cite{Ciarlet2000} or~\cite{Ciarlet2005} while, for an overview about the classical notions of functional analysis used in this paper see~\cite{Ciarlet2025}.

Latin indices, except $h$, take their values in the set $\{1,2,3\}$ while Greek indices, except $\nu$, $\kappa$ and $\varepsilon$, take their values in the set $\{1,2\}$. The Einstein summation convention is systematically used in conjunction with these two rules unless differently specified.
The notation $\mathbb{E}^3$ indicates the three-dimensional Euclidean space.
We denote by $\delta_{\alpha\beta}$ the standard Kronecker symbol.
Given an open subset $\Omega$ of $\mathbb{R}^n$, where $n \ge 1$, we denote the usual Lebesgue and Sobolev spaces by $L^2(\Omega)$, $H^1(\Omega)$, $H^1_0(\Omega)$, $H^2(\Omega)$, or $H^2_0(\Omega)$. Given a set $X$, the Cartesian product $X \times X$ is denoted by $\vec{X}$, while the Cartesian product $X\times X \times X$ is denoted by $\bm{X}$. In what follows, the abbreviations ``a.a.'' and ``a.e.'' stand for \emph{almost all} and \emph{almost everywhere}, respectively. Spaces of symmetric tensors are denoted in blackboard bold letters like for instance $\mathbb{L}^2(\Omega)$.

In order to favour the readability of this article, we resort to the notation $c_P$ to denote the positive constant of the Poincar\'e--Friedrichs inequality (cf., e.g., Theorem~8.3-3 in~\cite{Ciarlet2025}).

In all what follows, we denote by $\omega \subset \mathbb{R}^2$ a Lipschitz domain, namely a non-empty bounded open connected subset of $\mathbb{R}^2$ with Lipschitz continuous boundary $\gamma:=\partial \omega$ and such that $\omega$ is all \emph{on the same side of $\gamma$} (cf., e.g., Section~8.2 in~\cite{Ciarlet2025}). Let $y=(y_\alpha)$ denote a generic point in $\overline{\omega}$ and let $\partial_{\alpha}:=\partial/\partial y_\alpha$ and $\partial_{\alpha\beta}:=\partial^2/\partial y_\alpha \partial y_\beta$.
In what follows $\vec{\nu}$ denotes the outer unit normal vector field to the boundary $\gamma$ and $\partial_{\vec{\nu}}$ denotes the outer unit normal derivative operator along $\gamma$.

\section{A prototypical example of fourth order obstacle problem. The biharmonic model}
\label{sec1}

The first problem we study is a prototypical variational problem modelling the deflection of a linearly elastic plate. Let the function $f\in L^2(\omega)$ be given.

The free energy functional associated with the model under consideration is denoted by $J:H^2_0(\omega) \to \mathbb{R}$, and is defined by:
\begin{equation*}
J(v):=\dfrac{1}{2}\int_{\omega} |\Delta v|^2 \dd y -\int_{\omega} f v \dd y.
\end{equation*}

Let us recall that the bilinear form $a:H^2_0(\omega) \times H^2_0(\omega) \to \mathbb{R}$ defined by
\begin{equation*}
	a(u,v):=\int_{\omega}(\Delta u)(\Delta v) \dd y, \quad\textup{ for all }u,v \in H^2_0(\omega),
\end{equation*}
is continuous, symmetric, and $H^2_0(\omega)$-elliptic thanks to the Poincar\'e--Friedrichs inequality for fourth order operators (cf., e.g., Theorem~8.7-4 in~\cite{Ciarlet2025}). Also recall that the linear form $\ell:H^2_0(\omega) \to \mathbb{R}$ defined by
\begin{equation*}
	\ell(v):=\int_{\omega} f v \dd y,\quad\textup{ for all }v \in H^2_0(\omega),
\end{equation*}
is continuous.

Let $\theta:\overline{\omega} \to \mathbb{R}$ be a given function such that $\theta \in H^1(\omega) \cap \mathcal{C}^0(\overline{\omega})$ and such that $\theta(y)<0$, for all $y\in\overline{\omega}$. The given function $\theta$ represents the \emph{obstacle} that the deformed plate does not have to cross. In order to determine the equilibrium configuration of a plate subjected not to cross the given function $\theta$, we must restrict our search for solutions to functions $v \in H^2_0(\omega)$ such that $v(y) \ge \theta(y)$, for all $y \in \overline{\omega}$, and we note that the constraint has to hold everywhere up the boundary in light of the Rellich--Kondra\v{s}ov Theorem (cf., e.g., Theorem~8.4-3 in~\cite{Ciarlet2025}).

The latter considerations imply that we need to restrict our search for solutions to the functions in the following non-empty, closed and convex subset of $H^2_0(\omega)$, denoted and defined by:
\begin{equation*}
	U:=\{v \in H^2_0(\omega); v \ge \theta \textup{ in }\overline{\omega}\}.
\end{equation*}

It is thus classical to establish (cf., e.g., Section~8.8 in~\cite{Ciarlet2025}) that the following minimisation problem
\begin{equation}
	\label{mp1}
	\min_{v \in U} J(v),
\end{equation}
admits a unique solution $u\in U$. Recall that, according to~\cite{Frehse1971}, the solution of the minimisation problem~\eqref{mp1} satisfies $u \in H^3(\omega) \cap H^2_0(\omega)$.
Let us recall that a regularity of class $H^4(\omega)$ is \emph{not} in general achievable in light of the counterexample contrived by Caffarelli and his collaborators for fourth order obstacle problems~\cite{Caffarelli1979,CafFriTor1982}.

The minimisation problem~\eqref{mp1} can be equivalently formulated in terms of the following variational inequalities.

\begin{customprob}{$\mathcal{P}$}
	\label{problem1}
	Find $u \in U$ satisfying:
	\begin{equation*}
		\int_{\omega} (\partial_{\alpha\beta} u) \partial_{\alpha\beta}(v-u) \dd y \ge \int_{\omega}f (v-u) \dd y,
	\end{equation*}
	for all $v \in U$.
	\bqed
\end{customprob}

The numerical approximation of the solution for Problem~\ref{problem1} was recently studied in~\cite{Brenner2013}, where the authors exploited the properties of Enriching Operators to derive the error estimates for the difference between the exact solution and a suitable Finite Element approximation. The strategy in~\cite{Brenner2013} made use of Morley's triangles to derive the error estimates. The apparent drawback in the implementation of this strategy lies in the fact that Morley's triangles are not implemented in a number of Finite Element libraries.

A generalisation of the strategy presented in the paper~\cite{Brenner2013} was recently proposed in~\cite{PS}, where we studied a numerical method - based on Enriching Operators - for approximating the solution of an obstacle problem for linearly elastic shallow shells. Note, however, that the proofs in~\cite{PS} only work if the constraint is imposed on the sole transverse component of the solution of the problem under consideration, which represents the displacement of the linearly elastic shallow shell involved in the analysis. Apparently, the method of Enriching Operators does not work for treating the numerical convergence in the case of more general constraints, like for instance, when the deformed reference configuration is required to remain confined in the convex part of the Euclidean space identified by the intersection of two planar obstacles.

The alternative strategy we propose in this section to numerically approximate the solution of Problem~\ref{problem1} will serve to devise a new numerical method for approximating the displacement of linearly elastic shallow shells subjected to remaining confined in the convex portion of the Euclidean space identified by a finite number of planes in a way that the constraint will affect \emph{at once} all the components of the displacement field, thus generalising the result established in~\cite{PS}. The novelty in the strategy we are proposing consists in approximating the solution of Problem~\ref{problem1} by solely resorting to Courant triangles (cf., e.g., Section~2.2 in~\cite{PGCFEM}), i.e., without resorting to Morley's triangles.

We now consider a mixed formulation for Problem~\ref{problem1}. In order to have a sound formulation, the following assumption on the \emph{given} forcing term $f$ must be made: \emph{There exists a vector field $\vec{F} \in H(\textup{div};\omega)$ such that $f=-\textup{div } \vec{F}$ a.e. in $\omega$}. Note that this assumption is reasonable in light of Bogovskii's Theorem (cf., e.g., Theorem~8.12-1 in~\cite{Ciarlet2025} and the commentary part at the end of Section~7 of~\cite{PGCFEM}).
For a treatment of the spaces $H(\textup{div};\omega)$ see Section~8.15 in~\cite{Ciarlet2025}, for instance.
The corresponding mixed formulation associated with Problem~\ref{problem1}, which we shall be referring to as Problem~$\mathcal{Q}$, can thus be stated.

\begin{customprob}{$\mathcal{Q}$}
	\label{problem2-1}
	Find $(u,\vec{\xi}) \in \mathbb{U}:=\{(w,\vec{\varphi})\in H^1_0(\omega) \times \vec{H}^1_0(\omega); w \ge \theta \textup{ a.e. in }\omega \textup{ and } \vec{\varphi}=\nabla w \textup{ a.e. in }\omega\}$ satisfying:
	\begin{equation*}
	\int_{\omega} (\partial_{\alpha}\xi_{\beta}) \partial_{\alpha}(\eta_\beta -\xi_\beta) \dd y\ge \int_{\omega} \vec{F} \cdot (\vec{\eta}-\vec{\xi}) \dd y,
	\end{equation*}
	for all $(v,\vec{\eta}) \in \mathbb{U}$.
	\bqed
\end{customprob}

Clearly, Problem~\ref{problem2-1} admits a unique solution as a re-writing of Problem~\ref{problem1}. Note that $u \in H^2_0(\omega)$ since $u \in H^1_0(\omega)$ and $\vec{\xi}=\nabla u \in \vec{H}^1_0(\omega)$.
The next step consists in penalising Problem~\ref{problem2-1}. The weak formulation for one such variational problem  takes the following form.

\begin{customprob}{$\mathcal{Q}_\kappa$}
	\label{problem3}
	Find $(u_\kappa,\vec{\xi}_\kappa) \in H^1_0(\omega) \times \vec{H}^1_0(\omega)$ satisfying:
	\begin{equation*}
		\int_{\omega} (\partial_{\alpha}\xi_{\beta,\kappa}) (\partial_{\alpha} \eta_\beta) \dd y-\dfrac{1}{\kappa}\int_{\omega} \{u_\kappa -\theta\}^{-} v \dd y+\dfrac{1}{\kappa} \int_{\omega} (\nabla u_\kappa -\vec{\xi}_\kappa) \cdot (\nabla v - \vec{\eta}) \dd y= \int_{\omega} \vec{F} \cdot \vec{\eta} \dd y,
	\end{equation*}
	for all $(v,\vec{\eta}) \in H^1_0(\omega) \times \vec{H}^1_0(\omega)$.
	\bqed
\end{customprob}

The following convergence result can be established.
\begin{theorem}
	\label{th:2}
	For each $0<\kappa\le 2(1+c_P^2)$, Problem~\ref{problem3} admits a unique solution $(u_\kappa,\vec{\xi}_\kappa) \in H^1_0(\omega) \times \vec{H}^1_0(\omega)$. Besides, we have that the following convergences hold up to passing to subsequences
	\begin{align*}
		u_\kappa \rightharpoonup u&, \textup{ in }H^1_0(\omega) \textup{ as }\kappa \to 0,\\
		\vec{\xi}_\kappa \rightharpoonup \nabla u&, \textup{ in }\vec{H}^1_0(\omega) \textup{ as }\kappa \to 0,
	\end{align*}
	where $u\in U$ is the unique solution of Problem~\ref{problem1}.
\end{theorem}
\begin{proof}
To begin with, we observe that the left-hand side of the variational equations in Problem~\ref{problem3} is made of the sum of bilinear forms which are symmetric and continuous, and of a monotone operator. The right-hand side, instead, is a linear and continuous form.

In order to assert the existence and uniqueness of solutions, we need to establish the coerciveness of the left-hand side so as to apply the Minty--Browder Theorem (cf., e.g., Theorem~12.5-3 in ~\cite{Ciarlet2025}). An application of the Poincar\'{e}--Friedrichs inequality (cf., e.g., Theorem~8.3-3 in~\cite{Ciarlet2025}) and Young's inequality~\cite{Young1912} allow to obtain the following estimate
	\begin{equation*}
		\begin{aligned}
			&\dfrac{\|\vec{\xi}_\kappa\|_{\vec{H}^1_0(\omega)}^2}{2(1+c_P^2)}+\dfrac{\|u_\kappa\|_{H^1_0(\omega)}^2}{4(1+c_P^2)}^2\le\dfrac{\|\vec{\xi}_\kappa\|_{\vec{H}^1_0(\omega)}^2}{2(1+c_P^2)}+\dfrac{\|\nabla u_\kappa\|_{\vec{L}^2(\omega)}^2}{4(1+c_P^2)}\\
			&\le\dfrac{\|\vec{\xi}_\kappa\|_{\vec{H}^1_0(\omega)}^2}{2(1+c_P^2)}+\dfrac{\|\vec{\xi}_\kappa\|_{\vec{L}^2(\omega)}^2+\|\vec{\xi}_\kappa-\nabla u_\kappa\|_{\vec{L}^2(\omega)}^2}{2(1+c_P^2)}\\
			&\le\dfrac{\|\vec{\xi}_\kappa\|_{\vec{H}^1_0(\omega)}^2}{1+c_P^2}+\dfrac{\|\{u_\kappa-\theta\}^{-}\|_{L^2(\omega)}^2}{\kappa}+\dfrac{\|\vec{\xi}_\kappa-\nabla u_\kappa\|_{\vec{L}^2(\omega)}^2}{\kappa},
		\end{aligned}
	\end{equation*}
from which coerciveness follows. The strong monotonicity follows by applying the same coercivity estimate to the difference of two admissible pairs.

Specialising $(v,\vec{\eta})=(u_\kappa,\vec{\xi}_\kappa)$ in the variational equations of Problem~\ref{problem3}, we obtain that an application of the classical Poincar\'e--Friedrichs inequality (cf., e.g., Theorem~8.3-3 of~\cite{Ciarlet2025}) gives:
\begin{equation*}
	\dfrac{\|\vec{\xi}_\kappa\|_{\vec{H}^1_0(\omega)}^2}{1+c_P^2}+\dfrac{\|\{u_\kappa-\theta\}^{-}\|_{L^2(\omega)}^2}{\kappa}+\dfrac{\|\vec{\xi}_\kappa-\nabla u_\kappa\|_{\vec{L}^2(\omega)}^2}{\kappa} \le \|\vec{F}\|_{\vec{L}^2(\omega)} \|\vec{\xi}_\kappa\|_{\vec{H}^1_0(\omega)}.
\end{equation*}

The latter estimate in turn implies that $\|\vec{\xi}_\kappa\|_{\vec{H}^1_0(\omega)} \le (1+c_P^2)\|\vec{F}\|_{\vec{L}^2(\omega)}$, so that it results that the sequence $\{\vec{\xi}_\kappa\}_{\kappa>0}$ is bounded in $\vec{H}^1_0(\omega)$ independently of $\kappa$.
The computations above give:
\begin{equation}
	\label{eq4}
	\begin{aligned}
		\|\{u_\kappa-\theta\}^{-}\|_{L^2(\omega)} &\le\sqrt{1+c_P^2}\|\vec{F}\|_{\vec{L}^2(\omega)}\sqrt{\kappa},\\
		\|\vec{\xi}_\kappa-\nabla u_\kappa\|_{\vec{L}^2(\omega)}&\le\sqrt{1+c_P^2}\|\vec{F}\|_{\vec{L}^2(\omega)} \sqrt{\kappa}.
	\end{aligned}
\end{equation}

Therefore, we obtain that an application of the triangle inequality and of the second estimate in~\eqref{eq4} give:
\begin{equation*}
	\|\nabla u_\kappa\|_{\vec{L}^2(\omega)} \le \|\vec{\xi}_\kappa-\nabla u_\kappa\|_{\vec{L}^2(\omega)}+\|\vec{\xi}_\kappa\|_{\vec{L}^2(\omega)}\le \sqrt{1+c_P^2}\|\vec{F}\|_{\vec{L}^2(\omega)}\left(\sqrt{\kappa}+\sqrt{1+c_P^2}\right),
\end{equation*}
which in turn implies that the sequence $\{\nabla u_\kappa\}_{\kappa>0}$ is bounded in $\vec{L}^2(\omega)$ independently of $\kappa$. Another application of the classical Poincar\'e--Friedrichs inequality (cf., e.g., Theorem~8.3-3 in~\cite{Ciarlet2025}) gives that:
\begin{equation*}
	\|u_\kappa\|_{H^1_0(\omega)}\le\sqrt{1+c_P^2}\|\nabla u_\kappa\|_{\vec{L}^2(\omega)} \le (1+c_P^2)\|\vec{F}\|_{\vec{L}^2(\omega)}\left(\sqrt{\kappa}+\sqrt{1+c_P^2}\right),
\end{equation*}
which in turn implies that the sequence $\{u_\kappa\}_{\kappa>0}$ is bounded in $H^1_0(\omega)$ independently of $\kappa$. Therefore, up to passing to subsequences, we obtain that:
\begin{equation}
	\label{conv2}
	\begin{aligned}
		u_\kappa \rightharpoonup u&, \textup{ in }H^1_0(\omega) \textup{ as }\kappa\to0,\\
		\vec{\xi}_\kappa \rightharpoonup \vec{\xi}&, \textup{ in }\vec{H}^1_0(\omega) \textup{ as }\kappa\to0,\\
		\{u_\kappa-\theta\}^{-}\to 0&, \textup{ as }\kappa\to0,\\
		\|\vec{\xi}_\kappa-\nabla u_\kappa\|_{\vec{L}^2(\omega)}\to 0&, \textup{ as }\kappa\to0.
	\end{aligned}
\end{equation}

In particular, the fact that $u \in H^1_0(\omega)$ and $\nabla u \in \vec{H}^1_0(\omega)$ implies that $u\in H^2_0(\omega)$, being $\omega$ a Lipschitz domain. Hence, we have that $(u,\vec{\xi}) \in \mathbb{U}$. Let us now fix an arbitrary $v \in U$, and let us test the variational equations in Problem~\ref{problem3} on $(v-u_\kappa,\nabla v - \vec{\xi}_\kappa)$. We obtain that:

\begin{equation*}
	\begin{aligned}
	&\int_{\omega} (\partial_{\alpha}\xi_{\beta,\kappa}) (\partial_{\alpha\beta}v-\partial_{\alpha} \xi_{\beta,\kappa}) \dd y \underbrace{-\dfrac{1}{\kappa}\int_{\omega} \{u_\kappa -\theta\}^{-} (v-u_\kappa) \dd y}_{\le 0}\\
	&\qquad +\underbrace{\dfrac{1}{\kappa} \int_{\omega} (\nabla u_\kappa -\vec{\xi}_\kappa) \cdot \left[(\nabla v-\nabla u_\kappa) - (\nabla v-\vec{\xi}_\kappa)\right] \dd y}_{\le 0}
	= \int_{\omega} \vec{F} \cdot (\nabla v-\vec{\xi}_\kappa) \dd y.
	\end{aligned}
\end{equation*}

Upon passing to the $\limsup$ as $\kappa\to0$, the convergence process~\eqref{conv2} yields:

\begin{equation*}
	\begin{aligned}
	& \int_{\omega} (\partial_{\alpha}\xi_{\beta}) (\partial_{\alpha\beta}v-\partial_{\alpha} \xi_{\beta}) \dd y \ge \int_{\omega} \vec{F} \cdot (\nabla v-\vec{\xi}) \dd y,
	\end{aligned}
\end{equation*}

Since $(u,\vec{\xi}) \in \mathbb{U}$, we obtain that:

\begin{equation*}
\int_{\omega} (\partial_{\alpha\beta} u) \partial_{\alpha\beta}(v-u) \dd y \ge \int_{\omega} f(v-u) \dd y.
\end{equation*}

Given the arbitrariness of $v \in U$, we infer that $u$ is the solution of Problem~\ref{problem1}.
\end{proof}

In regard to the proof of Theorem~\ref{th:2}, we observe that the assumption on the applied body force $f$ plays a crucial role in establishing the convergences in~\eqref{conv2}. Indeed, without this \emph{regularity} assumption on $f$, the right-hand side of the non-linear variational equations in Problem~\ref{problem3} should be expressed in terms of the \emph{primal variable} and, therefore, the boundedness of the sequence $\{u_\kappa\}_{\kappa>0}$ in $H^1_0(\omega)$ independently of $\kappa$ could not be inferred by the energy estimates.

The next result establishes the higher regularity for the unique solution of Problem~\ref{problem3}. Since the argument presented here closely follows~\cite{AgmDouNir1959,AgmDouNir1964}, we limit ourselves to discussing in details the steps of the proof associated with the fact that the operator considered in the variational formulation of Problem~\ref{problem3} is \emph{not in the divergence form} (see Section~6.3.1 in~\cite{Evans2010}).

\begin{theorem}
	\label{th:biharmonic-reg}
	For each $0<\kappa\le 2(1+c_P^2)$, let $(u_\kappa,\vec{\xi}_\kappa)$ denote the solution of Problem~\ref{problem3}.
	Then $(u_\kappa,\vec{\xi}_\kappa)$ is of class $(H^1_0(\omega)\cap H^2(\omega))\times (\vec{H}^1_0(\omega)\cap\vec{H}^2(\omega))$.
\end{theorem}
\begin{proof}
	Testing the variational equations of Problem~\ref{problem3} at $(v,\vec{0})$, where $v$ varies arbitrarily in $H^1_0(\omega)$ gives:
	\begin{equation*}
		\int_{\omega}\nabla u_\kappa\cdot\nabla v\dd y=\int_{\omega}\left(\{u_\kappa-\theta\}^{-}-\textup{div }\vec{\xi}_\kappa\right)v\dd y.
	\end{equation*}
	
	Since Theorem~\ref{th:2} ensures that $\{\vec{\xi}_\kappa\}_{\kappa>0}$ is bounded in $\vec{H}^1_0(\omega)$ independently of $\kappa$, and that $\{u_\kappa\}_{\kappa>0}$ is bounded in $H^1_0(\omega)$ independently of $\kappa$, we infer (cf., e.g., Section~6.3 in~\cite{Evans2010}) that $u_\kappa\in H^2(\omega)$ and $\|u_\kappa\|_{H^2(\omega)}$ is bounded independently of $\kappa$.
	Testing the variational equations of Problem~\ref{problem3} at $(0,\vec{\eta})$, where $\vec{\eta}$ varies arbitrarily in $\vec{H}^1_0(\omega)$ gives:
	\begin{equation*}
		\int_{\omega}(\partial_{\alpha}\xi_{\beta,\kappa})(\partial_{\alpha}\eta_\beta)\dd y
		=\dfrac{1}{\kappa}\int_{\omega}(\nabla u_\kappa-\vec{\xi}_\kappa)\cdot\vec{\eta}\dd y+\int_{\omega}\vec{F}\cdot\vec{\eta}\dd y.
	\end{equation*}
	
	An application of the regularity theory in Section~6.3 of~\cite{Evans2010} together with the fact that $\kappa^{-1/2}\|\vec{\xi}_\kappa-\nabla u_\kappa\|_{\vec{L}^2(\omega)}$ is bounded independently of $\kappa$ (Theorem~\ref{th:2}) gives:
	\begin{equation*}
		\|\vec{\xi}_\kappa\|_{\vec{H}^2(\omega)}\le\dfrac{C}{\sqrt{\kappa}}\left(\|\vec{F}\|_{\vec{L}^2(\omega)}+\|\vec{\xi}_\kappa\|_{\vec{L}^2(\omega)}+\dfrac{\|\nabla u_\kappa-\vec{\xi}_\kappa\|_{\vec{L}^2(\omega)}}{\sqrt{\kappa}}\right)=\mathcal{O}(\kappa^{-1/2}).
	\end{equation*}
\end{proof}

Exploiting Theorem~\ref{th:biharmonic-reg}, we can show that the weak convergences established in Theorem~\ref{th:2} are actually strong.

\begin{theorem}
	\label{th:biharmonic-strong}
	Let $(u_\kappa,\vec{\xi}_\kappa)$ denote the solution of Problem~\ref{problem3}.
	Then, the following convergences hold up to passing to subsequences
	\begin{align*}
		u_\kappa \to u&, \textup{ in }H^1_0(\omega) \textup{ as }\kappa \to 0,\\
		\vec{\xi}_\kappa \to \nabla u&, \textup{ in }\vec{H}^1_0(\omega) \textup{ as }\kappa \to 0,
	\end{align*}
	where $u\in U$ is the unique solution of Problem~\ref{problem1}.
\end{theorem}
\begin{proof}
	Let us first observe that the fact that $\omega$ is a Lipschitz domain and the fact that $u\in H^2_0(\omega)$ imply that $u\in H^1_0(\omega)$ and $\nabla u \in \vec{H}^1_0(\omega)$.
	Testing the variational equations of Problem~\ref{problem3} at $(u_\kappa-u,\vec{\xi}_\kappa-\nabla u)$, an application of the Poincar\'{e}--Friedrichs inequality (cf., e.g., Theorem~8.3-3 in~\cite{Ciarlet2025}), the fact that $-\{\cdot\}^{-}$ is monotone, and a series of straightforward computation gives
	\begin{equation*}
		\dfrac{\|\xi_{\beta,\kappa}-\partial_\beta u\|_{H^1_0(\omega)}^2}{1+c_P^2} \le -\int_{\omega}\partial_{\alpha\beta}u (\partial_{\alpha}\xi_{\beta,\kappa}-\partial_{\alpha\beta}u) \dd y+\int_{\omega}\vec{F}\cdot(\vec{\xi}_\kappa-\nabla u)\dd y \to 0,\quad\textup{ as }\kappa\to0,
	\end{equation*}
	where the last convergence is due to Theorem~\ref{th:2}. This establishes that $\vec{\xi}_\kappa\to\nabla u$ in $\vec{H}^1_0(\omega)$ as $\kappa\to0$.
	
	Since, thanks to Theorem~\ref{th:biharmonic-reg}, the sequence $\{u_\kappa\}_{\kappa>0}$ is bounded in $H^2(\omega)$ independently of $\kappa$, we obtain that $\{\nabla u_\kappa\}_{\kappa>0}$ is bounded in $\vec{H}^1(\omega)$ independently of $\kappa$. Combining this result with the Rellich--Kondra\v{s}ov Theorem (cf., e.g., Theorem~8.4-3 in~\cite{Ciarlet2025}), Theorem~\ref{th:2}, the boundary conditions, and the compactness of the trace operator (cf., e.g., Theorem~8.4-8(b) of~\cite{Ciarlet2025}), we obtain that $u_\kappa\to u$ in $H^1_0(\omega)$ as $\kappa\to0$.
\end{proof}

The remainder of this section is devoted to the discretisation of the solution of Problem~\ref{problem3} via the Finite Element Method. In what follows, for each $h>0$, we denote by $\mathcal{T}_h$ an \emph{affine regular} triangulation of the Lipschitz domain $\omega$ (cf., e.g., \cite{PGCFEM}) made of Courant triangles. We denote by $V_h$ the finite-dimensional subspace of $H^1_0(\omega)$ associated with one such triangulation $\mathcal{T}_h$. The discretisation of Problem~\ref{problem3} can thus be formulated as follows.

\begin{customprob}{$\mathcal{Q}_\kappa^h$}
	\label{problem4}
	Find $(u_\kappa^h,\vec{\xi}_\kappa^h) \in V_h \times \vec{V}_h$ satisfying:
	\begin{equation*}
	\int_{\omega} (\partial_{\alpha}\xi_{\beta,\kappa}^h) (\partial_{\alpha} \eta_\beta) \dd y-\dfrac{1}{\kappa}\int_{\omega} \{u_\kappa^h -\theta\}^{-} v \dd y+\dfrac{1}{\kappa} \int_{\omega} (\nabla u_\kappa^h -\vec{\xi}_\kappa^h) \cdot (\nabla v - \vec{\eta}) \dd y= \int_{\omega} \vec{F} \cdot \vec{\eta} \dd y,
	\end{equation*}
	for all $(v,\vec{\eta}) \in V_h \times \vec{V}_h$.
	\bqed
\end{customprob}

It is straightforward to observe that, akin to its continuum counterpart, Problem~\ref{problem4} admits a unique solution $(u_\kappa^h,\vec{\xi}_\kappa^h) \in V_h \times \vec{V}_h$. In the next theorem we are going to show that the solution of Problem~\ref{problem4} converges, in some suitable sense, to the solution of Problem~\ref{problem3} as $h \to 0$.

\begin{theorem}
	\label{th:3}
	Let $0<\kappa\le 2(1+c_P^2)$. Let $(u_\kappa,\vec{\xi}_\kappa) \in H^1_0(\omega) \times \vec{H}^1_0(\omega)$ be the unique solution of Problem~\ref{problem3}, and let $(u_\kappa^h,\vec{\xi}_\kappa^h) \in V_h \times \vec{V}_h$ be the unique solution of Problem~\ref{problem4}. Then, there exists a constant $\tilde{C}>0$ independent of $h$ and $\kappa$ such that:
	\begin{equation*}
	\|u_\kappa^h-u_\kappa\|_{H^1_0(\omega)} + \|\vec{\xi}_\kappa^h-\vec{\xi}_\kappa\|_{\vec{H}^1_0(\omega)} \le \dfrac{\tilde{C}h}{\kappa}.
	\end{equation*}
\end{theorem}
\begin{proof}
	In what follows, we denote by $\Pi_h$ and $\vec{\Pi}_h$ the standard interpolation operators in $V_h$ and $\vec{V}_h$, respectively (cf., e.g., \cite{PGCFEM}).
	Since $u_\kappa \in H^2(\omega) \cap H^1_0(\omega)$ and since $\vec{\xi}_\kappa \in \vec{H}^2(\omega)\times \vec{H}^1_0(\omega)$, we are in a position to apply Theorem~3.2.1 in~\cite{PGCFEM} so as to infer that:
	\begin{equation}
	\label{eq5}
	\begin{aligned}
	\|u_\kappa-\Pi_h u_\kappa\|_{H^1_0(\omega)} &\le C h |u_\kappa|_{H^2(\omega)},\\
	\|\vec{\xi}_\kappa-\vec{\Pi}_h \vec{\xi}_\kappa\|_{\vec{H}^1_0(\omega)} &\le C h |\vec{\xi}_\kappa|_{\vec{H}^2(\omega)},
	\end{aligned}
	\end{equation}
	for some $C>0$ independent of $h$ and $\kappa$.
	
	Thanks to the fact that the forcing terms in the variational equations of Problem~\ref{problem3} and Problem~\ref{problem4} coincide, we have that an application of the classical Poincar\'e--Friedrichs inequality (cf., e.g., Theorem~8.3-3 in~\cite{Ciarlet2025}), Theorem~8.7-4 of~\cite{Ciarlet2025}, and Young's inequality (cf., e.g., \cite{Young1912}) gives:
	\begin{equation*}
	\begin{aligned}
	&\dfrac{\|\vec{\xi}_\kappa-\vec{\xi}_\kappa^h\|_{\vec{H}^1_0(\omega)}^2}{1+c_P^2}+\dfrac{1}{\kappa}\|(-\{u_\kappa-\theta\}^{-})-(-\{u_\kappa^h-\theta\}^{-})\|_{L^2(\omega)}^2+\dfrac{1}{\kappa}\|\nabla(u_\kappa-u_\kappa^h) -(\vec{\xi}_\kappa-\vec{\xi}_\kappa^h)\|_{\vec{L}^2(\omega)}^2\\
	&\le \int_{\omega} \sum_{\alpha,\beta} |\partial_{\alpha}(\xi_{\beta,\kappa}-\xi_{\beta,\kappa}^h)|^2 \dd y+\dfrac{1}{\kappa} \int_{\omega}\left[(-\{u_\kappa-\theta\}^{-})-(-\{u_\kappa^h-\theta\}^{-})\right] (u_\kappa-u_\kappa^h) \dd y\\
	&\qquad+\dfrac{1}{\kappa}\int_{\omega} |\nabla(u_\kappa-u_\kappa^h) -(\vec{\xi}_\kappa-\vec{\xi}_\kappa^h)|^2 \dd y\\
	&=\int_{\omega} \partial_{\alpha}(\xi_{\beta,\kappa}-\xi_{\beta,\kappa}^h) \partial_{\alpha}(\xi_{\beta,\kappa}-\Pi_h \xi_{\beta,\kappa}) \dd y\\
	&\qquad+\dfrac{1}{\kappa} \int_{\omega}\left[(-\{u_\kappa-\theta\}^{-})-(-\{u_\kappa^h-\theta\}^{-})\right] (u_\kappa-\Pi_h u_\kappa) \dd y\\
	&\qquad+\dfrac{1}{\kappa}\int_{\omega} \left[\nabla(u_\kappa-u_\kappa^h) -(\vec{\xi}_\kappa-\vec{\xi}_\kappa^h)\right] \cdot \left[\nabla(u_\kappa-\Pi_h u_\kappa) -(\vec{\xi}_\kappa-\vec{\Pi}_h \vec{\xi}_\kappa)\right] \dd y\\
	&\le\dfrac{1}{2(1+c_P^2)}\|\vec{\xi}_\kappa-\vec{\xi}_\kappa^h\|_{\vec{H}^1_0(\omega)}^2
	+\dfrac{1+c_P^2}{2} \|\vec{\xi}_\kappa-\vec{\Pi}_h \vec{\xi}_\kappa\|_{\vec{H}^1_0(\omega)}^2\\
	&\qquad+\dfrac{1}{2\kappa}\|(-\{u_\kappa-\theta\}^{-})-(-\{u_\kappa^h-\theta\}^{-})\|_{L^2(\omega)}^2+\dfrac{1}{2\kappa}\|u_\kappa-\Pi_h u_\kappa\|_{H^1_0(\omega)}^2\\
	&\qquad+\dfrac{1}{2\kappa}\|\nabla(u_\kappa-u_\kappa^h) -(\vec{\xi}_\kappa-\vec{\xi}_\kappa^h)\|_{\vec{L}^2(\omega)}^2
	+\dfrac{1}{2\kappa}\|\nabla(u_\kappa-\Pi_h u_\kappa) -(\vec{\xi}_\kappa-\vec{\Pi}_h \vec{\xi}_\kappa^h)\|_{\vec{L}^2(\omega)}^2.
	\end{aligned}
	\end{equation*}
	
	By the standard augmentation-of-regularity argument and the fact that the constraint is inactive near the boundary, we obtain that the semi-norms on the right-hand side of~\eqref{eq5} are both of order $\mathcal{O}(\kappa^{-1/2})$.
	An application of~\eqref{eq5} and Theorem~\ref{th:biharmonic-reg} thus yields:
	\begin{equation*}
	\begin{aligned}
		&\dfrac{1}{8(1+c_P^2)^2}\left\{\|\vec{\xi}_\kappa-\vec{\xi}_\kappa^h\|_{\vec{H}^1_0(\omega)}^2+\|u_\kappa-u_\kappa^h\|_{H^1_0(\omega)}^2\right\}\\
		&\le\dfrac{1}{4(1+c_P^2)}\|\vec{\xi}_\kappa-\vec{\xi}_\kappa^h\|_{\vec{H}^1_0(\omega)}^2+\dfrac{1}{8(1+c_P^2)}\|\nabla u_\kappa-\nabla u_\kappa^h\|_{\vec{L}^2(\omega)}^2\\
		&\le\dfrac{1}{2(1+c_P^2)}\|\vec{\xi}_\kappa-\vec{\xi}_\kappa^h\|_{\vec{H}^1_0(\omega)}^2+\dfrac{1}{2\kappa}\|\nabla(u_\kappa-u_\kappa^h) -(\vec{\xi}_\kappa-\vec{\xi}_\kappa^h)\|_{\vec{L}^2(\omega)}^2
		\le\dfrac{1+c_P^2}{2} \|\vec{\xi}_\kappa-\vec{\Pi}_h \vec{\xi}_\kappa\|_{\vec{H}^1_0(\omega)}^2\\ &\qquad+\dfrac{1}{2\kappa}\|u_\kappa-\Pi_h u_\kappa\|_{H^1_0(\omega)}^2
		+\dfrac{1}{2\kappa}\|\nabla(u_\kappa-\Pi_h u_\kappa)-(\vec{\xi}_\kappa-\vec{\Pi}_h \vec{\xi}_\kappa)\|_{\vec{L}^2(\omega)}^2\\
		&\le \dfrac{1+c_P^2}{2} \|\vec{\xi}_\kappa-\vec{\Pi}_h \vec{\xi}_\kappa\|_{\vec{H}^1_0(\omega)}^2 +\dfrac{1}{2\kappa}\|u_\kappa-\Pi_h u_\kappa\|_{H^1_0(\omega)}^2+\dfrac{1}{\kappa}\|\nabla(u_\kappa-\Pi_h u_\kappa)\|_{\vec{L}^2(\omega)}^2 +\dfrac{1}{\kappa}\|(\vec{\xi}_\kappa-\vec{\Pi}_h \vec{\xi}_\kappa)\|_{\vec{L}^2(\omega)}^2\\
		&\le \dfrac{2C^2 (1+c_P^2)h^2}{\kappa}\left(|u_\kappa|_{H^2(\omega)}^2+|\vec{\xi}_\kappa|_{\vec{H}^2(\omega)}^2\right) \le \dfrac{2C^2\hat{C}^2(1+c_P^2)h^2}{\kappa^2},
	\end{aligned}
	\end{equation*}
	where $\hat{C}>0$ in the last inequality is the constant associated with the augmentation-of-regularity argument (cf., e.g., \cite{Nec67}) and depends on $\omega$ and the scaled forcing term only. The proof is thus complete by letting $\tilde{C}:=8C\hat{C}(1+c_P^2)^{3/2}$.
\end{proof}

As a remark, if we specialise $\kappa = h^q$ with $0<q<1$ it results, thanks to an application of Theorem~\ref{th:2} and Theorem~\ref{th:3}, that the solution of Problem~\ref{problem4} converges to the solution of Problem~\ref{problem1} as $h\to0$. This is in line with the results presented in~\cite{Brenner2013}.

We note in passing that the techniques presented in Theorem~7.1.6 of~\cite{PGCFEM} for discretising the obstacle-free biharmonic problem hinge on a discretisation performed via conforming Finite Elements for fourth-order problems, whereas our technique \emph{only} exploits Courant triangles. Not only the technique we developed requires less regularity for the solution but, differently from the one in~\cite{CG75}, it can be implemented in Finite Element Analysis packages that do not come with Finite Elements for fourth order problems and with Enriched Finite Elements~\cite{AD23}. The crucial step for achieving this result is the regularity assumption of the given forcing term.

We also observe that the method we presented could be applied for discretising the solution of the biharmonic equation, as an alternative to the method originally proposed by Ciarlet \& Raviart~\cite{CiaRav74}, and then complemented by Ciarlet \& Glowinski~\cite{CiaGlo75}. The differences between these methods and the method we are here proposing are, first, that the method discussed in~\cite{CiaRav74} exploits a finite-dimensional space that is not a subspace of the space where the mixed formulation where the problem is defined, whereas our method discretises the mixed formulation over a finite-dimensional space of the space where the mixed formulation is posed and, second, the solution of the mixed formulation proposed in~\cite{CiaRav74} is to be approximated via Uzawa's method since the primal and dual variables do not vary independently, whereas in the mixed method we proposed the primal and dual variables vary independently.

\section{Numerical approximation of an obstacle problem for linearly elastic shallow shells}
\label{sec2}

In this section, we denote by $\omega$ a \emph{simply connected} Lipschitz domain.
Following~\cite{CiaMia1992} (see also Section~3.1 of~\cite{Ciarlet1997}), we now recall the \emph{rigorous} definition of a linearly elastic shallow shell (from now on, \emph{shallow shell} for the sake of brevity).

Assume that, for each $\varepsilon>0$, it is given a function $\theta^\varepsilon \in \mathcal{C}^3(\overline{\omega})$ that defines the middle surface of the corresponding shallow shell with thickness equal to $2\varepsilon$. 
According to the \emph{shallowness criterion} (cf., e.g., \cite{Ciarlet1997}), a linearly elastic shell is \emph{shallo}w if and only if there exists a function $\theta \in \mathcal{C}^3(\overline{\omega})$, independent of $\varepsilon$, such that:
\begin{equation}
	\label{critShall}
	\theta^\varepsilon(y)=\varepsilon \theta(y),\quad \textup{ for all }y \in \overline{\omega}.
\end{equation}

The relation in~\eqref{critShall} means that, up to an additive constant, the \emph{deviation} of the middle surface of the reference configuration of the shallow shell from a plane, which is measured via the function $\theta^\varepsilon:\overline{\omega} \to \mathbb{R}$, is of the same order as the thickness of the shell.

The middle surface of the corresponding shallow shell is thus parametrised in Cartesian coordinates by the mapping $\bm{\theta}^\varepsilon:\overline{\omega} \to \mathbb{E}^3$ defined by:
$$
\bm{\theta}^\varepsilon(y):=(y_1,y_2,\theta^\varepsilon(y)),\quad\textup{ for all }y\in\overline{\omega}.
$$

The shallow shells we are considering in this section are all made of a homogeneous and isotropic material, they are clamped on their entire lateral boundary, and they are subjected to the action of applied body forces only. For what concerns surface traction forces, the mathematical models characterised by the confinement condition considered in this paper (confinement condition which is also considered in~\cite{Leger2008} in a less complex geometrical framework) do not take any surface traction forces into account. 
Indeed, there could be no surface traction forces applied to the portion of the shell surface that engages contact with the obstacle.  The elastic behaviour of the shallow shell is then described by means of its two {Lam\'e constants} $\lambda \ge 0$ and $\mu >0$ (cf., e.g., \cite{Ciarlet1988}).

The function space over which the original problem is posed is the following (cf., e.g., \cite{Ciarlet1997}):
$$
\bm{V}(\omega):=\{\bm{\eta}=(\eta_i) \in H^1(\omega)\times H^1(\omega)\times H^2(\omega); \eta_i=\partial_{\vec{\nu}}\eta_3=0 \textup{ on }\gamma\}.
$$ 

The geometrical constraint we subject the shallow shell to amounts to requiring that the reference configuration of the shell and its admissible deformations remain confined in a prescribed half-space, which is identified with a unit-vector $\bm{q}$ that belongs to the orthogonal complement of the plane identifying one such half-space \emph{and} that points towards the prescribed half-space.
The obstacle the shallow shell does not have to cross is thus identified with the plane describing the boundary of the half-space where the shallow shell has to remain confined. As we shall see next, the corresponding confinement condition will have to bear \emph{at once} on all the components of the displacement vector field.
The confinement condition we are here considering is more general than the one considered in the papers~\cite{Leger2008,Leger2010}, where the authors required the shallow shell to lie above the plane $\{x_3=0\}$.
The assumption on the geometry of the obstacle made in~\cite{Leger2008,Leger2010} allowed the exploitation of the fact that the solution of the equilibrium problem for shallow shells is a Kirchhoff-Love field so as to ``separate'' the \emph{variational inequalities} governing the variations in the transverse component of the displacement vector field from the \emph{variational equations} governing the variations in the tangential components of the displacement vector field.
Apart from this, the assumptions on the geometry of the obstacle made in~\cite{Leger2008,Leger2010} ensured the convergence of the Finite Element Method proposed in~\cite{PS} for approximating the solution of the obstacle problem studied in~\cite{Leger2008,Leger2010}. 
For more general obstacle geometries like the ones we are considering in this section, the ``separation'' of the tangential components of the displacement vector field from the transverse component of the displacement vector field does not hold, in general. Moreover, it appears that the techniques proposed in~\cite{PS} are not applicable to the case where the confinement condition hinges at once on all the components of the displacement vector field.

In the same spirit as~\cite{Leger2008,Leger2010}, we assume that:
\begin{equation}
\label{refConf}
\bm{\theta}^\varepsilon \cdot \bm{q} >0,\quad\textup{ in }\overline{\omega}.
\end{equation}

For the sake of notational compactness, define the space associated with the tangential components of the displacement vector field
$$
\vec{V}_H(\omega):=\{\vec{\eta}_H=(\eta_\alpha) \in H^1(\omega)\times H^1(\omega); \eta_\alpha=0 \textup{ on }\gamma\}.
$$

Let $V_3(\omega):=H^2_0(\omega)$. Therefore, we can write:
$$
\bm{V}(\omega)=\vec{V}_H(\omega) \times V_3(\omega).
$$

The displacement vector field must thus to the following subset of the space $\bm{V}(\omega)$:
\begin{equation*}
	\bm{U}_K(\omega):=\{\bm{\eta}=(\vec{\eta}_H,\eta_3)\in \bm{V}(\omega); (\bm{\theta}^\varepsilon + \bm{\eta}) \cdot \bm{q} \ge 0\textup{ a.e. in }\omega\}.
\end{equation*}

The subscript ``\emph{K}''in the definition of the set $\bm{U}_K(\omega)$ aptly recalls the connection of the model we are studying with \emph{Koiter's model}~\cite{Koiter1959,Koiter1966,Koiter1970}, which is posed over the same functional space.

Let us observe that the set $\bm{U}_K(\omega)$ is non-empty as $\bm{\eta}=\bm{0} \in \bm{U}_K(\omega)$ in view of the assumption~\eqref{refConf}. It is also straightforward to observe that the set $\bm{U}_K(\omega)$ is closed and convex.

Let $\varepsilon>0$ be given once and for all, and define the set:
$$
\Omega^ \varepsilon := \omega \times (-\varepsilon,\varepsilon).
$$

Let $x^\varepsilon = (x^\varepsilon_i)$ denote a generic point in the set $\overline{\Omega^\varepsilon}$, with $x^\varepsilon_\alpha=y_\alpha$.
We are now ready to state the de-scaled two-dimensional limit obstacle problem - akin to the one recovered in~\cite{Leger2008,Leger2010} - governing the deformation of shallow shells subjected to remaining confined in a prescribed half-space. Note that, differently from other linearly elastic shells models that are formulated in \emph{curvilinear coordinates}, we are here opting for a formulation in \emph{Cartesian coordinates}, since it is judicious to regard shallow shell models as ``\emph{more similar}'' to linearly elastic plate models rather than to linearly elastic shell models.

\begin{customprob}{$\mathcal{P}^\varepsilon(\omega)$}
	\label{problem5}
	Find $\bm{\zeta}^\varepsilon=(\vec{\zeta}_H^\varepsilon, \zeta_3^\varepsilon) \in \bm{U}_K(\omega)$ satisfying:
	\begin{equation*}
		\begin{aligned}
			&-\int_{\omega}m_{\alpha\beta}^\varepsilon(\zeta_3^\varepsilon)\partial_{\alpha\beta}(\eta_3-\zeta_3^\varepsilon) \dd y 
			+\int_{\omega} n_{\alpha\beta}^{\theta,\varepsilon}(\bm{\zeta}^\varepsilon) e_{\alpha\beta}^{\theta,\varepsilon}(\bm{\eta}-\bm{\zeta}^\varepsilon) \dd y
			\ge \int_{\omega} p^{i,\varepsilon} (\eta_i-\zeta_i^\varepsilon) \dd y - \int_{\omega} s_\alpha^\varepsilon \partial_\alpha(\eta_3-\zeta_3^\varepsilon)\dd y,
		\end{aligned}
	\end{equation*}
	for all $\bm{\eta}=(\vec{\eta}_H,\eta_3) \in \bm{U}_K(\omega)$, where
	\begin{equation*}
		\begin{cases}
			&\lambda\ge 0, \mu>0\quad\textup{ are the Lam\'e constants},\\
			&c_0(\lambda,\mu):=\dfrac{\lambda\mu}{\lambda+2\mu},\\
			&m_{\alpha\beta}^\varepsilon(\zeta_3^\varepsilon):=-\varepsilon^3\left\{\dfrac{4}{3}c_0(\lambda,\mu) \Delta \zeta_3^\varepsilon \delta_{\alpha\beta}+\dfrac{4}{3}\mu \partial_{\alpha\beta} \zeta_3^\varepsilon\right\},\\
			&e_{\alpha\beta}^{\theta,\varepsilon}(\bm{\zeta}^\varepsilon):=\dfrac{1}{2}(\partial_\alpha \zeta_\beta^\varepsilon+\partial_\beta\zeta_\alpha^\varepsilon)+\dfrac{1}{2}(\partial_\alpha\theta^\varepsilon \partial_\beta\zeta_3^\varepsilon+\partial_\beta\theta^\varepsilon \partial_\alpha\zeta_3^\varepsilon),\\
			&n_{\alpha\beta}^{\theta,\varepsilon}(\bm{\zeta}^\varepsilon):=\varepsilon\left\{4c_0(\lambda,\mu)e_{\sigma\sigma}^{\theta,\varepsilon}(\bm{\zeta}^\varepsilon)\delta_{\alpha\beta}+4\mu e_{\alpha\beta}^{\theta,\varepsilon}(\bm{\zeta}^\varepsilon)\right\},\\
			&p^{i,\varepsilon}:=\int_{-\varepsilon}^{\varepsilon}f_i^\varepsilon \dd x_3^\varepsilon,\\
			&s_\alpha^\varepsilon:=\int_{-\varepsilon}^{\varepsilon} x_3^\varepsilon f_\alpha^\varepsilon \dd x_3^\varepsilon.
		\end{cases}
	\end{equation*}
	\bqed
\end{customprob}

We also make the following assumptions on the scalings of the applied body forces. Recall that these scalings are justified by the rigorous asymptotic analysis carried out in~\cite{Ciarlet1997}. Let us define $\Omega:=\omega\times(-1,1)$. We assume that there exist functions $f_i \in L^2(\Omega )$ \emph{independent of} $\varepsilon$ such that the following \emph{assumptions on the data} hold:
\begin{equation}
\label{scalingsdata}
\begin{aligned}
f^\varepsilon_\alpha(x^\varepsilon)&=\varepsilon^2 f_\alpha(x),\quad\textup{ at each }x=(x_i) \in \Omega,\\
f^\varepsilon_3(x^\varepsilon)&=\varepsilon^3 f_3(x),\quad\textup{ at each }x=(x_i) \in \Omega.
\end{aligned}
\end{equation}

Note that the left-hand side of the variational inequalities in Problem~\ref{problem5} is associated with the symmetric, continuous and $\bm{V}(\omega)$-elliptic (cf., e.g., Theorem~3.6-1 of~\cite{Ciarlet1997}) bilinear form $b^\varepsilon(\cdot,\cdot)$ given by (cf. Sections~3.5, 3.6 and~3.7 of~\cite{Ciarlet1997}):
\begin{equation*}
	b^\varepsilon(\bm{\zeta}^\varepsilon,\bm{\eta})=-\int_{\omega} m_{\alpha\beta}^\varepsilon(\zeta_3^\varepsilon) \partial_{\alpha\beta} \eta_3 \dd y  + \int_{\omega} n_{\alpha\beta}^{\theta,\varepsilon}(\bm{\zeta}^\varepsilon) e_{\alpha\beta}^{\theta,\varepsilon}(\bm{\eta}) \dd y.
\end{equation*}

A straightforward computation shows that:
\begin{equation*}
	b^\varepsilon(\bm{\eta},\bm{\eta}):=\int_{\omega}4c_0(\lambda,\mu)\left\{\dfrac{\varepsilon^3}{3}(\Delta \eta_3)^2+\varepsilon (e_{\sigma\sigma}^{\theta,\varepsilon}(\bm{\eta}))^2\right\} \dd y+4\mu \left\{\dfrac{\varepsilon^3}{3}\sum_{\alpha,\beta}\|\partial_{\alpha\beta}\eta_3\|_{L^2(\omega)}^2+\varepsilon\sum_{\alpha,\beta} \|e_{\alpha\beta}^{\theta,\varepsilon}(\bm{\eta})\|_{L^2(\omega)}^2\right\},
\end{equation*}
for all $\bm{\eta}=(\eta_i)\in\bm{V}(\omega)$.

Likewise, we associate the sum of the right-hand sides of the variational inequalities in Problem~\ref{problem5} with a linear and continuous form $\ell^\varepsilon$ defined as follows:
\begin{equation*}
	\ell^\varepsilon(\bm{\eta}):=\int_{\omega} p^{i,\varepsilon} \eta_i \dd y - \int_{\omega} s_\alpha^\varepsilon \partial_\alpha\eta_3 \dd y, \quad\textup{ for all } \bm{\eta}=(\eta_i)\in \bm{V}(\omega).
\end{equation*}

Therefore, Problem~\ref{problem5} admits a unique solution $\bm{\zeta}^\varepsilon$. Equivalently, we have that $\bm{\zeta}^\varepsilon \in \bm{U}_K(\omega)$ is the unique element in the set $\bm{U}_K(\omega)$ that minimises the energy functional
$$
J^\varepsilon(\bm{\eta}):=\dfrac{1}{2} b^\varepsilon(\bm{\eta},\bm{\eta}) -\ell^\varepsilon(\bm{\eta}),
$$
over the set $\bm{U}_K(\omega)$.

We now consider, in the same fashion as section~\ref{sec1}, a mixed formulation for Problem~\ref{problem5}. Prior to doing so, however, we need to assume - as in section~\ref{sec1} - that there exists a vector field $\vec{P}^\varepsilon \in H(\textup{div};\omega)$ such that $\textup{div }\vec{P}^\varepsilon = p^{3,\varepsilon}$ a.e. in $\omega$ (cf., e.g., Theorem~8.12-1 in~\cite{Ciarlet2025}).

\begin{customprob}{$\mathcal{Q}^\varepsilon(\omega)$}
	\label{problem7}
	Find $(\bm{\zeta}^\varepsilon,\vec{\xi}^\varepsilon) \in \mathbb{U}_K(\omega):=\{(\bm{\eta},\vec{\varphi}) \in \bm{H}^1_0(\omega) \times \vec{H}^1_0(\omega); \vec{\varphi} = \nabla \eta_3 \textup{ a.e. in }\omega \textup{ and } (\bm{\theta}^\varepsilon +\bm{\eta}) \cdot\bm{q} \ge 0 \textup{ a.e. in }\omega \}$ satisfying:
	\begin{equation*}
	\begin{aligned}
	&\dfrac{\varepsilon^3}{3} \int_{\omega}\left\{4c_0(\lambda,\mu)(\textup{div }\vec{\xi}^\varepsilon) \delta_{\alpha\beta} +4\mu \partial_{\alpha}\xi^\varepsilon_\beta \right\}\partial_{\alpha}(\varphi_\beta-\xi^\varepsilon_\beta) \dd y
	+\int_{\omega} n_{\alpha\beta}^{\theta,\varepsilon}(\vec{\zeta}^\varepsilon_H,\vec{\xi}^\varepsilon) e_{\alpha\beta}^{\theta,\varepsilon}(\vec{\eta}_H-\vec{\zeta}^\varepsilon_H,\vec{\varphi}-\vec{\xi}^\varepsilon) \dd y\\
	&\ge \int_{\omega} p^{\alpha,\varepsilon} (\eta_\alpha-\zeta_\alpha^\varepsilon) \dd y 
	-\int_{\omega} \vec{P}^\varepsilon \cdot (\vec{\varphi}-\vec{\xi}^\varepsilon) \dd y
	- \int_{\omega} s_\alpha^\varepsilon (\varphi_\alpha-\xi^\varepsilon_\alpha) \dd y,
	\end{aligned}
	\end{equation*}
	for all $(\bm{\eta},\vec{\varphi}) \in \mathbb{U}_K(\omega)$, where
		\begin{equation*}
	\begin{cases}
	&\lambda\ge 0, \mu>0\quad\textup{ are the Lam\'e constants},\\
	&c_0(\lambda,\mu)=\dfrac{\lambda\mu}{\lambda+2\mu},\\
	&e_{\alpha\beta}^{\theta,\varepsilon}(\vec{\zeta}^\varepsilon_H,\vec{\xi}^\varepsilon):=\dfrac{1}{2}(\partial_\alpha \zeta_\beta^\varepsilon+\partial_\beta\zeta_\alpha^\varepsilon)+\dfrac{1}{2}\left((\partial_\alpha\theta^\varepsilon) \xi^\varepsilon_\beta +(\partial_\beta\theta^\varepsilon) \xi^\varepsilon_\alpha\right),\\
	&n_{\alpha\beta}^{\theta,\varepsilon}(\vec{\zeta}^\varepsilon_H,\vec{\xi}^\varepsilon):=\varepsilon\left\{4c_0(\lambda,\mu)e_{\sigma\sigma}^{\theta,\varepsilon}(\vec{\zeta}^\varepsilon_H,\vec{\xi}^\varepsilon)\delta_{\alpha\beta}+4\mu e_{\alpha\beta}^{\theta,\varepsilon}(\vec{\zeta}^\varepsilon_H,\vec{\xi}^\varepsilon)\right\},\\
	&p^{i,\varepsilon}:=\int_{-\varepsilon}^{\varepsilon}f_i^\varepsilon \dd x_3^\varepsilon,\\
	&p^{3,\varepsilon} = \textup{div }\vec{P}^\varepsilon,\\
	&s_\alpha^\varepsilon:=\int_{-\varepsilon}^{\varepsilon} x_3^\varepsilon f_\alpha^\varepsilon \dd x_3^\varepsilon.
	\end{cases}
	\end{equation*}
	\bqed
\end{customprob}

Clearly, Problem~\ref{problem7} admits a unique solution as a re-writing of Problem~\ref{problem5}. It is straightforward to observe that the set $\mathbb{U}_K(\omega)$ is non-empty, closed and convex, and that the components of the tensor $\bm{e}^{\theta,\varepsilon}(\vec{\eta}_H,\vec{\varphi})=(e_{\alpha\beta}^{\theta,\varepsilon}(\vec{\eta}_H,\vec{\varphi}))$ are linear and continuous with respect to the natural norm of the vector space $\bm{H}^1_0(\omega) \times \vec{H}^1_0(\omega)$.

Once again, we denote by $\kappa$ a positive penalty parameter that is meant to approach zero. Define the mapping $\bm{\beta}^\varepsilon:\bm{L}^2(\omega) \to \bm{L}^2(\omega)$ by:
\begin{equation*}
	\bm{\beta}^\varepsilon(\bm{\eta}):=-\{(\bm{\theta}^\varepsilon+\bm{\eta})\cdot \bm{q}\}^{-} \bm{q},\quad\textup{ for all }\bm{\eta} \in \bm{V}(\omega).
\end{equation*}

In the same fashion as in section~\ref{sec1}, we add a further penalty term which is meant to steer the solutions towards states where the \emph{dual variable} is \emph{ as close as possible} to the gradient of the transverse component of the \emph{primal variable}.

\begin{customprob}{$\mathcal{Q}^\varepsilon_\kappa(\omega)$}
	\label{problem8}
	Find $(\bm{\zeta}^\varepsilon_\kappa,\vec{\xi}^\varepsilon_\kappa) \in \bm{H}^1_0(\omega) \times \vec{H}^1_0(\omega)$ such that $\nabla \vec{\xi}^\varepsilon_\kappa = (\nabla \vec{\xi}^\varepsilon_\kappa)^T$ a.e. in $\omega$, and satisfying:
	\begin{equation*}
	\begin{aligned}
	&\dfrac{\varepsilon^3}{3} \int_{\omega}\left\{4c_0(\lambda,\mu)(\textup{div }\vec{\xi}^\varepsilon_\kappa) \delta_{\alpha\beta} +4\mu \partial_{\alpha}\xi^\varepsilon_{\beta,\kappa} \right\} (\partial_{\alpha}\varphi_\beta) \dd y+\int_{\omega} n_{\alpha\beta}^{\theta,\varepsilon}(\vec{\zeta}^\varepsilon_{H,\kappa},\vec{\xi}^\varepsilon_\kappa) e_{\alpha\beta}^{\theta,\varepsilon}(\vec{\eta}_H,\vec{\varphi}) \dd y\\
	&\qquad+\dfrac{\varepsilon^3}{\kappa}\int_{\omega} (\vec{\xi}^\varepsilon_\kappa - \nabla \zeta^\varepsilon_{3,\kappa}) \cdot (\vec{\varphi} - \nabla \eta_3) \dd y+\dfrac{\varepsilon^3}{\kappa}\int_{\omega} \bm{\beta}^\varepsilon(\bm{\zeta}^\varepsilon_\kappa) \cdot \bm{\eta} \dd y
	= \int_{\omega} p^{\alpha,\varepsilon} \eta_\alpha \dd y -\int_{\omega} \vec{P}^\varepsilon \cdot \vec{\varphi} \dd y - \int_{\omega} s_\alpha^\varepsilon \varphi_\alpha \dd y,
	\end{aligned}
	\end{equation*}
	for all $(\bm{\eta},\vec{\varphi}) \in \bm{H}^1_0(\omega) \times \vec{H}^1_0(\omega)$ such that $\nabla \vec{\varphi} = (\nabla \vec{\varphi})^T$ a.e. in $\omega$.
	\bqed
\end{customprob}

As a remark, we observe that the choice for the powers of the parameter $\varepsilon$ - which, we recall, is regarded as fixed - made in Problem~\ref{problem8} are derived from the de-scalings descending from the asymptotic analysis carried out in Theorem~3.5-1 of~\cite{Ciarlet1997}. We anticipate that these de-scalings will play no role in determining the convergence of the numerical schemes for approximating the solution of Problem~\ref{problem5}. As a result of the de-scalings for Problem~\ref{problem5}, we consider the following regularity-preserving natural de-scalings for the solutions of Problem~\ref{problem8} below:
\begin{equation}
	\label{descalings-1}
	\begin{aligned}
		\zeta^\varepsilon_{\alpha,\kappa} &= \varepsilon^2 \zeta_{\alpha,\kappa},\\
		\zeta^\varepsilon_{3,\kappa} &= \varepsilon \zeta_{3,\kappa},\\
		\xi^\varepsilon_{\alpha,\kappa} &=\varepsilon \xi_{\alpha,\kappa}.
	\end{aligned}
\end{equation}

We now establish the existence an uniqueness of solutions for Problem~\ref{problem8}. The first step to achieve this goal consists in establishing the coerciveness of the left-hand side of the variational equations appearing in Problem~\ref{problem8} by means of an inequality of Korn's type in \emph{mixed coordinates}. Critical to establishing this estimate is the assumed \emph{simple connectedness} of the Lipschitz domain $\omega$, which will allow us to apply the weak Poincar\'e lemma (cf., e.g., Theorem~8.20-4 of~\cite{Ciarlet2025}).

\begin{theorem}
\label{KornMixed}
Let $\omega \subset \mathbb{R}^2$ be a simply connected Lipschitz domain. Let $\theta \in \mathcal{C}^3(\overline{\omega})$ be given.
Then, there exists a constant $C_K>0$ such that:
\begin{equation*}
	\left\{\|\vec{\eta}_H\|_{\vec{H}^1_0(\omega)}^2+\|\vec{\varphi}\|_{\vec{H}^1_0(\omega)}^2\right\}^{1/2}\le C_K \left\{\dfrac{1}{3}\sum_{\alpha,\beta} \|\partial_{\alpha}\varphi_\beta\|_{L^2(\omega)}^2 + \sum_{\alpha,\beta} \|e_{\alpha\beta}^\theta(\vec{\eta}_H,\vec{\varphi})\|_{L^2(\omega)}^2\right\}^{1/2},
\end{equation*}
for all $(\vec{\eta}_H,\vec{\varphi}) \in \vec{H}^1_0(\omega) \times \vec{H}^1_0(\omega)$ such that $\nabla \vec{\varphi} = (\nabla \vec{\varphi})^T$ a.e. in $\omega$.
\end{theorem}
\begin{proof}	
	Let $\Omega:=\omega \times (-1,1)$ and define the space $\bm{V}(\Omega):=\{\bm{w}=(w_i) \in \bm{H}^1(\Omega); w_i =0 \textup{ on }\gamma \times (-1,1)\}$. For each $\bm{v}=(v_i) \in \bm{V}(\Omega)$, define the tensor $\bar{\bm{e}}^\theta(\bm{v})=(\bar{e}_{\alpha\beta}^\theta(\bm{v}))$ as follows (cf., e.g., Theorem~3.4-1 of~\cite{Ciarlet1997}):
	\begin{equation*}
	\begin{aligned}
		\bar{e}_{\alpha\beta}^\theta(\bm{v})&:=\dfrac{1}{2}(\partial_{\alpha}v_\beta+\partial_{\beta}v_\alpha)-\dfrac{1}{2}\left((\partial_{\beta}\theta)(\partial_3 v_\alpha)+(\partial_{\alpha}\theta)(\partial_3 v_\beta)\right),\\
		\bar{e}_{\alpha3}^\theta(\bm{v})&:=\dfrac{1}{2}(\partial_{\alpha}v_3+\partial_3 v_\alpha)-\dfrac{1}{2}(\partial_{\alpha}\theta)(\partial_3 v_3),\\
		\bar{e}_{33}^\theta(\bm{v})&:=\partial_3 v_3.
	\end{aligned}
	\end{equation*}
	
	Let us recall that a generalised inequality of Korn's type (cf., e.g., Theorem~3.4-1 of~\cite{Ciarlet1997} for the notation appearing in the next inequality) with homogeneous Dirichlet boundary conditions on $\gamma\times (-1,1)$ gives that there exists a constant $C>0$ such that:
	\begin{equation*}
		|\bm{v}|_{\bm{H}^1(\Omega)} \le C \left\{\sum_{i,j}\|\bar{e}_{ij}^{\theta}(\bm{v})\|_{L^2(\Omega)}^2\right\}^{1/2}, \quad\textup{ for all } \bm{v} \in \bm{V}(\Omega).
	\end{equation*}
	
	Observe that the symmetry of $\nabla\vec{\varphi}$ implies that $\textup{curl }\vec{\varphi}=0 \in \mathbb{R}$ a.e. in $\omega$.
	An application of the weak Poincar\'e lemma (cf., e.g., Theorem~8.20-4 of~\cite{Ciarlet2025}) gives that there exists a function $\eta_3 \in L^2(\omega)$ such that
	\begin{equation}
		\label{Poincare}
		\nabla \eta_3 = \vec{\varphi},\quad\textup{ in }\vec{L}^2(\omega),
	\end{equation}
	and, moreover, any other function $\tilde{\eta}_3 \in L^2(\omega)$ such that $\nabla\tilde{\eta}_3=\vec{\varphi}$ a.e. in $\omega$ is of the form $\tilde{\eta}_3=\eta_3+\tilde{C}$, for some constant $\tilde{C} \in \mathbb{R}$.
	Observe that, since $\eta_3 \in H^1(\omega)$ and that $\nabla \eta_3 =\vec{0}$ on $\gamma$, we can apply Lemma~8.1 in~\cite{Brez11} on each of the \emph{overlapping} local charts covering the boundary of the boundary $\gamma$ so as to infer that $\eta_3$ is constant along $\gamma$. Letting $\tilde{C}= -\eta_3|_{\gamma}$ we obtain that $\tilde{\eta}_3 \in H^1_0(\omega)$.
	
	In light of~\eqref{Poincare}, we are in a position to define the vector field $\bm{v}=(v_i) \in \bm{V}(\Omega)$ as follows:
	\begin{equation*}
		\begin{aligned}
			v_\alpha(y,x_3)&:=\eta_\alpha -x_3 \varphi_\alpha, \quad(y,x_3) \in \Omega,\\
			v_3(y,x_3)&:= \tilde{\eta}_3(y), \quad (y,x_3) \in \Omega.
		\end{aligned}
	\end{equation*}
	Therefore, it is immediate to verify the following identities:
	\begin{equation*}
		\begin{aligned}
			\bar{e}_{\alpha3}^\theta(\bm{v})&=\dfrac{\partial_{\alpha}\tilde{\eta}_3+\partial_3 v_\alpha}{2}-\dfrac{1}{2}(\partial_{\alpha} \theta) (\partial_3 \tilde{\eta}_3)
			=0,\quad\textup{ a.e in }\omega,\\
			\bar{e}_{33}^\theta(\bm{v})&=0,\quad\textup{ a.e. in }\omega,\\
			\bar{e}_{\alpha\beta}^\theta(\bm{v})&=\dfrac{1}{2}\left(\partial_{\alpha}\eta_\beta+\partial_{\beta}\eta_\alpha-x_3(\partial_{\alpha}\varphi_\beta+\partial_{\beta}\varphi_\alpha)\right)+\dfrac{1}{2}\left((\partial_{\beta}\theta)\varphi_\alpha+(\partial_{\alpha}\theta)\varphi_\beta\right)\\
			&=\dfrac{1}{2}\left(\partial_{\alpha}\eta_\beta+\partial_{\beta}\eta_\alpha+(\partial_{\beta}\theta)\varphi_\alpha+(\partial_{\alpha}\theta)\varphi_\beta\right)-\dfrac{x_3}{2}(\partial_{\alpha}\varphi_\beta+\partial_{\beta}\varphi_\alpha)\\
			&=e_{\alpha\beta}^\theta(\vec{\eta}_H,\vec{\varphi})-x_3 \partial_{\alpha}\varphi_\beta, \quad\textup{ a.e. in }\omega.
		\end{aligned}
	\end{equation*}
	
	Note that the last equality in the algebraic manipulation of the terms $\bar{e}_{\alpha\beta}^\theta(\bm{v})$ descends as a result of the assumed symmetry for $\nabla \vec{\varphi}$. We denote by $e_{\alpha\beta}^\theta(\vec{\eta}_H,\vec{\varphi})$ the components of the tensor introduced in Problem~\ref{problem7} in the special case where $\varepsilon=1$.
	A straightforward computation gives
	\begin{equation}
		\label{eq7}
		\begin{aligned}
		&\sum_{i,j}\|\bar{e}_{ij}^\theta(\bm{v})\|_{L^2(\Omega)}^2 = \sum_{\alpha,\beta}\|\bar{e}_{\alpha\beta}^\theta(\bm{v})\|_{L^2(\Omega)}^2
		=\sum_{\alpha,\beta}\int_{\Omega} \left\{|e_{\alpha\beta}^\theta(\vec{\eta}_H,\vec{\varphi})|^2 +x_3^2 |\partial_{\alpha}\varphi_\beta|^2\right\} \dd y \dd x_3\\
		&=2\sum_{\alpha,\beta} \|e_{\alpha\beta}^\theta(\vec{\eta}_H,\vec{\varphi})\|_{L^2(\omega)}^2+\dfrac{2}{3}\sum_{\alpha,\beta}\|\partial_{\beta}\varphi_\alpha\|_{L^2(\omega)}^2,
	\end{aligned}
	\end{equation}
	on the one hand. On the other hand, the definition of $v_3$ gives:
	\begin{equation}
		\label{eq8}
		\begin{aligned}
			&|\bm{v}|_{\bm{H}^1(\Omega)}^2=\|\partial_1 \eta_1-x_3 \partial_1\varphi_1\|_{L^2(\Omega)}^2+\|\partial_2 \eta_1-x_3 \partial_2\varphi_1\|_{L^2(\Omega)}^2+\|\partial_1 \eta_2-x_3 \partial_1\varphi_2\|_{L^2(\Omega)}^2\\
			&\qquad+\|\partial_2\eta_2-x_3 \partial_2\varphi_2\|_{L^2(\Omega)}^2+2\sum_{\alpha}\|\varphi_\alpha\|_{L^2(\Omega)}^2\\
			&=2\sum_{\alpha,\beta}\|\partial_{\beta}\eta_\alpha\|_{L^2(\omega)}^2+\dfrac{2}{3}\sum_{\alpha,\beta}\|\partial_{\beta}\varphi_\alpha\|_{L^2(\omega)}^2+4\sum_{\alpha}\|\varphi_\alpha\|_{L^2(\omega)}^2.
		\end{aligned}
	\end{equation}
	
	Combining~\eqref{eq7}, \eqref{eq8} and the generalised Korn's inequality in Theorem~3.4-1 of~\cite{Ciarlet1997} and the standard Poincar\'e--Friedrichs inequality (cf., e.g., Theorem~8.3-3 of~\cite{Ciarlet2025}), we obtain that:
	\begin{equation*}
		\sum_{\alpha,\beta} \|e_{\alpha\beta}^\theta(\vec{\eta}_H,\vec{\varphi})\|_{L^2(\omega)}^2+\dfrac{1}{3}\sum_{\alpha,\beta}\|\partial_{\beta}\varphi_\alpha\|_{L^2(\omega)}^2
		\ge \dfrac{1}{3C^2(1+c_P^2)}\left(\|\vec{\eta}_H\|_{\vec{H}^1_0(\omega)}^2+\|\vec{\varphi}\|_{\vec{H}^1_0(\omega)}^2\right).
	\end{equation*}
	
	The proof is complete by letting $C_K:=\sqrt{3(1+c_P^2)}C$.
\end{proof}

Observe that the inequality of Korn's type in \emph{mixed coordinates} established in Theorem~\ref{KornMixed} required the assumption that the gradient matrix of the \emph{dual variable} was symmetric to compensate the \emph{lack of rigidity} for surfaces which are compatible with the definition of shallow shells. Indeed, differently from other types of linearly elastic shells (e.g., linearly elastic elliptic membrane shells~\cite{CiaLods1996a,CiaLods1996b} or linearly elastic flexural shells~\cite{CiaLodsMia1996}) for which a rigidity theorem is available for the corresponding two-dimensional limit models recovered as a result of a rigorous asymptotic analysis, the two-dimensional limit model for shallow shells does not enjoy, in general, one such \emph{rigidity} property.

It is straightforward to see that, in the case where the middle surface of the linearly elastic shallow shell under consideration is flat, the components $e^{\theta,\varepsilon}_{\alpha\beta}$ of the strain tensor in \emph{mixed coordinates} reduce to the standard symmetrised gradient with respect to the \emph{primal variable}, namely,
$$
e^{\theta,\varepsilon}_{\alpha\beta}(\vec{\eta}_H,\vec{\varphi})=e^\varepsilon_{\alpha\beta}(\vec{\eta}_H):=\dfrac{1}{2}(\partial_\alpha\eta_\beta+\partial_{\beta}\eta_\alpha),\quad\textup{ for all } (\vec{\eta}_H,\vec{\varphi}) \in \vec{H}^1_0(\omega) \times \vec{H}^1_0(\omega).
$$

Therefore, in the case where the middle surface of the linearly elastic shallow shell under consideration is flat, the inequality of Korn's type in \emph{mixed coordinates} established in Theorem~\ref{KornMixed} reduces to the standard Korn's inequality in $\mathbb{R}^2$ (cf., e.g., Theorem~8.17-4 of~\cite{Ciarlet2025}) and it is thus not necessary to assume the symmetry of the gradient matrix for the \emph{dual variable} in order to establish one such inequality.

We are now ready to establish the existence and uniqueness of solutions for Problem~\ref{problem8}.

\begin{theorem}
	\label{exunp8}
	For each $\varepsilon>0$ and for each $0<\kappa\le 2C_K^2$, where $C_K$ is the positive constant constructed in Theorem~\ref{KornMixed}, Problem~\ref{problem8} admits a unique solution $(\bm{\zeta}^\varepsilon_\kappa,\vec{\xi}^\varepsilon_\kappa)$. Besides, we have that the following convergences hold
	\begin{align*}
	\bm{\zeta}^\varepsilon_\kappa \rightharpoonup \bm{\zeta}^\varepsilon&, \textup{ in }\bm{H}^1_0(\omega) \textup{ as }\kappa \to 0,\\
	\vec{\xi}^\varepsilon_\kappa \rightharpoonup \nabla \zeta^\varepsilon_3&, \textup{ in }\vec{H}^1_0(\omega) \textup{ as }\kappa \to 0,
	\end{align*}
	where $\bm{\zeta}^\varepsilon \in \bm{U}_K(\omega)$ is the unique solution of Problem~\ref{problem5}.
\end{theorem}
\begin{proof}
	In light of the Korn inequality in \emph{mixed coordinates} (Theorem~\ref{KornMixed}), the \emph{shallowness criterion}~\eqref{critShall}, the de-scalings~\eqref{descalings-1}, and Young's inequality~\cite{Young1912}, we have that there exists a constant $C_K>0$ independent of $\varepsilon$ such that:
	\begin{equation}
		\label{KornMixed2}
	\begin{aligned}
		&\dfrac{\varepsilon^3}{3} \sum_{\alpha,\beta}\|\partial_{\beta}\xi^\varepsilon_{\alpha,\kappa}\|_{L^2(\omega)}^2
		+\dfrac{\varepsilon^3}{\kappa}\|\vec{\xi}^\varepsilon_\kappa-\nabla\zeta^\varepsilon_{3,\kappa}\|_{\vec{L}^2(\omega)}^2
		+\varepsilon\sum_{\alpha,\beta}\|e_{\alpha\beta}^{\theta,\varepsilon}(\vec{\zeta}^\varepsilon_{H,\kappa},\vec{\xi}^\varepsilon_\kappa)\|_{L^2(\omega)}^2\\
		&=\dfrac{\varepsilon^5}{3} \sum_{\alpha,\beta}\|\partial_{\beta}\xi_{\alpha,\kappa}\|_{L^2(\omega)}^2
		+\dfrac{\varepsilon^5}{\kappa}\|\vec{\xi}_\kappa-\nabla\zeta_{3,\kappa}\|_{\vec{L}^2(\omega)}^2
		+\varepsilon^5 \sum_{\alpha,\beta}\|e^\theta_{\alpha,\beta}(\vec{\zeta}_{H,\kappa},\vec{\xi}_\kappa)\|_{L^2(\omega)}^2\\
		&\ge \dfrac{\varepsilon^5}{C_K^2}\left\{\|\vec{\zeta}_{H,\kappa}\|_{\vec{H}^1_0(\omega)}^2+\|\vec{\xi}_\kappa\|_{\vec{H}^1_0(\omega)}^2\right\}+\dfrac{\varepsilon^5}{\kappa}\|\vec{\xi}_\kappa-\nabla\zeta_{3,\kappa}\|_{\vec{L}^2(\omega)}^2\\
		&\ge\dfrac{\varepsilon^5}{C_K^2}\left(\|\vec{\zeta}_{H,\kappa}\|_{\vec{H}^1_0(\omega)}^2+\dfrac{\|\vec{\xi}_\kappa\|_{\vec{H}^1_0(\omega)}^2}{2}\right)+\dfrac{\varepsilon^5}{2C_K^2}\left(\|\vec{\xi}_\kappa-\nabla\zeta_{3,\kappa}\|_{\vec{L}^2(\omega)}^2+\|\vec{\xi}_\kappa\|_{\vec{L}^2(\omega)}^2\right)\\
		&\ge\dfrac{\varepsilon^3}{2C_K^2}\left(\|\vec{\zeta}^\varepsilon_{H,\kappa}\|_{\vec{H}^1_0(\omega)}^2+\|\vec{\xi}^\varepsilon_\kappa\|_{\vec{H}^1_0(\omega)}^2\right)+\dfrac{\varepsilon^3}{4C_K^2(1+c_P^2)}\|\zeta^\varepsilon_{3,\kappa}\|_{H^1_0(\omega)}^2\\
		&\ge\dfrac{\varepsilon^3}{4C_K^2(1+c_P^2)}\left(\|\bm{\zeta}^\varepsilon_\kappa\|_{\bm{H}^1_0(\omega)}^2+\|\vec{\xi}^\varepsilon_\kappa\|_{\vec{H}^1_0(\omega)}^2\right).
	\end{aligned}
	\end{equation}
	
	From~\eqref{KornMixed2} we infer that the non-linear operator associated with the left-hand side of Problem~\ref{problem8} is coercive. The strong monotonicity follows by applying the same coercivity estimate to the difference of two admissible pairs. An application of the Minty--Browder Theorem (cf., e.g., Theorem~12.5-3 of~\cite{Ciarlet2025}) allows us to infer that Problem~\ref{problem8} admits a unique solution.
	
	Let us now establish the announced weak convergences. Specialise $(\bm{\eta},\vec{\varphi})=(\bm{\zeta}^\varepsilon_\kappa,\vec{\xi}^\varepsilon_\kappa)$ in the variational equations of Problem~\ref{problem8}. We obtain that an application of~\eqref{KornMixed2} and Young's inequality~\cite{Young1912} gives:
	\begin{equation*}
	\begin{aligned}
	&\dfrac{\mu\varepsilon^3}{(1+c_P^2)C_K^2}\left\{\|\vec{\zeta}^\varepsilon_{H,\kappa}\|_{\vec{H}^1_0(\omega)}^2+\|\vec{\xi}^\varepsilon_\kappa\|_{\vec{H}^1_0(\omega)}^2\right\}\\
	&\le\dfrac{4}{3} c_0(\lambda,\mu)\varepsilon^3\|\textup{div }\vec{\xi}^\varepsilon_\kappa\|_{L^2(\omega)}^2 +\dfrac{\varepsilon^3}{3} 4\mu \sum_{\alpha,\beta}\|\partial_{\alpha}\xi^\varepsilon_{\beta,\kappa}\|_{L^2(\omega)}^2
	+4c_0(\lambda,\mu)\varepsilon\|e_{\sigma\sigma}^{\theta,\varepsilon}(\vec{\zeta}^\varepsilon_{H,\kappa},\vec{\xi}^\varepsilon_\kappa)\|_{L^2(\omega)}^2\\
	&\qquad+4\mu\varepsilon\sum_{\alpha,\beta}\|e_{\alpha\beta}^{\theta,\varepsilon}(\vec{\zeta}^\varepsilon_{H,\kappa},\vec{\xi}^\varepsilon_\kappa)\|_{L^2(\omega)}^2
	+\dfrac{\varepsilon^3}{\kappa}\|\bm{\beta}^\varepsilon(\bm{\zeta}^\varepsilon_\kappa)\|_{\bm{L}^2(\omega)}^2+\dfrac{\varepsilon^3}{\kappa}\int_{\omega} \left\{\left[\bm{\theta}^\varepsilon+\bm{\zeta}^\varepsilon_\kappa\right]\cdot\bm{q}\right\}^{-} (\bm{\theta}^\varepsilon \cdot\bm{q}) \dd y +\dfrac{\varepsilon^3}{\kappa} \|\vec{\xi}^\varepsilon_\kappa-\nabla\zeta^\varepsilon_{3,\kappa}\|_{\vec{L}^2(\omega)}^2\\
	&=\int_{\omega} p^{\alpha,\varepsilon} \zeta^\varepsilon_{\alpha,\kappa} \dd y -\int_{\omega} \vec{P}^\varepsilon \cdot \vec{\xi}^\varepsilon_\kappa \dd y-\int_{\omega}s^\varepsilon_\alpha \xi^\varepsilon_{\alpha,\kappa}\dd y\\
	&\le \dfrac{(1+c_P^2)C_K^2}{2\mu\varepsilon^3}\left\{\sum_{\alpha} \left(\|p^{\alpha,\varepsilon}\|_{L^2(\omega)}+\|s^\varepsilon_\alpha\|_{L^2(\omega)}\right)+\|\vec{P}^\varepsilon\|_{\vec{L}^2(\omega)}\right\}^2 +\dfrac{\mu\varepsilon^3}{2(1+c_P^2)C_K^2}\left\{\|\vec{\zeta}^\varepsilon_{H,\kappa}\|_{\vec{H}^1_0(\omega)}^2+\|\vec{\xi}^\varepsilon_\kappa\|_{\vec{H}^1_0(\omega)}^2\right\}.
	\end{aligned}
	\end{equation*}
	
	Therefore, we obtain that
	\begin{equation}
	\label{bdd1}
	\|\vec{\zeta}^\varepsilon_{H,\kappa}\|_{\vec{H}^1_0(\omega)}^2+\|\vec{\xi}^\varepsilon_\kappa\|_{\vec{H}^1_0(\omega)}^2
	\le \dfrac{(1+c_P^2)^2C_K^4}{\mu^2 \varepsilon^6}\left\{\sum_{\alpha} \left(\|p^{\alpha,\varepsilon}\|_{L^2(\omega)}+\|s^\varepsilon_\alpha\|_{L^2(\omega)}\right)+\|\vec{P}^\varepsilon\|_{\vec{L}^2(\omega)}\right\}^2,
	\end{equation}
	where the right-hand side is independent of $\kappa$, and is also independent of $\varepsilon$ in light of the scalings for the applied body forces~\eqref{scalingsdata}.
	In light of~\eqref{bdd1}, we obtain that
	\begin{equation}
	\label{bdd2}
	\|\vec{\xi}^\varepsilon_\kappa-\nabla\zeta^\varepsilon_{3,\kappa}\|_{\vec{L}^2(\omega)} \le \dfrac{\sqrt{\mu(1+c_P^2)}C_K}{\mu\varepsilon^3} \left\{\sum_{\alpha} \left(\|p^{\alpha,\varepsilon}\|_{L^2(\omega)}+\|s^\varepsilon_\alpha\|_{L^2(\omega)}\right)+\|\vec{P}^\varepsilon\|_{\vec{L}^2(\omega)}\right\}\sqrt{\kappa},
	\end{equation}
	and that:
	\begin{equation}
	\label{bdd3}
	\|\bm{\beta}^\varepsilon(\bm{\zeta}^\varepsilon_\kappa)\|_{\bm{L}^2(\omega)}
	\le\dfrac{\sqrt{\mu(1+c_P^2)}C_K}{\mu\varepsilon^3}\left\{\sum_{\alpha} \left(\|p^{\alpha,\varepsilon}\|_{L^2(\omega)}+\|s^\varepsilon_\alpha\|_{L^2(\omega)}\right)+\|\vec{P}^\varepsilon\|_{\vec{L}^2(\omega)}\right\}\sqrt{\kappa}.
	\end{equation}
	
	In view of~\eqref{bdd1}--\eqref{bdd3}, we have that, up to passing to a subsequence,
	\begin{equation}
	\label{conv3}
	\begin{aligned}
	\vec{\zeta}^\varepsilon_{H,\kappa} \rightharpoonup \vec{\zeta}^\varepsilon_H&,\quad\textup{ in }\vec{H}^1_0(\omega) \textup{ as }\kappa\to0,\\
	\vec{\xi}^\varepsilon_\kappa \rightharpoonup \vec{\xi}^\varepsilon&,\quad\textup{ in }\vec{H}^1_0(\omega) \textup{ as }\kappa\to0,\\
	\|\vec{\xi}^\varepsilon_\kappa-\nabla\zeta^\varepsilon_{3,\kappa}\|_{\vec{L}^2(\omega)} \to 0&,\quad\textup{ as }\kappa\to0,\\
	\|\bm{\beta}^\varepsilon(\bm{\zeta}^\varepsilon_\kappa)\|_{\bm{L}^2(\omega)}\to 0&,\quad\textup{ as }\kappa\to0.
	\end{aligned}
	\end{equation}
	
	As a result of the linearity and continuity of the divergence operator, we obtain that:
	\begin{equation*}
		\textup{div }\vec{\xi}^\varepsilon_\kappa \rightharpoonup \textup{div }\vec{\xi}^\varepsilon ,\quad\textup{ in } L^2(\omega) \textup{ as }\kappa\to0.
	\end{equation*}
	
	Besides, combining~\eqref{bdd1} and \eqref{bdd2} with the triangle inequality gives:
	\begin{equation}
	\label{bdd4}
	\begin{aligned}
	&\|\nabla\zeta^\varepsilon_{3,\kappa}\|_{\vec{L}^2(\omega)} \le \|\vec{\xi}^\varepsilon_\kappa-\nabla\zeta^\varepsilon_{3,\kappa}\|_{\vec{L}^2(\omega)} + \|\vec{\xi}^\varepsilon_\kappa\|_{\vec{L}^2(\omega)}\\
	&\le\left(\sqrt{\mu(1+c_P^2)\kappa}+(1+c_P^2)C_K\right)\dfrac{C_K}{\mu\varepsilon^3}\left\{\sum_{\alpha} \left(\|p^{\alpha,\varepsilon}\|_{L^2(\omega)}+\|s^\varepsilon_\alpha\|_{L^2(\omega)}\right)+\|\vec{P}^\varepsilon\|_{\vec{L}^2(\omega)}\right\}.
	\end{aligned}
	\end{equation}
	
	Thanks to the de-scalings~\eqref{descalings-1}, the right-hand side of~\eqref{bdd4} is bounded independently of $\kappa$ and $\varepsilon$. Thanks to the classical Poincar\'e--Friedrichs inequality (cf., e.g., Theorem~8.3-3 in~\cite{Ciarlet2025}), we have that, up to passing to a subsequence, the following weak convergence takes place:
	\begin{equation}
	\label{conv4}
	\zeta^\varepsilon_{3,\kappa} \rightharpoonup \zeta^\varepsilon_3,\quad\textup{ in }H^1_0(\omega) \textup{ as }\kappa\to 0.
	\end{equation}
	
	Combining the estimate~\eqref{bdd3} with~\eqref{conv3}, \eqref{conv4}, the Rellich--Kondra\v{s}ov Theorem (cf., e.g., Theorem~8.4-3 in~\cite{Ciarlet2025}), and the strong monotonicity and Lipschitz continuity of the operator $\bm{\beta}^\varepsilon$ (cf., e.g., \cite{EvansGariepy2015}), we obtain that $\bm{\beta}^\varepsilon(\bm{\zeta}^\varepsilon)=\bm{0}$ in $\bm{L}^2(\omega)$, so that:
	\begin{equation*}
	(\bm{\theta}^\varepsilon+\bm{\zeta}^\varepsilon) \cdot \bm{q} \ge 0,\quad\textup{ a.e. in }\omega.
	\end{equation*}
	
	As a result of~\eqref{conv4}, we have that:
	\begin{equation*}
		\sqrt{\kappa} \|\nabla\zeta^\varepsilon_{3,\kappa}\|_{\vec{L}^2(\omega)} \to 0, \textup{ as } \kappa \to 0.
	\end{equation*}
	
	Combining the second and third convergences in~\eqref{conv3} with~\eqref{conv4}, and applying the Fatou Lemma gives
	\begin{equation}
	\label{eq9}
	\|\vec{\xi}^\varepsilon-\nabla\zeta^\varepsilon_3\|_{\vec{L}^2(\omega)}
	\le \liminf_{\kappa\to 0}	\|\vec{\xi}^\varepsilon_\kappa-\nabla\zeta^\varepsilon_{3,\kappa}\|_{\vec{L}^2(\omega)} =0,
	\end{equation}
	so that $\nabla\zeta^\varepsilon_3 = \vec{\xi}^\varepsilon$ in $\vec{H}^1_0(\omega)$. The fact that $\omega$ is a Lipschitz domain in turn gives that $\zeta^\varepsilon_3 \in H^2_0(\omega)$ so that $(\bm{\zeta}^\varepsilon,\vec{\xi}^\varepsilon) \in \mathbb{U}_K(\omega)$ or, equivalently, that $\bm{\zeta}^\varepsilon \in \bm{U}_K(\omega)$.
	
	The last thing to show is that the weak limit $(\bm{\zeta}^\varepsilon,\vec{\xi}^\varepsilon)$ is in fact the unique solution of Problem~\ref{problem7} or, equivalently, that $\bm{\zeta}^\varepsilon$ is the unique solution of Problem~\ref{problem5}. To this aim, let us test the variational equations in Problem~\ref{problem8} at $(\bm{\eta}-\bm{\zeta}^\varepsilon_\kappa,\vec{\varphi}-\vec{\xi}^\varepsilon_\kappa)$, where $(\bm{\eta},\vec{\varphi})$ is chosen arbitrarily in $\mathbb{U}_K(\omega)$. We obtain that:
	\begin{equation}
	\label{eq:3}
	\begin{aligned}
	&\dfrac{\varepsilon^3}{3} \int_{\omega}\left[4c_0(\lambda,\mu)(\textup{div }\vec{\xi}^\varepsilon_\kappa) \delta_{\alpha\beta}\right] \partial_{\alpha}(\varphi_\beta-\xi^\varepsilon_{\beta,\kappa}) \dd y\\
	&\qquad+\dfrac{\varepsilon^3}{3} \int_{\omega}4\mu (\partial_{\alpha}\xi^\varepsilon_{\beta,\kappa})  \partial_{\alpha}(\varphi_\beta-\xi^\varepsilon_{\beta,\kappa}) \dd y\\
	&\qquad+\int_{\omega} \varepsilon\left\{4c_0(\lambda,\mu) e_{\sigma\sigma}^{\theta,\varepsilon}(\vec{\zeta}^\varepsilon_{H,\kappa},\vec{\xi}^\varepsilon_\kappa)\delta_{\alpha\beta}+4\mu e_{\alpha\beta}^{\theta,\varepsilon}(\vec{\zeta}^\varepsilon_{H,\kappa},\vec{\xi}^\varepsilon_\kappa)\right\} e_{\alpha\beta}^{\theta,\varepsilon}(\vec{\eta}_H-\vec{\zeta}^\varepsilon_{H,\kappa},\vec{\varphi}-\vec{\xi}^\varepsilon_\kappa) \dd y\\
	&\qquad+\underbrace{\dfrac{\varepsilon^3}{\kappa}\int_{\omega} (\vec{\xi}^\varepsilon_\kappa - \nabla \zeta^\varepsilon_{3,\kappa}) \cdot \left((\vec{\varphi}-\vec{\xi}^\varepsilon_\kappa) - \nabla (\eta_3-\zeta^\varepsilon_{3,\kappa})\right) \dd y}_{\le 0}\\
	&\qquad-\dfrac{\varepsilon^3}{\kappa}\int_{\omega} \left\{(\bm{\theta}^\varepsilon+\bm{\zeta}^\varepsilon_\kappa) \cdot \bm{q}\right\}^{-} \left((\bm{\eta}-\bm{\zeta}^\varepsilon_\kappa)\cdot \bm{q}\right) \dd y\\
	=&\int_{\omega} p^{\alpha,\varepsilon} (\eta_\alpha-\zeta^\varepsilon_{\alpha,\kappa}) \dd y 
	-\int_{\omega} \vec{P}^\varepsilon \cdot (\vec{\varphi}-\vec{\xi}^\varepsilon_\kappa) \dd y
	- \int_{\omega} s_\alpha^\varepsilon (\varphi_\alpha-\xi^\varepsilon_{\alpha,\kappa}) \dd y.
	\end{aligned}
	\end{equation}

Passing to the $\limsup$ as $\kappa\to0$, and applying the Fatou Lemma gives:
	\begin{equation}
	\label{eq:3bis}
	\begin{aligned}
	&\limsup_{\kappa\to 0}\left(-\dfrac{\varepsilon^3}{3} \int_{\omega}4c_0(\lambda,\mu)(\textup{div }\vec{\xi}^\varepsilon_\kappa)^2 \dd y\right) \le -\dfrac{\varepsilon^3}{3} \int_{\omega}4c_0(\lambda,\mu)(\textup{div }\vec{\xi}^\varepsilon)^2 \dd y,\\
	&\limsup_{\kappa\to 0}\left(-\dfrac{\varepsilon^3}{3} \int_{\omega}4\mu \sum_{\alpha,\beta}(\partial_{\alpha}\xi^\varepsilon_{\beta,\kappa})^2 \dd y\right)
	\le -\dfrac{\varepsilon^3}{3} \int_{\omega}4\mu \sum_{\alpha,\beta}(\partial_{\alpha}\xi^\varepsilon_{\beta})^2 \dd y,\\
	&\limsup_{\kappa\to 0}\left(-\varepsilon \int_{\omega} 4c_0(\lambda,\mu)|e_{\sigma\sigma}^{\theta,\varepsilon}(\vec{\zeta}^\varepsilon_{H,\kappa},\vec{\xi}^\varepsilon_\kappa)|^2 \dd y\right)\le-4c_0(\lambda,\mu)\varepsilon\|e_{\sigma\sigma}^{\theta,\varepsilon}(\vec{\zeta}^\varepsilon_{H},\vec{\xi}^\varepsilon)\|_{L^2(\omega)}^2,\\
	&\limsup_{\kappa\to 0}\left(-4\mu\varepsilon\sum_{\alpha,\beta}\int_{\omega}|e_{\alpha\beta}^{\theta,\varepsilon}(\vec{\zeta}^\varepsilon_{H,\kappa},\vec{\xi}^\varepsilon_\kappa)|^2\dd y\right)\le-4\mu\varepsilon\sum_{\alpha,\beta}\int_{\omega}|e_{\alpha\beta}^{\theta,\varepsilon}(\vec{\zeta}^\varepsilon_H,\vec{\xi}^\varepsilon)|^2\dd y.
	\end{aligned}
	\end{equation}

Since $(\bm{\eta},\vec{\varphi})\in\mathbb{U}_K(\omega)$, we have that:
	\begin{equation}
	\label{eq:3ter}
	-\dfrac{\varepsilon^3}{\kappa}\int_{\omega} \left\{(\bm{\theta}^\varepsilon+\bm{\zeta}^\varepsilon_\kappa)\cdot\bm{q}\right\}^{-} (\bm{\eta}-\bm{\zeta}^\varepsilon_\kappa)\cdot\bm{q} \dd y
	\le-\dfrac{\varepsilon^3}{\kappa}\int_{\omega} \left|\left\{(\bm{\theta}^\varepsilon+\bm{\zeta}^\varepsilon_\kappa)\cdot \bm{q}\right\}^{-}\right|^2 \dd y\le 0.
	\end{equation}

Passing to the $\limsup$ as $\kappa\to0$ in~\eqref{eq:3}, and exploiting the convergences~\eqref{conv3}--\eqref{eq9} and~\eqref{eq:3bis}, \eqref{eq:3ter}, we obtain that:
	\begin{equation*}
	\begin{aligned}
	&\dfrac{\varepsilon^3}{3} \int_{\omega}\left[4c_0(\lambda,\mu)(\textup{div }\vec{\xi}^\varepsilon) \delta_{\alpha\beta}\right] \partial_{\alpha}(\varphi_\beta-\xi^\varepsilon_{\beta}) \dd y+\dfrac{\varepsilon^3}{3} \int_{\omega}4\mu (\partial_{\alpha}\xi^\varepsilon_{\beta})  \partial_{\alpha}(\varphi_\beta-\xi^\varepsilon_{\beta}) \dd y\\
	&\qquad+\int_{\omega} \varepsilon\left\{4c_0(\lambda,\mu)e_{\sigma\sigma}^{\theta,\varepsilon}(\vec{\zeta}^\varepsilon_H,\vec{\xi}^\varepsilon)\delta_{\alpha\beta}+4\mu e_{\alpha\beta}^{\theta,\varepsilon}(\vec{\zeta}^\varepsilon_H,\vec{\xi}^\varepsilon)\right\} e_{\alpha\beta}^{\theta,\varepsilon}(\vec{\eta}_H-\vec{\zeta}^\varepsilon_{H},\vec{\varphi}-\vec{\xi}^\varepsilon) \dd y\\
	&\ge\int_{\omega} p^{\alpha,\varepsilon} (\eta_\alpha-\zeta^\varepsilon_{\alpha}) \dd y 
	-\int_{\omega} \vec{P}^\varepsilon \cdot (\vec{\varphi}-\vec{\xi}^\varepsilon) \dd y
	-\int_{\omega} s_\alpha^\varepsilon (\varphi_\alpha-\xi^\varepsilon_{\alpha}) \dd y,
	\end{aligned}
	\end{equation*}
	for all $(\bm{\eta},\vec{\varphi}) \in \mathbb{U}_K(\omega)$. We have thus shown that $(\bm{\zeta}^\varepsilon_\kappa,\vec{\xi}^\varepsilon_\kappa)$ converges to the unique solution of the mixed variational inequalities  in Problem~\ref{problem7}.
\end{proof}

Let us now discuss the higher regularity for the solution of Problem~\ref{problem8}.
\begin{theorem}
	\label{th:shallow-reg}
	For each $0<\kappa\le 2C_K^2$, where $C_K$ is the positive constant constructed in Theorem~\ref{KornMixed}, the solution $(\bm{\zeta}^\varepsilon_\kappa,\vec{\xi}_\kappa^\varepsilon)$ of Problem~\ref{problem8} is of class $(\bm{V}(\omega)\cap (H^2(\omega)\times H^2(\omega) \times H^3(\omega)))\times (\vec{H}^1_0(\omega)\cap\vec{H}^2(\omega))$. Besides, there exists a constant $C>0$ independent of $\kappa$ but, possibly, dependent on $\varepsilon$, such that $\|\zeta^\varepsilon_{3,\kappa}\|_{H^2(\omega)}<C$ for all $0<\kappa\le 2C_K^2$.
\end{theorem}
\begin{proof}
	To begin with, we recall that, according to~\cite{Pie2020-1}, the solution $\bm{\zeta}^\varepsilon$ of Problem~\ref{problem5} is of class $(H^1_0(\omega)\cap H^2(\omega))\times (H^1_0(\omega)\cap H^2(\omega))\times (H^2_0(\omega)\cap H^3(\omega))$. Therefore, we can conclude in the same fashion as Theorem~\ref{th:biharmonic-reg}, and establish that there exists a constant $\hat{C}_\varepsilon>0$ that solely depends on $\varepsilon$, $\omega$, $\|\bm{\zeta}^\varepsilon_\kappa\|_{\bm{L}^2(\omega)}$ and $\|\vec{\xi}^\varepsilon_\kappa\|_{\vec{H}^1_0(\omega)}$ and the given forcing terms for which:
	\begin{equation*}
		\|\bm{\zeta}^\varepsilon_\kappa\|_{\bm{H}^2(\omega)\times \bm{H}^2(\omega)\times \bm{H}^3(\omega)}+\|\vec{\xi}^\varepsilon_\kappa\|_{\vec{H}^2(\omega)}\le \dfrac{\hat{C}_\varepsilon}{\sqrt{\kappa}}.
	\end{equation*}
	
	Additionally, observe that if we test Problem~\ref{problem8} at $((0,0,\eta_3),\vec{0})$, where $\eta_3$ varies arbitrarily in $H^1_0(\omega)$, we obtain
	\begin{equation*}
		\int_{\omega}\nabla \zeta^\varepsilon_{3,\kappa} \cdot\nabla \eta_3 \dd y=-\int_{\omega}\bm{\beta}^\varepsilon(\bm{\zeta}^\varepsilon_\kappa)\cdot(0,0,\eta_3) \dd y-\int_{\omega}(\textup{div }\vec{\xi}^\varepsilon_\kappa)\eta_3\dd y,
	\end{equation*}
	which in turn implies, by the standard augmentation-of-regularity theory (cf., e.g., \cite{Evans2010}) and Theorem~\ref{exunp8}, that $\{\zeta^\varepsilon_{3,\kappa}\}_{\kappa>0}$ is bounded in $H^2(\omega)$ independently of $\kappa$.
\end{proof}

Thanks to Theorem~\ref{th:shallow-reg}, we obtain that the weak convergences established in Theorem~\ref{exunp8} are actually strong.

\begin{theorem}
	\label{th:shallow-strong}
	For each $0<\kappa\le 2C_K^2$, where $C_K$ is the positive constant constructed in Theorem~\ref{KornMixed}, let $(\bm{\zeta}^\varepsilon_\kappa,\vec{\xi}_\kappa^\varepsilon)$ denote the solution  of Problem~\ref{problem8}.
	Then, the following convergences hold up to passing to subsequences
	\begin{align*}
		\bm{\zeta}^\varepsilon_\kappa \to \bm{\zeta}^\varepsilon&, \textup{ in }\bm{H}^1_0(\omega) \textup{ as }\kappa \to 0,\\
		\vec{\xi}^\varepsilon_\kappa \to \nabla \zeta^\varepsilon_3&, \textup{ in }\vec{H}^1_0(\omega) \textup{ as }\kappa \to 0,
	\end{align*}
	where $\bm{\zeta}^\varepsilon\in\bm{U}_K(\omega)$ is the unique solution of Problem~\ref{problem5}.
\end{theorem}
\begin{proof}
	Testing the variational equations of Problem~\ref{problem8} at $(\bm{\zeta}^\varepsilon_\kappa-\bm{\zeta}^\varepsilon,\vec{\xi}^\varepsilon_\kappa-\nabla\zeta^\varepsilon_3)$ gives:
	\begin{equation}
		\label{eq:mixed}
		\begin{aligned}
			&\dfrac{\varepsilon^3}{3}\int_{\omega}\left\{4c_0(\lambda,\mu)(\textup{div }\vec{\xi}^\varepsilon_\kappa)\delta_{\alpha\beta}+4\mu\partial_\alpha\xi_{\beta,\kappa}\right\}(\partial_{\alpha}\xi^\varepsilon_{\beta,\kappa}-\partial_{\alpha\beta}\zeta^\varepsilon_3)\dd y\\
			&\qquad+\int_{\omega}n_{\alpha\beta}^{\theta,\varepsilon}(\vec{\zeta}^\varepsilon_{H,\kappa},\vec{\xi}^\varepsilon_\kappa) e_{\alpha\beta}^{\theta,\varepsilon}(\vec{\zeta}^\varepsilon_{H,\kappa}-\vec{\zeta}^\varepsilon_H,\vec{\xi}^\varepsilon_\kappa-\nabla\zeta^\varepsilon_3) \dd y\\
			&\qquad-\dfrac{\varepsilon^3}{\kappa}\int_{\omega}\{(\bm{\theta}^\varepsilon+\bm{\zeta}^\varepsilon_\kappa)\cdot\bm{q}\}^{-}\left((\bm{\theta}^\varepsilon+\bm{\zeta}^\varepsilon_\kappa)\cdot\bm{q}-(\bm{\theta}^\varepsilon+\bm{\zeta}^\varepsilon)\cdot\bm{q}\right)\dd y\\
			&\qquad+\dfrac{\varepsilon^3}{\kappa}\int_{\omega}|\vec{\xi}^\varepsilon_\kappa-\nabla\zeta^\varepsilon_{3,\kappa}|^2 \dd y\\
			&=\int_{\omega}p^{\alpha,\varepsilon}(\zeta^\varepsilon_{\alpha,\kappa}-\zeta^\varepsilon_{\alpha})\dd y-\int_{\omega}\vec{P}^\varepsilon\cdot(\vec{\xi}^\varepsilon_\kappa-\nabla\zeta^\varepsilon_3)\dd y-\int_{\omega}s^\varepsilon_\alpha(\xi^\varepsilon_{\alpha,\kappa}-\partial_{\alpha}\zeta^\varepsilon_3)\dd y.
		\end{aligned}
	\end{equation}
	
	Following the calculations in~\eqref{KornMixed2} and applying Theorem~\ref{KornMixed} to~\eqref{eq:mixed} gives:
	\begin{equation}
		\label{eq:mixed-2}
		\begin{aligned}
			&\dfrac{4\mu\varepsilon^3}{C_K^2}\left(\|\bm{\zeta}^\varepsilon_{H,\kappa}-\bm{\zeta}^\varepsilon_H\|_{\vec{H}^1_0(\omega)}^2+\|\vec{\xi}^\varepsilon_\kappa-\nabla\zeta^\varepsilon_3\|_{\vec{H}^1_0(\omega)}^2\right)\le \dfrac{4}{3}\mu\varepsilon^3\|\partial_\alpha\xi^\varepsilon_{\beta,\kappa}-\partial_{\alpha\beta}\zeta^\varepsilon_3\|_{L^2(\omega)}^2\\
			&\qquad+4\mu\varepsilon\sum_{\alpha,\beta}\int_{\omega} \left|e_{\alpha\beta}^{\theta,\varepsilon}(\vec{\zeta}^\varepsilon_{H,\kappa}-\vec{\zeta}^\varepsilon_{H},\vec{\xi}^\varepsilon_\kappa-\nabla\zeta^\varepsilon_3)\right|^2 \dd y\\
			&\le \int_{\omega}p^{\alpha,\varepsilon}(\zeta^\varepsilon_{\alpha,\kappa}-\zeta^\varepsilon_{\alpha})\dd y-\int_{\omega}\vec{P}^\varepsilon\cdot(\vec{\xi}^\varepsilon_\kappa-\nabla\zeta^\varepsilon_3)\dd y-\int_{\omega}s^\varepsilon_\alpha(\xi^\varepsilon_{\alpha,\kappa}-\partial_{\alpha}\zeta^\varepsilon_3)\dd y\\
			&\qquad-\dfrac{\varepsilon^3}{3}\int_{\omega}\left\{4c_0(\lambda,\mu)\partial_{\sigma\sigma}\zeta^\varepsilon_3\delta_{\alpha\beta}+4\mu\partial_{\alpha\beta}\zeta^\varepsilon_3\right\}(\partial_{\alpha}\xi_{\beta,\kappa}-\partial_{\alpha\beta}\zeta^\varepsilon_3)\dd y\\
			&\qquad-\int_{\omega}n_{\alpha\beta}^{\theta,\varepsilon}(\vec{\zeta}^\varepsilon_H,\nabla\zeta^\varepsilon_3) e_{\alpha\beta}^{\theta,\varepsilon}(\vec{\zeta}^\varepsilon_{H,\kappa}-\vec{\zeta}^\varepsilon_H,\vec{\xi}^\varepsilon_\kappa-\nabla\zeta^\varepsilon_3) \dd y.
		\end{aligned}
	\end{equation}
	
	Since, thanks to Theorem~\ref{exunp8}, the right-hand side of~\eqref{eq:mixed-2} tends to zero as $\kappa\to 0$, we obtain that the following strong convergences hold:
	\begin{equation*}
		\begin{aligned}
			\vec{\zeta}^\varepsilon_{H,\kappa}\to\vec{\zeta}^\varepsilon_H&,\quad\textup{ in }\vec{H}^1_0(\omega) \textup{ as }\kappa\to 0,\\
			\vec{\xi}^\varepsilon_\kappa\to\nabla\zeta^\varepsilon_3&,\quad\textup{ in }\vec{H}^1_0(\omega) \textup{ as }\kappa\to 0.
		\end{aligned}
	\end{equation*}
	
	Since, thanks to Theorem~\ref{th:shallow-reg}, it results that there exists a constant $C>0$ independent of $\kappa$ such that $\|\zeta^\varepsilon_{3,\kappa}\|_{H^2(\omega)}\le C$ for all $\kappa>0$, up to passing to a subsequence, we obtain that $\zeta^\varepsilon_{3,\kappa}\to\zeta^\varepsilon_3$ in $H^1_0(\omega)$ as $\kappa\to 0$. This completes the proof.
\end{proof}

The problem of approximating the solution of Problem~\ref{problem8} via the Finite Element Method is not as straightforward as the prototypical problem discussed in section~\ref{sec1}. The reason for this is that, in general, the symmetry of the gradient matrix of the \emph{dual variable} cannot be handled directly by means of an \emph{ad hoc} Finite Element. Let us show how the theory of fluid mechanics (cf., e.g., \cite{GR86,Temam2001}) helps us overcome this difficulty.
Define the vector space
\begin{equation*}
\vec{W}:=\{\vec{\varphi} \in \vec{H}^1_0(\omega); \nabla \vec{\varphi} = (\nabla \vec{\varphi})^T \textup{ a.e. in }\omega\},
\end{equation*}
and define the function $\textup{switch}:\vec{H}^1_0(\omega) \to \vec{H}^1_0(\omega)$ by:
\begin{equation*}
\textup{switch }\vec{\varphi}:=(-\varphi_2,\varphi_1),\quad\textup{ for all }\vec{\varphi}=(\varphi_1,\varphi_2) \in \vec{H}^1_0(\omega).
\end{equation*}

The operator $\textup{switch}$ is clearly well-defined, linear, bounded and isometric.
Therefore, the operator $\textup{switch}$ is one-to-one. 
It is also easy to see that the operator $\textup{switch}$ is onto, and we thus observe that the operator $\textup{switch}$ is an isometric isomorphism, whose inverse is once again an isometric isomorphism and is given by:
\begin{equation*}
\textup{switch}^{-1} = -\textup{switch}.
\end{equation*}

Let us observe that, if $\vec{\varphi} \in \vec{W}$, then $\textup{div }\textup{switch }\vec{\varphi}=0$. This leads us to consider the vector space
\begin{equation*}
\vec{\mathcal{W}}:=\{\vec{\psi}\in \vec{H}^1_0(\omega); \textup{div }\vec{\psi}=0 \textup{ a.e. in }\omega\},
\end{equation*}
which is widely used in the treatment of the Stokes equations and of the Navier-Stokes equations (cf., e.g., \cite{GR86} and~\cite{Temam2001}).

For the purpose of approximating the solution of Problem~\ref{problem8}, we need to show that $\vec{W}$ can be identified with $\vec{\mathcal{W}}$. To this aim, it suffices to observe that $\textup{switch}^{-1}:\vec{\mathcal{W}} \to \vec{W}$ is an isometric isomorphism.

Therefore, numerically approximating vector fields with symmetric gradient is equivalent to numerically approximating vector fields with vanishing divergence. The advantage of the latter approach is due to the plethora of available results in the context of the study of the Stokes equations and Navier-Stokes equations. Indeed, according to page~74 and formula~(1.49) in~\cite{Temam2001}, we have that:
\begin{equation*}
\vec{\mathcal{W}} \cong \overrightarrow{\textup{curl}}\,V_3(\omega),
\end{equation*}
where $\overrightarrow{\textup{curl}}\,\phi :=(\partial_2 \phi,-\partial_1\phi)$, for all $\phi$ smooth enough, and we recall that $V_3(\omega)=H^2_0(\omega)$. As a result, we have that for each $\vec{\psi}\in \vec{\mathcal{W}}$, we have that there exists a unique $\phi\in V_3(\omega)$ such that $\vec{\psi}=\overrightarrow{\textup{curl}}\phi$. Moreover, it is straightforward to observe that the stream function $\phi$ is associated with the transverse component of the displacement entering the formulation of the original problem.

Let $h>0$ denote the parameter associated with the mesh size of an affine regular triangulation $\mathcal{T}_h$ of the Lipschitz domain $\omega$. Let $V_{3,h}$ be the finite-dimensional subspace of $V_3(\omega)$ associated with a triangulation made of Hsieh-Clough-Tocher Finite Elements (cf., e.g., Chapter~6 in~\cite{PGCFEM}). 

For each $h>0$, we can thus define the space
\begin{equation*}
\vec{\mathcal{W}}_h := \overrightarrow{\textup{curl}}\, V_{3,h} \subset \vec{\mathcal{W}},
\end{equation*}
as well as the space:
\begin{equation*}
\vec{W}_h := \textup{switch}^{-1}\,\vec{\mathcal{W}}_h = -\textup{switch }\vec{\mathcal{W}}_h \subset \vec{W}.
\end{equation*}

We have that the following density result holds.
\begin{theorem}
	\label{th:4}
	Assume that $\omega$ is triangulated by means of an affine regular triangulation.
	Assume that the space $V_3(\omega)=H^2_0(\omega)$ is discretised by means of a conforming Finite Element.
	Then, the following density result holds:
	$$
	\overline{\bigcup_{h>0} \vec{W}_h}^{\|\cdot\|_{\vec{H}^1_0(\omega)}} = \vec{W}.
	$$
\end{theorem}
\begin{proof}
	Let $\vec{\psi} \in \vec{W}$. We have that there exists a unique $\phi \in V_3(\omega)$ such that $\vec{\psi}=-\textup{switch }\overrightarrow{\textup{curl}}\, \phi=-\nabla\phi$.
	Given any $\epsilon>0$, there exists a function $\phi_h \in V_{3,h}$, for $h>0$ small enough, such that $\|\phi-\phi_h\|_{H^2_0(\omega)}<\epsilon$. Defining $\vec{\psi}_h:= -\textup{switch }\overrightarrow{\textup{curl}}\, \phi_h$, we have that $\vec{\psi}_h \in \vec{W}_h$ and that
	\begin{equation*}
	\|\vec{\psi}-\vec{\psi}_h\|_{\vec{H}^1_0(\omega)}=\left\|\overrightarrow{\textup{curl}}\,(\phi-\phi_h)\right\|_{\vec{H}^1_0(\omega)} \le C \|\phi-\phi_h\|_{H^2_0(\omega)}<C\epsilon,
	\end{equation*}
	where $C>0$ is the continuity constant associated with the $\overrightarrow{\textup{curl}}$ operator. Thanks to the arbitrariness of $\epsilon>0$, the proof is complete.
\end{proof}

An immediate consequence of Theorem~\ref{th:4} is that the elements in $\vec{W}$ can be approximated by means of an \emph{ad hoc} conforming Galerkin method. Note that the previous proof is not in general true if the discretisation of the space $V_3(\omega)$ is performed via non-conforming Finite Elements. 
For each $h>0$, let us denote by $V_h$ a finite-dimensional space of $H^1_0(\omega)$ associated with an affine regular triangulation $\mathcal{T}_h$ made of Courant triangles.
We are thus in a position to write down the discretisation of Problem~\ref{problem8}.

\begin{customprob}{$\mathcal{Q}^{\varepsilon,h}_\kappa(\omega)$}
	\label{problem9}
	Find $(\bm{\zeta}^{\varepsilon,h}_\kappa,\vec{\xi}^{\varepsilon,h}_\kappa) \in \bm{V}_h \times \vec{W}_h$ satisfying:
	\begin{equation*}
	\begin{aligned}
	&\dfrac{\varepsilon^3}{3} \int_{\omega}\left\{4c_0(\lambda,\mu)(\textup{div }\vec{\xi}^{\varepsilon,h}_\kappa) \delta_{\alpha\beta} +4\mu \partial_{\alpha}\xi^{\varepsilon,h}_{\beta,\kappa} \right\} (\partial_{\alpha}\varphi_\beta) \dd y
	+\int_{\omega} n_{\alpha\beta}^{\theta,\varepsilon}(\vec{\zeta}^{\varepsilon,h}_{H,\kappa},\vec{\xi}^{\varepsilon,h}_\kappa) e_{\alpha\beta}^{\theta,\varepsilon}(\vec{\eta}_H,\vec{\varphi}) \dd y\\
	&\qquad+\dfrac{\varepsilon^3}{\kappa}\int_{\omega} (\vec{\xi}^{\varepsilon,h}_\kappa - \nabla \zeta^{\varepsilon,h}_{3,\kappa}) \cdot (\vec{\varphi} - \nabla \eta_3) \dd y+\dfrac{\varepsilon^3}{\kappa}\int_{\omega} \bm{\beta}^\varepsilon(\bm{\zeta}^{\varepsilon,h}_\kappa) \cdot \bm{\eta} \dd y
	= \int_{\omega} p^{\alpha,\varepsilon} \eta_\alpha \dd y -\int_{\omega} \vec{P}^\varepsilon \cdot \vec{\varphi} \dd y- \int_{\omega} s_\alpha^\varepsilon \varphi_\alpha \dd y,
	\end{aligned}
	\end{equation*}
	for all $(\bm{\eta},\vec{\varphi}) \in \bm{V}_h \times \vec{W}_h$.
	\bqed
\end{customprob}

The existence and uniqueness of solutions for Problem~\ref{problem9} can be established in the same fashion as Theorem~\ref{th:3}, as an application of the Minty--Browder Theorem (cf., e.g., Theorem~12.5-3 in~\cite{Ciarlet2025}).
The next result establishes the convergence of the solution $(\bm{\zeta}^{\varepsilon,h}_\kappa,\vec{\xi}^{\varepsilon,h}_\kappa)$ of Problem~\ref{problem9} to the solution $(\bm{\zeta}^\varepsilon_\kappa,\vec{\xi}^\varepsilon_\kappa)$ of Problem~\ref{problem8}.

\begin{theorem}
	\label{th:5}
	Let $0<\kappa\le\min\left\{2,\frac{1}{2\mu}\right\}C_K^2$, where $C_K$ is the positive constant constructed in Theorem~\ref{KornMixed}.
	Let $(\bm{\zeta}^\varepsilon_\kappa,\vec{\xi}^\varepsilon_\kappa) \in (\bm{H}^1_0(\omega)\cap\bm{H}^2(\omega)) \times (\vec{H}^1_0(\omega)\cap\vec{H}^3(\omega))$ be the solution of Problem~\ref{problem8}. Define $c_0(\lambda,\mu)=\lambda\mu/(\lambda+2\mu)$ and $c_1(\lambda,\mu):=16c_0(\lambda,\mu)+16\mu$.
	Let $(\bm{\zeta}^{\varepsilon,h}_\kappa,\vec{\xi}^{\varepsilon,h}_\kappa) \in \bm{V}_h \times \vec{W}_h$ be the solution of Problem~\ref{problem9}.
	Then, there exists a constant $C>0$ depending on $\theta$ and $\omega$ such that:
	\begin{equation*}
		\|\bm{\zeta}^\varepsilon_\kappa-\bm{\zeta}^{\varepsilon,h}_\kappa\|_{\bm{H}^1_0(\omega)}^2+\|\vec{\xi}^\varepsilon_\kappa-\xi^{\varepsilon,h}_\kappa\|_{\vec{H}^1_0(\omega)}^2\le\dfrac{2C_K^2(1+c_P^2)}{\mu\varepsilon^2\kappa}\left[2c_0(\lambda,\mu)+c_1(\lambda,\mu)C+2+\dfrac{2\mu}{3}\right]C h^2\left(|\bm{\zeta}^\varepsilon_\kappa|_{\bm{H}^2(\omega)}^2+ h^2|\vec{\xi}^\varepsilon_\kappa|_{\vec{H}^3(\omega)}^2\right).
	\end{equation*}
\end{theorem}
\begin{proof}
In what follows, we denote by
\begin{equation*}
\overrightarrow{\tilde{\Pi}}_h\vec{\xi}^\varepsilon_\kappa:=-\textup{switch }\overrightarrow{\textup{curl}}\, \Pi_h \phi^\varepsilon_\kappa = -\nabla\left(\Pi_h \phi^\varepsilon_\kappa\right) \in \vec{W}_h,
\end{equation*}
where $\Pi_h$ is the standard interpolation operator. Denote by $\vec{\Pi}_h$ and $\bm{\Pi}_h$ the \emph{two-dimensional} and \emph{three-dimensional} analogues of $\Pi_h$.
In light of the variational equations of Problem~\ref{problem8}, Problem~\ref{problem9}, Young's inequality~\cite{Young1912} and H\"older's inequality, we obtain:
\begin{align*}
&\dfrac{4c_0(\lambda,\mu)\varepsilon^3}{3}\|\textup{div }(\vec{\xi}^\varepsilon_\kappa-\vec{\xi}^{\varepsilon,h}_\kappa)\|_{L^2(\omega)}^2
+\dfrac{4\mu \varepsilon^3}{3}\|\nabla\vec{\xi}^\varepsilon_\kappa - \nabla\vec{\xi}^{\varepsilon,h}_\kappa\|_{\mathbb{L}^2(\omega)}^2
+4c_0(\lambda,\mu)\varepsilon\|e^{\theta,\varepsilon}_{\sigma\sigma}(\vec{\zeta}^\varepsilon_{H,\kappa}-\vec{\zeta}^{\varepsilon,h}_{H,\kappa},\vec{\xi}^\varepsilon_\kappa -\vec{\xi}^{\varepsilon,h}_\kappa)\|_{L^2(\omega)}^2\\
&\qquad+4\mu\varepsilon\sum_{\alpha,\beta}\|e_{\alpha\beta}^{\theta,\varepsilon}(\vec{\zeta}^\varepsilon_{H,\kappa}-\vec{\zeta}^{\varepsilon,h}_{H,\kappa},\vec{\xi}^\varepsilon_\kappa -\vec{\xi}^{\varepsilon,h}_\kappa)\|_{L^2(\omega)}^2
+\dfrac{\varepsilon^3}{\kappa}\|(\vec{\xi}^\varepsilon_\kappa-\vec{\xi}^{\varepsilon,h}_\kappa)-\nabla(\zeta^\varepsilon_{3,\kappa}-\zeta^{\varepsilon,h}_{3,\kappa})\|_{\vec{L}^2(\omega)}^2
\\
&\qquad+\dfrac{\varepsilon^3}{\kappa}\|\bm{\beta}^\varepsilon(\bm{\zeta}^\varepsilon_\kappa)-\bm{\beta}^\varepsilon(\bm{\zeta}^{\varepsilon,h}_\kappa)\|_{\bm{L}^2(\omega)}^2
\\
&= \dfrac{4c_0(\lambda,\mu)\varepsilon^3}{3} \int_{\omega} \textup{div }(\vec{\xi}^\varepsilon_\kappa-\vec{\xi}^{\varepsilon,h}_\kappa) \delta_{\alpha\beta} \partial_{\alpha}(\xi^\varepsilon_{\beta,\kappa}-\tilde{\Pi}_h\xi^\varepsilon_{\beta,\kappa}) \dd y+\dfrac{4\mu\varepsilon^3}{3} \int_{\omega}\partial_{\alpha}(\xi^\varepsilon_{\beta,\kappa}-\xi^{\varepsilon,h}_{\beta,\kappa}) \partial_{\alpha}(\xi^\varepsilon_{\beta,\kappa}-\tilde{\Pi}_h \xi^\varepsilon_{\beta,\kappa}) \dd y\\
&\qquad+4c_0(\lambda,\mu)\varepsilon \int_{\omega} e^{\theta,\varepsilon}_{\sigma\sigma}(\vec{\zeta}^\varepsilon_{H,\kappa}-\vec{\zeta}^{\varepsilon,h}_{H,\kappa},\vec{\xi}^\varepsilon_\kappa-\vec{\xi}^{\varepsilon,h}_\kappa) e^{\theta,\varepsilon}_{\tau\tau}\left(\vec{\zeta}^\varepsilon_{H,\kappa}-\vec{\Pi}_h\vec{\zeta}^{\varepsilon,h}_{H,\kappa},\vec{\xi}^\varepsilon_\kappa-\overrightarrow{\tilde{\Pi}}_h\xi^\varepsilon_\kappa\right) \dd y\\
&\qquad+4\mu\varepsilon\int_{\omega} e^{\theta,\varepsilon}_{\alpha\beta}(\vec{\zeta}^\varepsilon_{H,\kappa}-\vec{\zeta}^{\varepsilon,h}_{H,\kappa},\vec{\xi}^\varepsilon_\kappa-\vec{\xi}^{\varepsilon,h}_\kappa) e^{\theta,\varepsilon}_{\alpha\beta}\left(\vec{\zeta}^\varepsilon_{H,\kappa}-\vec{\Pi}_h\vec{\zeta}^{\varepsilon,h}_{H,\kappa},\vec{\xi}^\varepsilon_\kappa-\overrightarrow{\tilde{\Pi}}_h\xi^\varepsilon_\kappa\right) \dd y\\
&\qquad+\dfrac{\varepsilon^3}{\kappa}\int_{\omega}\left[(\vec{\xi}^\varepsilon_\kappa-\vec{\xi}^{\varepsilon,h}_\kappa) -\nabla(\zeta^\varepsilon_{3,\kappa}- \zeta^{\varepsilon,h}_{3,\kappa})\right] \cdot \left[(\vec{\xi}^\varepsilon_\kappa-\overrightarrow{\tilde{\Pi}}_h\vec{\xi}^\varepsilon_\kappa) -\nabla(\zeta^\varepsilon_{3,\kappa}-\Pi_h\zeta^\varepsilon_{3,\kappa})\right] \dd y\\
&\qquad+\dfrac{\varepsilon^3}{\kappa} \int_{\omega} \left[\bm{\beta}^\varepsilon(\bm{\zeta}^\varepsilon_\kappa) - \bm{\beta}^\varepsilon(\bm{\zeta}^{\varepsilon,h}_\kappa)\right] \cdot \left(\bm{\zeta}^\varepsilon_\kappa - \bm{\Pi}_h \bm{\zeta}^\varepsilon_\kappa\right) \dd y\\
&\le\dfrac{2c_0(\lambda,\mu)\varepsilon^3}{3}\|\textup{div }(\vec{\xi}^\varepsilon_\kappa-\vec{\xi}^{\varepsilon,h}_\kappa)\|_{L^2(\omega)}^2
+\dfrac{2c_0(\lambda,\mu)\varepsilon^3}{3}\|\textup{div }(\vec{\xi}^\varepsilon_\kappa-\overrightarrow{\tilde{\Pi}}_h\vec{\xi}^\varepsilon_\kappa)\|_{L^2(\omega)}^2+\dfrac{2\mu\varepsilon^3}{3}\|\nabla\vec{\xi}^\varepsilon_\kappa - \nabla\vec{\xi}^{\varepsilon,h}_\kappa\|_{\mathbb{L}^2(\omega)}^2\\
&\qquad+\dfrac{2\mu\varepsilon^3}{3}\left\|\nabla\vec{\xi}^\varepsilon_\kappa-\nabla\left(\overrightarrow{\tilde{\Pi}}_h \vec{\xi}^\varepsilon_\kappa\right)\right\|_{\mathbb{L}^2(\omega)}^2
+2c_0(\lambda,\mu)\varepsilon\|e^{\theta,\varepsilon}_{\sigma\sigma}(\vec{\zeta}^\varepsilon_{H,\kappa}-\vec{\zeta}^{\varepsilon,h}_{H,\kappa},\vec{\xi}^\varepsilon_\kappa-\vec{\xi}^{\varepsilon,h}_\kappa)\|_{L^2(\omega)}^2\\
&\qquad+2c_0(\lambda,\mu)\varepsilon\|e^{\theta,\varepsilon}_{\tau\tau}(\vec{\zeta}^\varepsilon_{H,\kappa}-\vec{\Pi}_h\vec{\zeta}^{\varepsilon}_{H,\kappa},\vec{\xi}^\varepsilon_\kappa-\overrightarrow{\tilde{\Pi}}_h\vec{\xi}^\varepsilon_\kappa)\|_{L^2(\omega)}^2+2\mu\varepsilon\sum_{\alpha,\beta}\|e^{\theta,\varepsilon}_{\alpha\beta}(\vec{\zeta}^\varepsilon_{H,\kappa}-\vec{\zeta}^{\varepsilon,h}_{H,\kappa},\vec{\xi}^\varepsilon_\kappa-\vec{\xi}^{\varepsilon,h}_\kappa)\|_{L^2(\omega)}^2\\
&\qquad+2\mu\varepsilon\sum_{\alpha,\beta}\|e^{\theta,\varepsilon}_{\alpha\beta}(\vec{\zeta}^\varepsilon_{H,\kappa}-\vec{\Pi}_h\vec{\zeta}^{\varepsilon}_{H,\kappa},\vec{\xi}^\varepsilon_\kappa-\overrightarrow{\tilde{\Pi}}_h\vec{\xi}^\varepsilon_\kappa)\|_{L^2(\omega)}^2+\dfrac{\varepsilon^3}{2\kappa}\left\|(\vec{\xi}^\varepsilon_\kappa-\vec{\xi}^{\varepsilon,h}_\kappa) -\nabla(\zeta^\varepsilon_{3,\kappa}-\zeta^{\varepsilon,h}_{3,\kappa})\right\|_{\vec{L}^2(\omega)}^2\\
&\qquad+\dfrac{\varepsilon^3}{2\kappa}\left\|(\vec{\xi}^\varepsilon_\kappa-\overrightarrow{\tilde{\Pi}}_h\vec{\xi}^\varepsilon_\kappa) -\nabla(\zeta^\varepsilon_{3,\kappa}-\Pi_h \zeta^\varepsilon_{3,\kappa})\right\|_{\vec{L}^2(\omega)}^2
+\dfrac{\varepsilon^3}{2\kappa}\|\bm{\beta}^\varepsilon(\bm{\zeta}^\varepsilon_\kappa)-\bm{\beta}^\varepsilon(\bm{\zeta}^{\varepsilon,h}_\kappa)\|_{\bm{L}^2(\omega)}^2
+\dfrac{\varepsilon^3}{2\kappa}\|\bm{\zeta}^\varepsilon_\kappa-\bm{\Pi}_h\bm{\zeta}^\varepsilon_\kappa\|_{\bm{L}^2(\omega)}^2.
\end{align*}

Thanks to Korn's inequality in \emph{mixed coordinates} (Theorem~\ref{KornMixed}), the classical Poincar\'e--Friedrichs inequality (cf., e.g., Theorem~8.3-3 in~\cite{Ciarlet2025}), and Young's inequality~\cite{Young1912}, we obtain:
\begin{align*}
&\dfrac{\mu\varepsilon^3}{2C_K^2(1+c_P^2)}\left\{\|\bm{\zeta}^\varepsilon_\kappa-\bm{\zeta}^{\varepsilon,h}_\kappa\|_{\bm{H}^1_0(\omega)}^2+\|\vec{\xi}^\varepsilon_\kappa-\xi^{\varepsilon,h}_\kappa\|_{\vec{H}^1_0(\omega)}^2\right\}\\
&\le\dfrac{2\mu\varepsilon^3}{C_K^2}\left\{\|\vec{\zeta}^\varepsilon_{H,\kappa}-\vec{\zeta}^{\varepsilon,h}_{H,\kappa}\|_{\vec{H}^1_0(\omega)}^2+\dfrac{\|\vec{\xi}^\varepsilon_\kappa-\xi^{\varepsilon,h}_\kappa\|_{\vec{H}^1_0(\omega)}^2}{2}\right\}+\dfrac{\mu\varepsilon^3}{2C_K^2}\|\nabla\zeta^\varepsilon_{3,\kappa}-\nabla^{\varepsilon,h}_{3,\kappa}\|_{\vec{L}^2(\omega)}^2\\
&\le\dfrac{2c_0(\lambda,\mu)\varepsilon^3}{3}\|\textup{div }(\vec{\xi}^\varepsilon_\kappa-\vec{\xi}^{\varepsilon,h}_\kappa)\|_{L^2(\omega)}^2
+\dfrac{2\mu \varepsilon^3}{3}\|\nabla\vec{\xi}^\varepsilon_\kappa - \nabla\vec{\xi}^{\varepsilon,h}_\kappa\|_{\mathbb{L}^2(\omega)}^2
+2c_0(\lambda,\mu)\varepsilon\|e^{\theta,\varepsilon}_{\sigma\sigma}(\vec{\zeta}^\varepsilon_{H,\kappa}-\vec{\zeta}^{\varepsilon,h}_{H,\kappa},\vec{\xi}^\varepsilon_\kappa -\vec{\xi}^{\varepsilon,h}_\kappa)\|_{L^2(\omega)}^2\\
&\qquad+2\mu\varepsilon\sum_{\alpha,\beta}\|e_{\alpha\beta}^{\theta,\varepsilon}(\vec{\zeta}^\varepsilon_{H,\kappa}-\vec{\zeta}^{\varepsilon,h}_{H,\kappa},\vec{\xi}^\varepsilon_\kappa -\vec{\xi}^{\varepsilon,h}_\kappa)\|_{L^2(\omega)}^2
+\dfrac{\varepsilon^3}{2\kappa}\|(\vec{\xi}^\varepsilon_\kappa-\vec{\xi}^{\varepsilon,h}_\kappa)-\nabla(\zeta^\varepsilon_{3,\kappa}-\zeta^{\varepsilon,h}_{3,\kappa})\|_{\vec{L}^2(\omega)}^2
\\
&\qquad+\dfrac{\varepsilon^3}{2\kappa}\|\bm{\beta}^\varepsilon(\bm{\zeta}^\varepsilon_\kappa)-\bm{\beta}^\varepsilon(\bm{\zeta}^{\varepsilon,h}_\kappa)\|_{\bm{L}^2(\omega)}^2\\
&= \dfrac{2c_0(\lambda,\mu)\varepsilon^3}{3}\|\textup{div }(\vec{\xi}^\varepsilon_\kappa-\overrightarrow{\tilde{\Pi}}_h\vec{\xi}^\varepsilon_\kappa)\|_{L^2(\omega)}^2
+\dfrac{2\mu\varepsilon^3}{3}\left\|\nabla\vec{\xi}^\varepsilon_\kappa-\nabla\left(\overrightarrow{\tilde{\Pi}}_h \vec{\xi}^\varepsilon_\kappa\right)\right\|_{\mathbb{L}^2(\omega)}^2\\
&\qquad+2c_0(\lambda,\mu)\varepsilon\|e^{\theta,\varepsilon}_{\tau\tau}(\vec{\zeta}^\varepsilon_{H,\kappa}-\vec{\Pi}_h\vec{\zeta}^{\varepsilon}_{H,\kappa},\vec{\xi}^\varepsilon_\kappa-\overrightarrow{\tilde{\Pi}}_h\vec{\xi}^\varepsilon_\kappa)\|_{L^2(\omega)}^2\\
&\qquad+2\mu\varepsilon\sum_{\alpha,\beta}\|e^{\theta,\varepsilon}_{\alpha\beta}(\vec{\zeta}^\varepsilon_{H,\kappa}-\vec{\Pi}_h\vec{\zeta}^{\varepsilon}_{H,\kappa},\vec{\xi}^\varepsilon_\kappa-\overrightarrow{\tilde{\Pi}}_h\vec{\xi}^\varepsilon_\kappa)\|_{L^2(\omega)}^2
+\dfrac{\varepsilon^3}{2\kappa}\left\|(\vec{\xi}^\varepsilon_\kappa-\overrightarrow{\tilde{\Pi}}_h\vec{\xi}^\varepsilon_\kappa) -\nabla(\zeta^\varepsilon_{3,\kappa}-\Pi_h \zeta^\varepsilon_{3,\kappa})\right\|_{\vec{L}^2(\omega)}^2\\
&\qquad+\dfrac{\varepsilon^3}{2\kappa}\|\bm{\zeta}^\varepsilon_\kappa-\bm{\Pi}_h\bm{\zeta}^\varepsilon_\kappa\|_{\bm{L}^2(\omega)}^2\\
&\le2\mu\varepsilon^3\left(\dfrac{\lambda}{\lambda+2\mu}+\dfrac{1}{3}\right)\|\vec{\xi}^\varepsilon_\kappa-\overrightarrow{\tilde{\Pi}}_h\vec{\xi}^\varepsilon_\kappa\|_{\vec{H}^1_0(\omega)}^2
+c_1(\lambda,\mu)C\varepsilon\left(\|\vec{\zeta}^\varepsilon_{H,\kappa}-\vec{\Pi}_h\vec{\zeta}^\varepsilon_{H,\kappa}\|_{\vec{H}^1_0(\omega)}^2+\|\vec{\xi}^\varepsilon_\kappa-\overrightarrow{\tilde{\Pi}}_h\vec{\xi}^\varepsilon_\kappa\|_{\vec{H}^1_0(\omega)}^2\right)\\
&\qquad+\dfrac{\varepsilon^3}{\kappa}\|\zeta^\varepsilon_{3,\kappa}-\Pi_h\zeta^\varepsilon_{3,\kappa}\|_{H^1_0(\omega)}^2+\dfrac{\varepsilon^3}{\kappa}\|\vec{\xi}^\varepsilon_\kappa-\overrightarrow{\tilde{\Pi}}_h\vec{\xi}^\varepsilon_\kappa\|_{\vec{H}^1_0(\omega)}^2
+\dfrac{\varepsilon^3}{2\kappa}\|\bm{\zeta}^\varepsilon_\kappa-\bm{\Pi}_h\bm{\zeta}^\varepsilon_\kappa\|_{\bm{L}^2(\omega)}^2\\
&\le\left[c_1(\lambda,\mu)C+3\right]\dfrac{\varepsilon}{\kappa}\|\bm{\zeta}^\varepsilon_\kappa-\bm{\Pi}_h\bm{\zeta}^\varepsilon_\kappa\|_{\bm{H}^1_0(\omega)}^2
+\left[2c_0(\lambda,\mu)+c_1(\lambda,\mu)C+1+\dfrac{2\mu}{3}\right]\dfrac{\varepsilon}{\kappa}\|\vec{\xi}^\varepsilon_\kappa-\overrightarrow{\tilde{\Pi}}_h\vec{\xi}^\varepsilon_\kappa\|_{\vec{H}^1_0(\omega)}^2\\
&\le\left[2c_0(\lambda,\mu)+c_1(\lambda,\mu)C+2+\dfrac{2\mu}{3}\right]\dfrac{\varepsilon}{\kappa}\left(\|\bm{\zeta}^\varepsilon_\kappa-\bm{\Pi}_h\bm{\zeta}^\varepsilon_\kappa\|_{\bm{H}^1_0(\omega)}^2+\|\vec{\xi}^\varepsilon_\kappa-\overrightarrow{\tilde{\Pi}}_h\vec{\xi}^\varepsilon_\kappa\|_{\vec{H}^1_0(\omega)}^2\right)\\
&\le\left[2c_0(\lambda,\mu)+c_1(\lambda,\mu)C+2+\dfrac{2\mu}{3}\right]\dfrac{\varepsilon}{\kappa}\left(C h^2 |\bm{\zeta}^\varepsilon_\kappa|_{\bm{H}^2(\omega)}^2+\left\|\overrightarrow{\textup{curl}}(\phi^\varepsilon_\kappa-\Pi_h\phi^\varepsilon_\kappa)\right\|_{\vec{H}^1_0(\omega)}^2\right)\\
&\le\left[2c_0(\lambda,\mu)+c_1(\lambda,\mu)C+2+\dfrac{2\mu}{3}\right]\dfrac{\varepsilon}{\kappa}\left(C h^2 |\bm{\zeta}^\varepsilon_\kappa|_{\bm{H}^2(\omega)}^2+\left\|\phi^\varepsilon_\kappa-\Pi_h\phi^\varepsilon_\kappa\right\|_{H^2(\omega)}^2\right)\\
&\le\left[2c_0(\lambda,\mu)+c_1(\lambda,\mu)C+2+\dfrac{2\mu}{3}\right]\dfrac{\varepsilon}{\kappa}\left(C h^2 |\bm{\zeta}^\varepsilon_\kappa|_{\bm{H}^2(\omega)}^2+C h^4|\phi^\varepsilon_\kappa|_{H^4(\omega)}^2\right)\\
&=\left[2c_0(\lambda,\mu)+c_1(\lambda,\mu)C+2+\dfrac{2\mu}{3}\right]\dfrac{\varepsilon}{\kappa}\left(C h^2 |\bm{\zeta}^\varepsilon_\kappa|_{\bm{H}^2(\omega)}^2+C h^4|\vec{\xi}^\varepsilon_\kappa|_{\vec{H}^3(\omega)}^2\right),
\end{align*}
where $C>0$ is a positive constant which only depends on $\theta$ and $\omega$, the fourth-last inequality descends from an application of Theorem~3.1-6 in~\cite{PGCFEM} and the fact that the operator $\textup{switch}$ is an isometric isomorphism, the third-last inequality is due to the continuity of the $\overrightarrow{\textup{curl}}$ operator, the second-last inequality follows from the remarks preceding Theorem~\ref{th:5} and an application of Theorem~6.1.3 in~\cite{PGCFEM} with $m=2$. 
The last equality descends from the fact that $|\phi^\varepsilon_\kappa|_{H^4(\omega)}=\left|\overrightarrow{\textup{curl}}\,\phi^\varepsilon_\kappa\right|_{\vec{H}^3(\omega)}$ and the fact that the $\textup{switch}$ operator is an isometric isomorphism.
\end{proof}

One of the main difficulties we encountered for showing that the solution of Problem~\ref{problem9} converges to the solution of Problem~\ref{problem8} is that it is not possible to define the interpolation operator for vector fields in $\vec{H}^1_0(\omega)$ with symmetric gradient matrix. However, thanks to the isomorphisms associated with the $\textup{switch}$ and $\overrightarrow{\textup{curl}}$ operators and the theory presented in~\cite{GR86,Temam2001}, we can consider the interpolation operator for the space $H^2_0(\omega)$, which can be approximated by means of the Finite Element space $V_{3,h}$ associated with Hsieh-Clough-Tocher triangles. We can then transform the interpolation function defined over the space $V_{3,h}$ into a function of $\vec{W}_h$.
This is, in fact, the function that we exploit to derive the estimates that put in a position to apply Cea's Lemma (cf., e.g., Theorem~2.4.1 in~\cite{PGCFEM}).

Another issue that appears in the derivation of the estimates in Theorem~\ref{th:5} is the regularity requirement for the stream function $\phi^\varepsilon_\kappa$, which has to be \emph{at least} of class $H^4(\omega)$.
Note that this kind of augmentation of regularity is totally licit for the penalised problem, even though it is not licit, in general, for the transverse component of the solution of the variational inequalities in Problem~\ref{problem5} to enjoy a regularity higher than $H^3(\omega)$ in light of the results in~\cite{Caffarelli1979,CafFriTor1982}.

\section{Numerical experiments for the biharmonic obstacle problem}
\label{sec1-bis}

In this last section of the paper, we implement numerical simulations aiming to validate the theoretical results presented in section~\ref{sec1}.
Consider as a domain a circle of radius $r_A:=1.0$, and denote one such domain by $\omega$:
\begin{equation*}
	\omega:=\left\{y=(y_\alpha)\in \mathbb{R}^2;\sqrt{y_1^2+y_2^2}<r_A\right\}.
\end{equation*}

We constrain solution $u$ of Problem~\ref{problem1} to remain confined above the function $\theta$ defined by:

\begin{equation*}
	\theta(y):=-1.0,\quad\textup{ for all } y\in\overline{\omega}.
\end{equation*}

The numerical simulations are performed by means of the software FEniCSx~\cite{Fenics2016} version~0.9.0, and the visualization is performed by means of the software ParaView~\cite{Ahrens2005}. The plots were created by means of the \verb*|matplotlib| libraries from a Python~3.9.8 installation.

The first batch of numerical experiments is meant to validate the claim of Theorem~\ref{th:biharmonic-strong}.  After fixing the mesh size $0<h<<1$, we let $\kappa$ tend to zero in Problem~\ref{problem4}. Let $(u_{\kappa_1}^h,\vec{\xi}_{\kappa_1}^h)$ and $(u_{\kappa_2}^h,\vec{\xi}_{\kappa_2}^h)$ be two solutions of Problem~\ref{problem4}, where $\kappa_2:=2\kappa_1$ and $\kappa_1>0$.
The solution of Problem~\ref{problem3} is discretised by Courant triangles (cf., e.g., \cite{PGCFEM}) and homogeneous Dirichlet boundary conditions are imposed for all the components.
At each iteration, Problem~\ref{problem4} is solved by Newton's method\footnote{The Newton method of the FEniCSx package can properly handle \emph{mild} non-smoothness like, for instance, those with bounded derivative.}. The algorithm stops when the error residual with respect to the standard norm of $H^1_0(\omega)$ is smaller than $1.0 \times 10^{-8}$. We observe that the error remains bounded below $C\sqrt{\kappa}$, where $C=0.3$. The forcing term $f$ entering the first experiment is given by:
$$
f(y):=
\begin{cases}
	-(-7.5 y_1^2-7.5 y_2^2+0.295) &,\textup{ if } |y|^2< 0.060,\\
	0&,\textup{otherwise}.
\end{cases}
$$

The results for the first batch of experiments is presented below.

\begin{figure}[H]
	\centering
	\subfloat[Error convergence.]{
		\begin{tabular}{lcr}
\toprule
$\kappa$ && Error\\
\midrule
5.96E-9&&2.65E-6\\
2.98E-9&&1.33E-6\\
1.49E-9&&6.63E-7\\
7.45E-10&&3.31E-7\\
3.73E-10&&1.66E-7\\
1.86E-10&&8.28E-8\\
9.31E-11&&4.14E-8\\
4.66E-11&&2.07E-8\\
2.33E-11&&1.04E-8\\
1.16E-11&&5.18E-9\\
\bottomrule
\end{tabular}

	}
	\hspace{1.5cm}
	\begin{subfigure}[t]{0.52\linewidth}
		\includegraphics[width=1.0\linewidth]{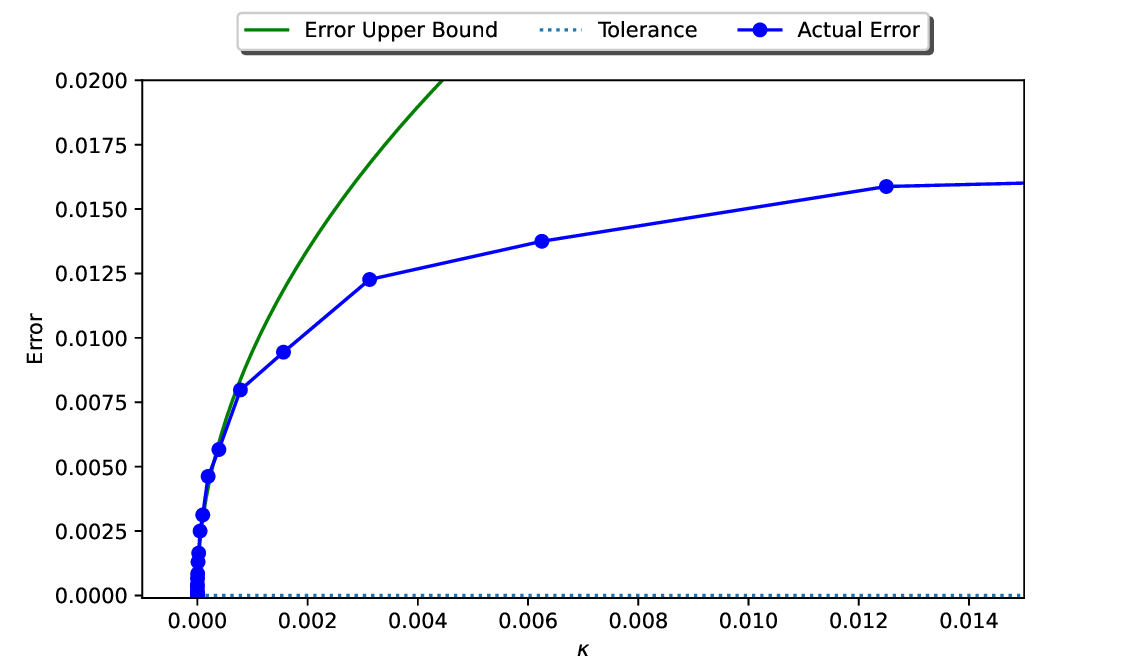}
		\subcaption{Error convergence as $\kappa \to 0$ for $h=0.0625$.}
	\end{subfigure}%
\end{figure}

\begin{figure}[H]
	\ContinuedFloat
	\centering
	\subfloat[Error convergence.]{
		\begin{tabular}{lcr}
\toprule
$\kappa$ && Error\\
\midrule
2.98E-9&&5.02E-6\\
1.49E-9&&2.51E-6\\
7.45E-10&&1.25E-6\\
3.73E-10&&6.27E-7\\
1.86E-10&&3.14E-7\\
9.31E-11&&1.57E-7\\
4.66E-11&&7.84E-8\\
2.33E-11&&3.92E-8\\
1.16E-11&&1.96E-8\\
5.82E-12&&9.80E-9\\
\bottomrule
\end{tabular}

	}
	\hspace{1.5cm}
	\begin{subfigure}[t]{0.52\linewidth}
		\includegraphics[width=1.0\linewidth]{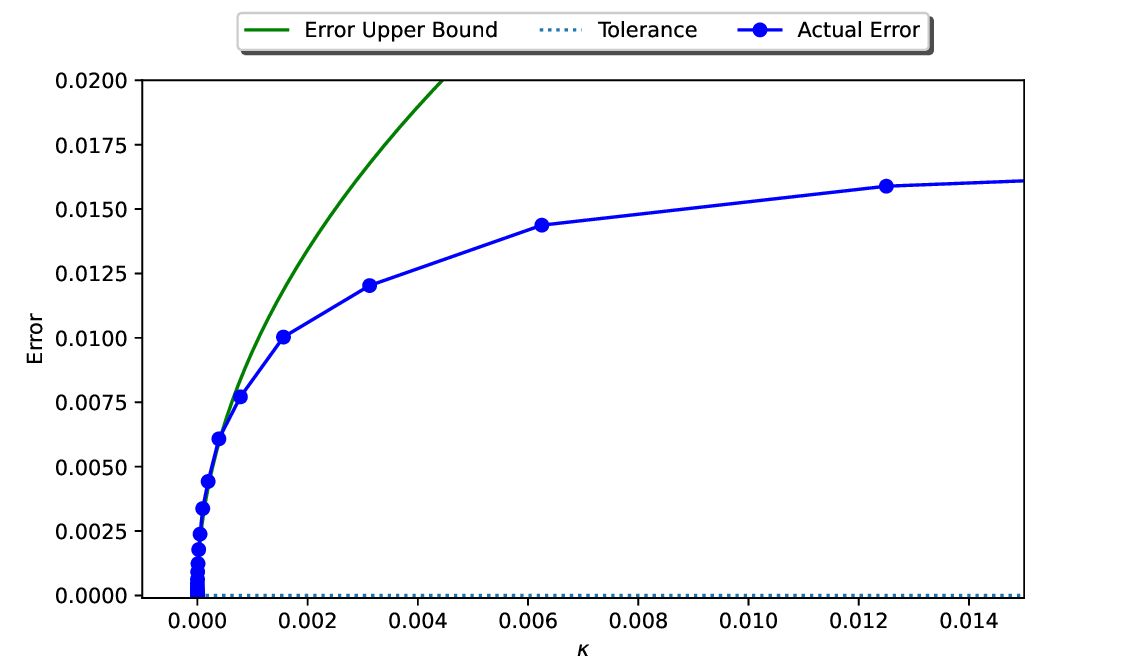}
		\subcaption{Error convergence as $\kappa \to 0$ for $h=0.03125$.}
	\end{subfigure}%
\end{figure}

\begin{figure}[H]
	\ContinuedFloat
	\centering
	\subfloat[Error convergence.]{
		\begin{tabular}{lcr}
\toprule
$\kappa$ && Error\\
\midrule
1.86E-10&&4.96E-6\\
9.31E-11&&2.48E-6\\
4.66E-11&&1.24E-6\\
2.33E-11&&6.19E-7\\
1.16E-11&&3.10E-7\\
5.82E-12&&1.55E-7\\
2.91E-12&&7.74E-8\\
1.46E-12&&3.87E-8\\
7.28E-13&&1.94E-8\\
3.64E-13&&9.68E-9\\
\bottomrule
\end{tabular}

	}
	\hspace{1.5cm}
	\begin{subfigure}[t]{0.52\linewidth}
		\includegraphics[width=1.0\linewidth]{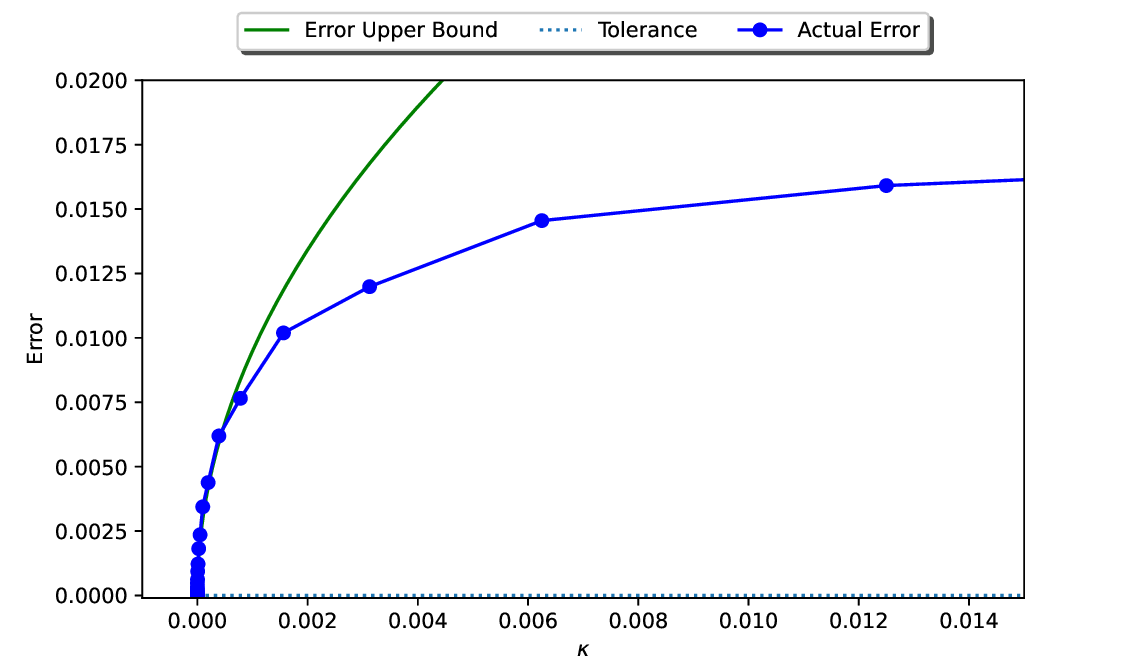}
		\subcaption{Error convergence as $\kappa \to 0$ for $h=0.0078125$.}
	\end{subfigure}%
\end{figure}

\begin{figure}[H]
	\ContinuedFloat
	\centering
	\subfloat[Error convergence.]{
		\begin{tabular}{lcr}
\toprule
$\kappa$ && Error\\
\midrule
1.16E-11&&1.22E-6\\
5.82E-12&&9.50E-7\\
2.91E-12&&6.12E-7\\
1.46E-12&&4.75E-7\\
7.28E-13&&3.06E-7\\
3.64E-13&&1.53E-7\\
1.82E-13&&7.65E-8\\
9.09E-14&&3.82E-8\\
4.55E-14&&1.91E-8\\
2.27E-14&&9.56E-9\\
\bottomrule
\end{tabular}

	}
	\hspace{1.5cm}
	\begin{subfigure}[t]{0.52\linewidth}
		\includegraphics[width=1.0\linewidth]{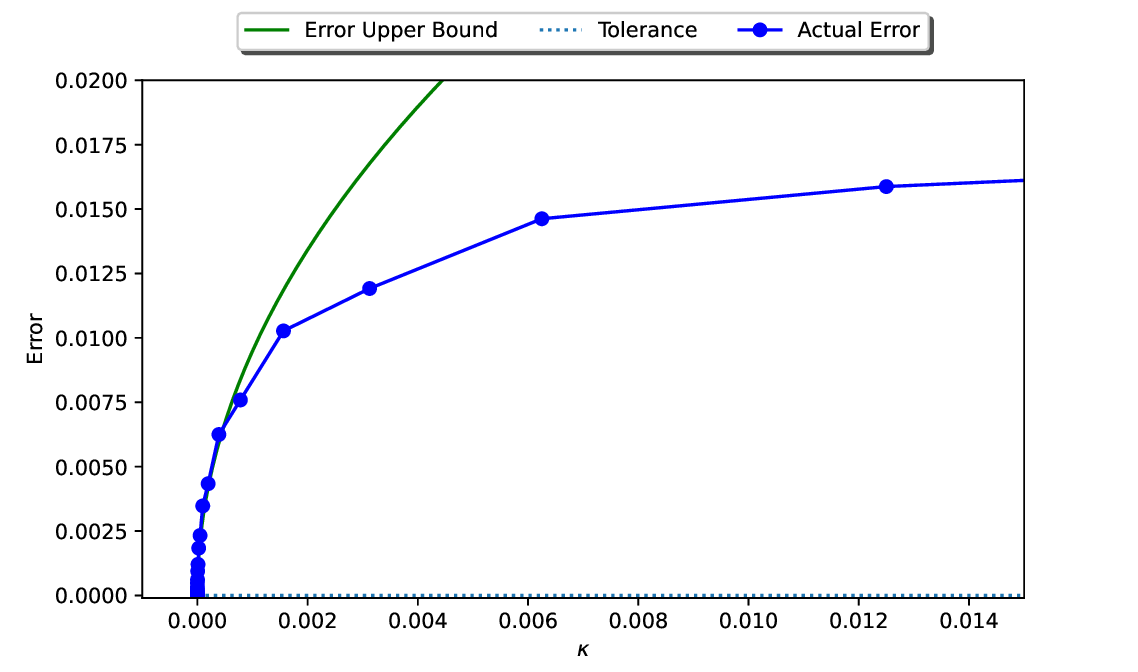}
		\subcaption{Error convergence as $\kappa \to 0$ for $h=0.001953125$.}
	\end{subfigure}%
	\caption{According to Theorem~\ref{th:biharmonic-strong}, the first component of the solution $(u_\kappa,\vec{\xi}_\kappa)$ of Problem~\ref{problem3} converges with respect to the standard norm of $H^1_0(\omega)$ as $\kappa\to 0$. In this numerical experiment, we verify that , given $0<h<<1$, $\|u^h_{\kappa_1}-u^h_{\kappa_2}\|_{H^1_0(\omega)}\to 0$ as $\kappa_1\to 0$, where $\kappa_2=2\kappa_1$}.
	\label{fig:1bis}
\end{figure}

From the data patterns in Figure~\ref{fig:1bis} we observe that, for a given mesh size $h$, the solution $(u_\kappa^h,\vec{\xi}_\kappa^h)$ of Problem~\ref{problem4} converges as $\kappa \to 0$. This is coherent with the conclusion of Theorem~\ref{th:biharmonic-strong}.

The second batch of numerical experiments verifies that the numerical method here proposed is not affected by the locking phenomenon.
In order to show that the numerical method is not affected by locking, we manufacture a solution $u_{\textup{exact}}\in H^2_0(\omega)$ for Problem~\ref{problem1} and we show that, for a fixed $0< q <1$ as in Theorem~\ref{th:3}, the solution $u_\kappa^h$ of Problem~\ref{problem4} converges to $u_{\textup{exact}}$ with respect to the standard $H^1_0(\omega)$ norm. The expression for the function $u_{\textup{exact}}$ is given by:
\begin{equation*}
	u_{\textup{exact}}(y):=
	\begin{cases}
		-1&, \textup{ if } \sqrt{y_1^2+y_2^2}<\dfrac{1}{2},\\
		\\
		4-24\sqrt{y_1^2+y_2^2}+36(y_1^2+y_2^2)-16(y_1^2+y_2^2)^{3/2}&, \textup{ if } \dfrac{1}{2}\le\sqrt{y_1^2+y_2^2}<1.
	\end{cases}
\end{equation*}

Note that the function $u_{\textup{exact}}$ is radial. For sake of clarity, we sketch the graph of one of the profiles, denoted by $\phi(r)$, and of its derivative $\phi'(r)$.

\begin{figure}[H]
	\centering
	\begin{subfigure}[t]{0.30\linewidth}
		\includegraphics[width=1.0\linewidth]{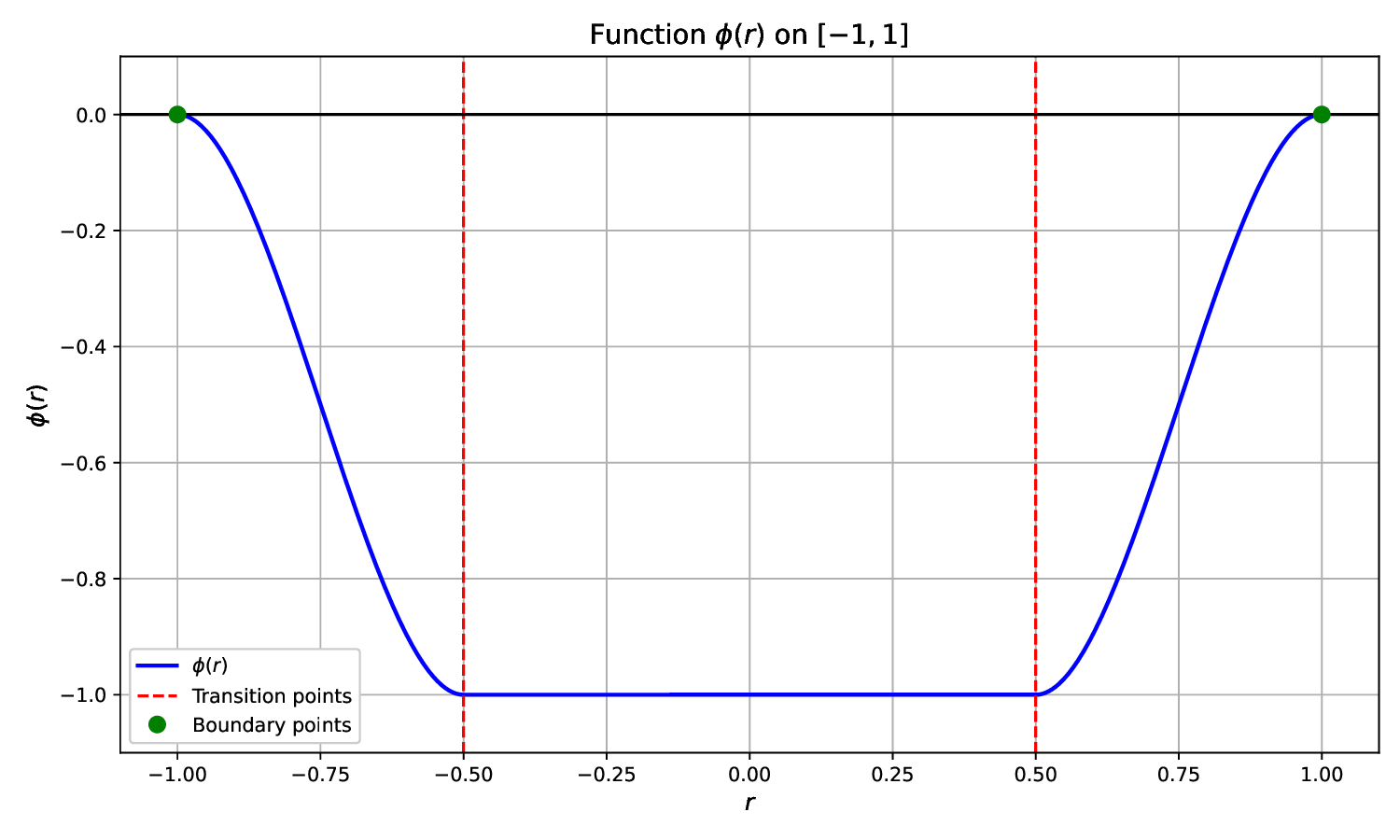}
		\subcaption{Graph of $\phi(r)$.}
	\end{subfigure}%
	\hspace{0.5cm}
	\begin{subfigure}[t]{0.30\linewidth}
		\includegraphics[width=1.0\linewidth]{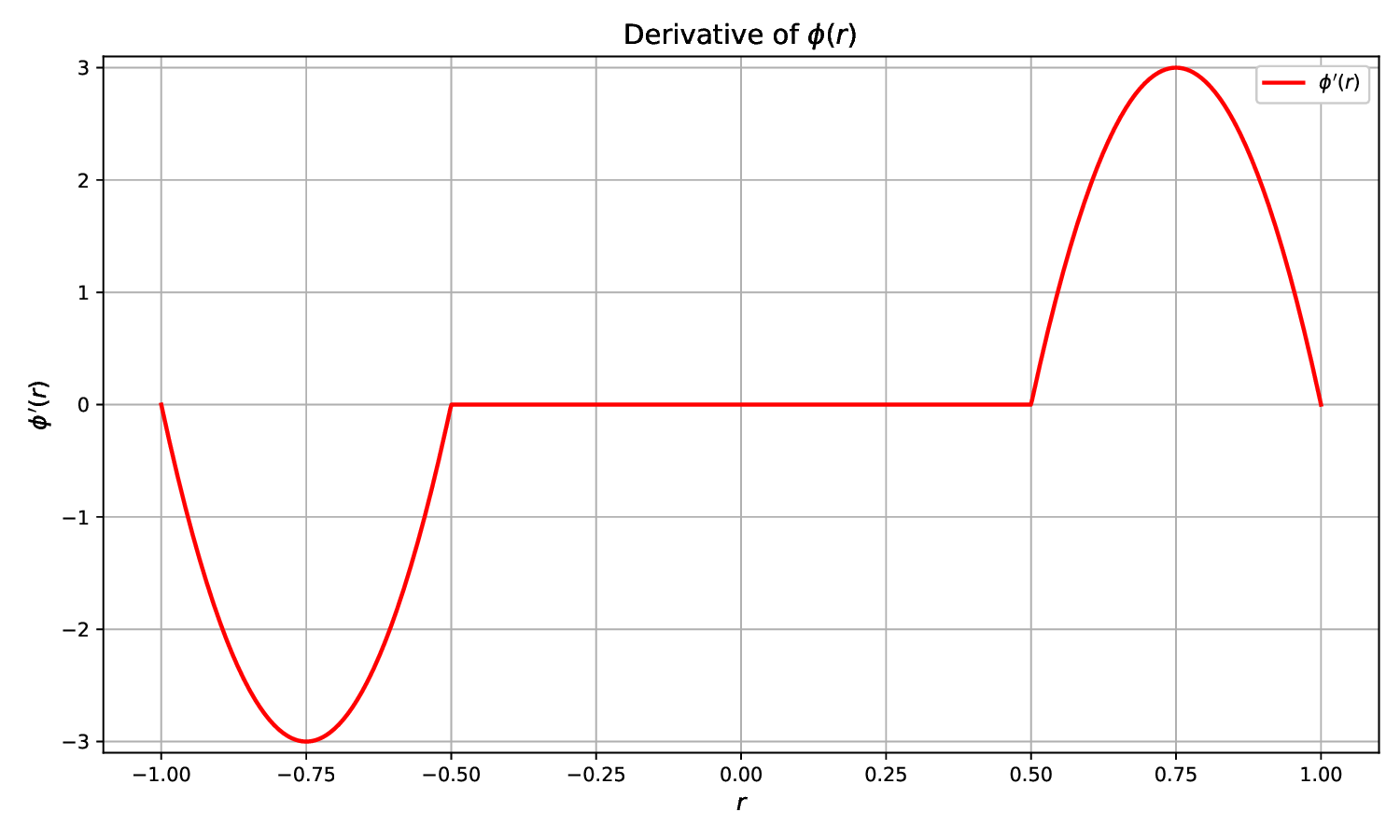}
		\subcaption{Graph of $\phi'(r)$.}
	\end{subfigure}%
	\hspace{0.5cm}
	\begin{subfigure}[t]{0.30\linewidth}
		\includegraphics[width=1.0\linewidth]{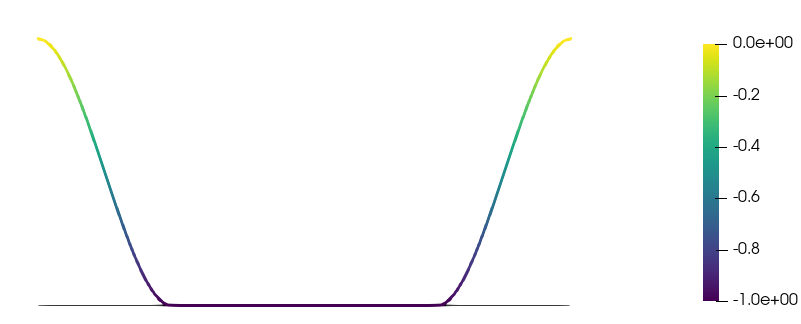}
		\subcaption{Cross section of the numerical solution.}
	\end{subfigure}%
	\caption{The manufactured solution $u_{\textup{exact}}$ enjoys radial symmetry and is of class $H^2_0(\omega)$. The numerical solution is very close to the manufactured solution.}
	\label{fig:uexact}
\end{figure}

In order to allow the non-linearity associated with the presence of the obstacle to effectively enter the numerical computations, we perturb the vector field $\vec{F}=-\nabla(\Delta u_{\textup{exact}})$ constituting the \emph{unique} forcing term associated with the manufactured solution $u_{\textup{exact}}$ by a term of order $\mathcal{O}(h)$. The intervention of the non-linearity in the model is observed by the Newton solver requiring more than one iteration to complete. The results of these experiments are reported in Figures~\ref{fig:2ter} below.

\begin{figure}[H]
	\centering
	\subfloat[Error convergence for $q=0.5$.]{
		\begin{tabular}{lcr}
\toprule
$h$ && Error\\
\midrule
4.59E-3&&7.94E-4\\
4.29E-3&&7.58E-4\\
4.03E-3&&7.25E-4\\
3.80E-3&&6.93E-4\\
3.60E-3&&6.65E-4\\
3.41E-3&&6.39E-4\\
3.25E-3&&6.16E-4\\
3.10E-3&&5.93E-4\\
2.96E-3&&5.73E-4\\
2.83E-3&&5.53E-4\\
\bottomrule
\end{tabular}

	}
	\hspace{0.5cm}
	\begin{subfigure}[t]{0.35\linewidth}
		\includegraphics[width=1.0\linewidth]{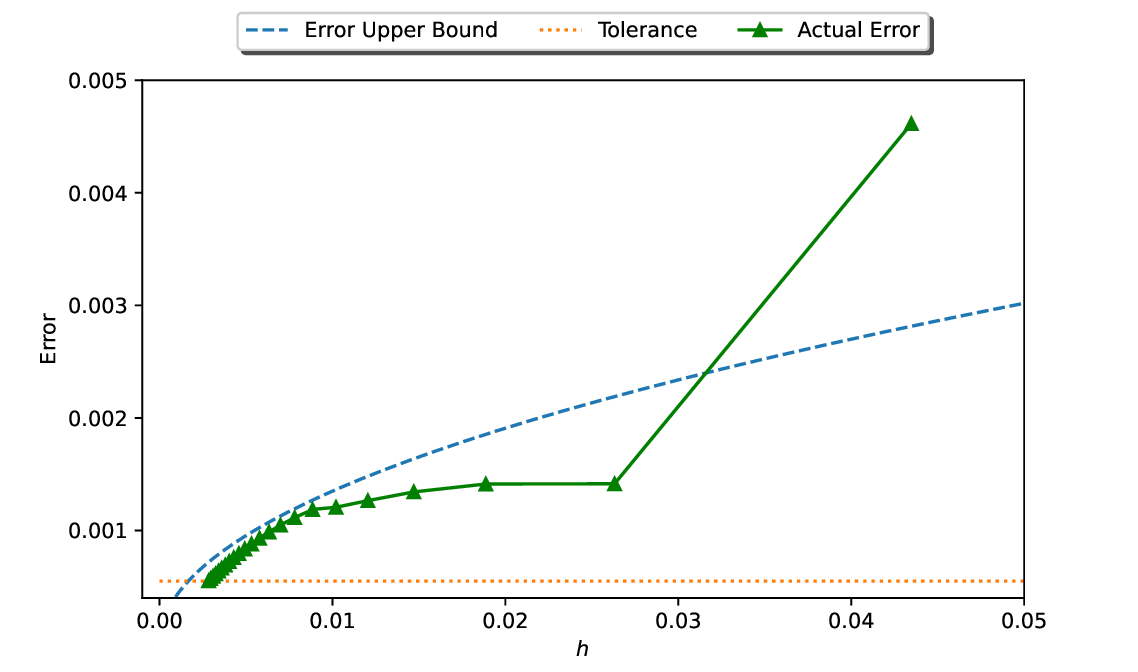}
		\subcaption{For $q=0.5$ the error displayed in the table on the left is of the order $h^{1/2}$ up to a multiplicative constant.}
	\end{subfigure}%
	\hspace{0.5cm}
	\begin{subfigure}[t]{0.35\linewidth}
		\includegraphics[width=1.0\linewidth]{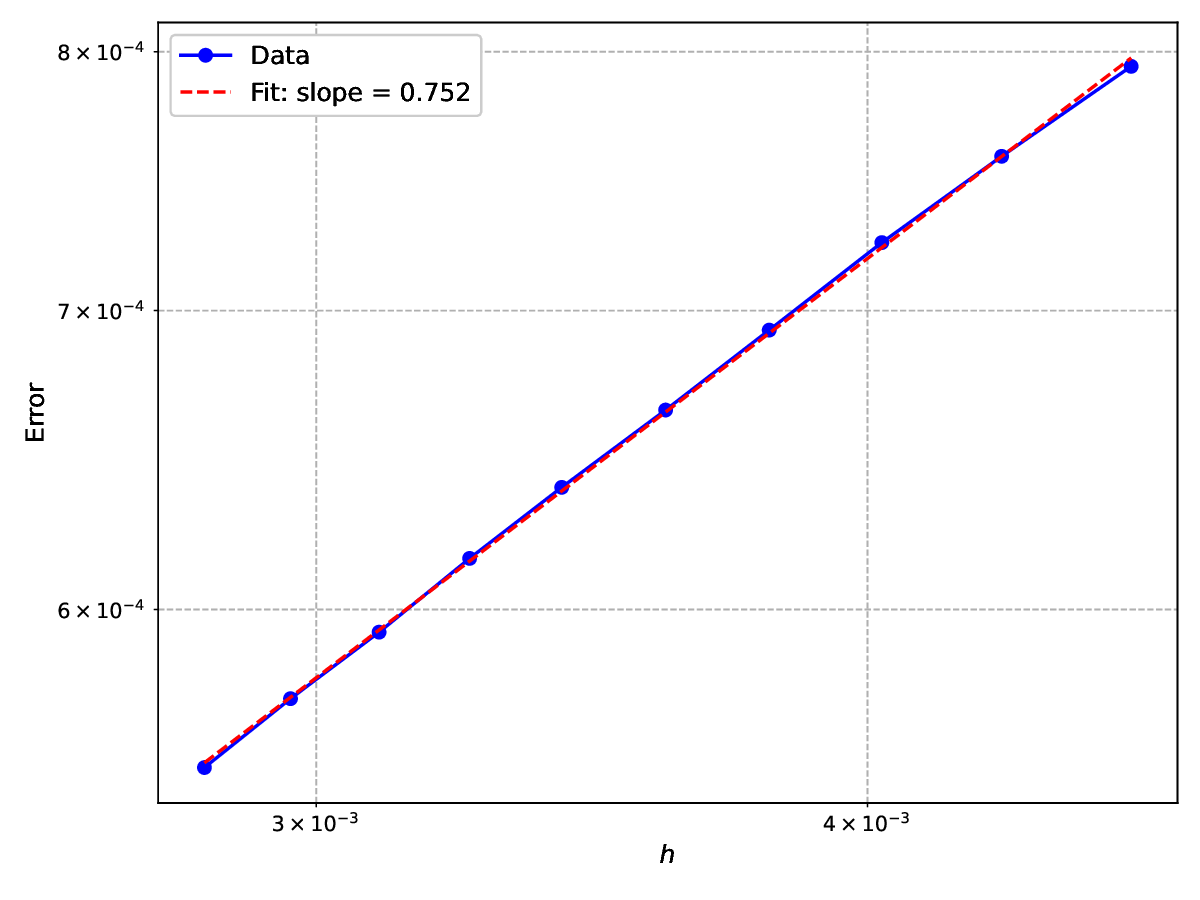}
		\subcaption{The $\log$-$\log$ version of the adjacent plot for $q=0.5$.}
	\end{subfigure}%
\end{figure}

\begin{figure}[H]
	\ContinuedFloat
	\centering
	\subfloat[Error convergence for $q=0.6$.]{
		\begin{tabular}{lcr}
\toprule
$h$ && Error\\
\midrule
4.59E-3&&3.27E-4\\
4.29E-3&&3.13E-4\\
4.03E-3&&2.99E-4\\
3.80E-3&&2.86E-4\\
3.60E-3&&2.74E-4\\
3.41E-3&&2.63E-4\\
3.25E-3&&2.53E-4\\
3.10E-3&&2.44E-4\\
2.96E-3&&2.35E-4\\
2.83E-3&&2.26E-4\\
\bottomrule
\end{tabular}

	}
	\hspace{0.5cm}
	\begin{subfigure}[t]{0.35\linewidth}
		\includegraphics[width=1.0\linewidth]{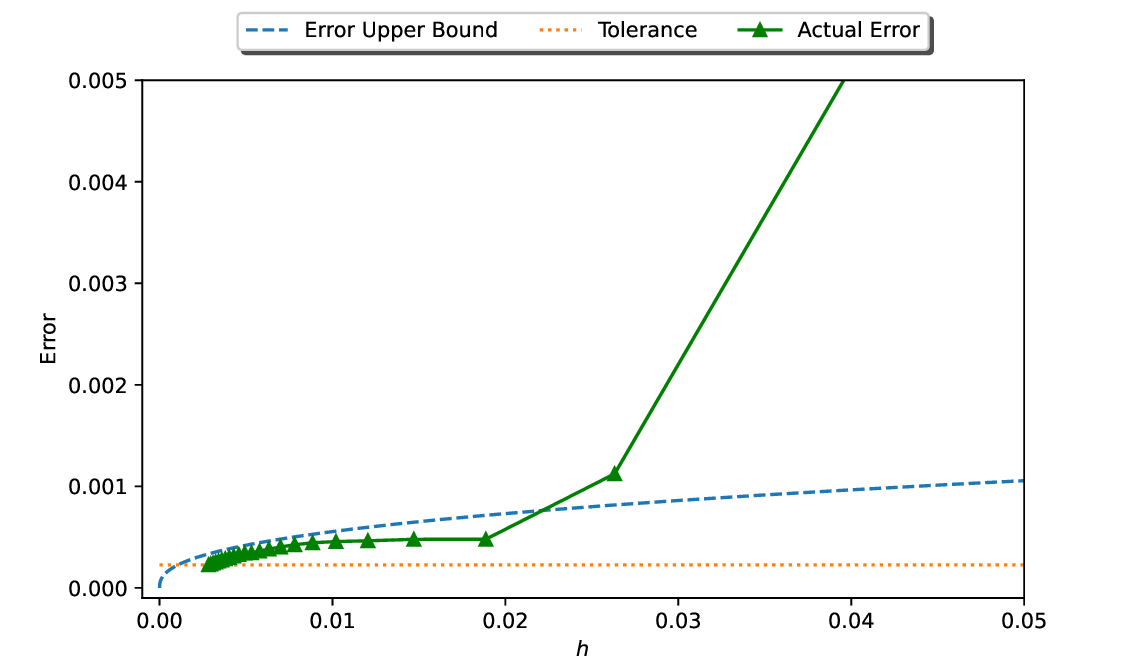}
		\subcaption{For $q=0.6$ the error is of the order $h^{2/5}$ up to a multiplicative constant.}
	\end{subfigure}%
	\hspace{0.5cm}
	\begin{subfigure}[t]{0.35\linewidth}
		\includegraphics[width=1.0\linewidth]{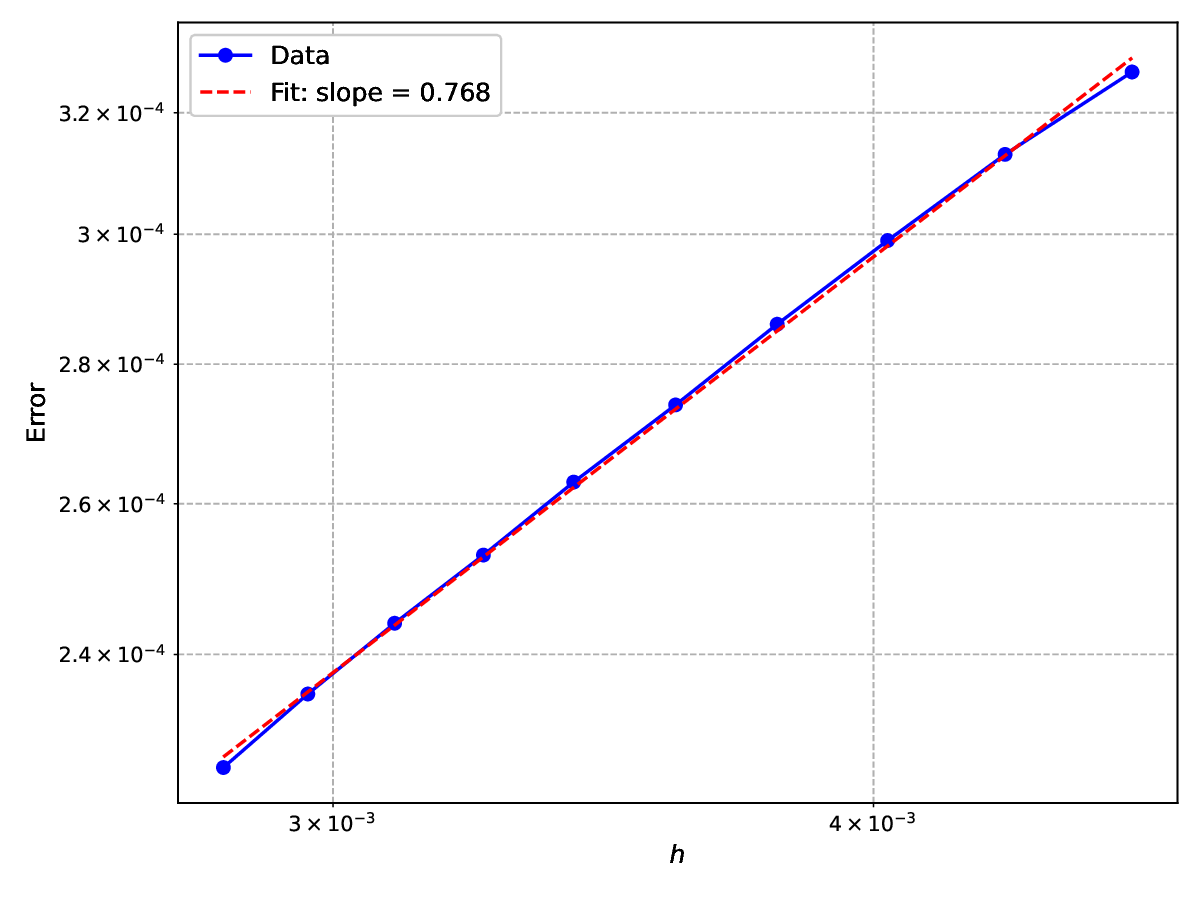}
		\subcaption{The $\log$-$\log$ version of the adjacent plot for $q=0.6$.}
	\end{subfigure}%
	\caption{Given $0<q<1$, the error $\|u_{h^{q}}^{h}-u_{\textup{exact}}\|_{H^1_0(\omega)}$ converges to zero as $h\to0$ with order of convergence $\mathcal{O}(h^{1-q})$ (viz. Theorem~\ref{th:3}).}
	\label{fig:2ter}
\end{figure}

For the second batch of experiments, we also considered the case where the obstacle is non-flat. To study this case, we consider Example~1 in~\cite{Brenner2013}. In this case, the domain $\omega$ is the circle with radius $r_A:=2.0$ and the obstacle function is given by:
\begin{equation*}
	\theta(y)=1-|y|^2,\quad\textup{ for all }y\in\overline{\omega}.
\end{equation*}

As a manufactured solution, we consider the function $u_{\textup{exact}}$ given by:
\begin{equation*}
	u_{\textup{exact}}(y):=
	\begin{cases}
		1-|y|^2&, \textup{ if } |y|\le 0.181345,\\
		\\
		0.525041|y|^2\ln|y|-0.628609|y|^2+0.017266\ln|y|+1.046746&, \textup{ if }0.181345<|y|\le 2.
	\end{cases}
\end{equation*}

Note that the function $u_{\textup{exact}}$ is radial. For sake of clarity, we sketch the graph of one of the profiles, denoted by $\phi(r)$, and of its derivative $\phi'(r)$.

\begin{figure}[H]
	\centering
	\begin{subfigure}[t]{0.30\linewidth}
		\includegraphics[width=1.0\linewidth]{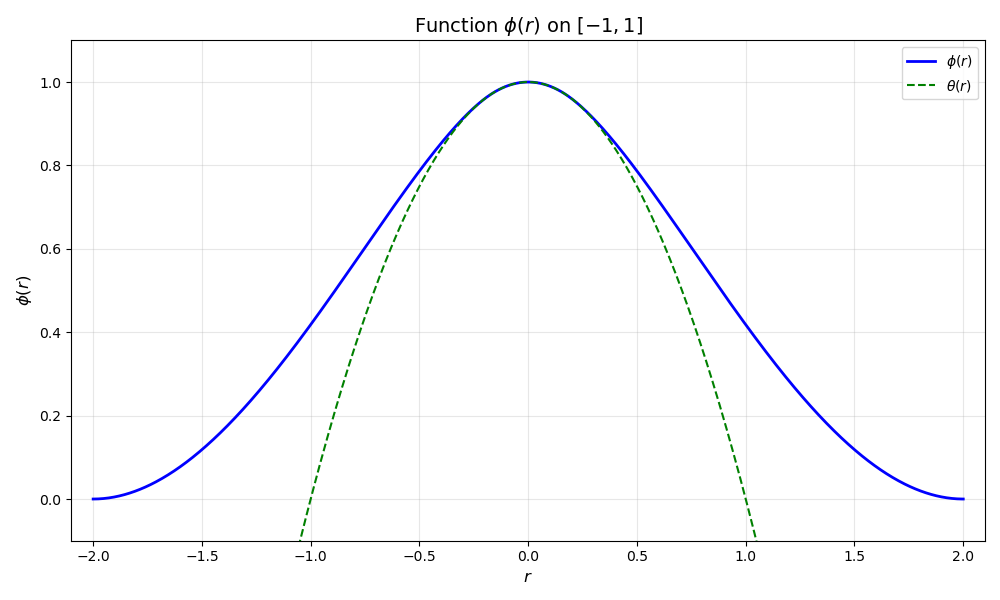}
		\subcaption{Graph of $\phi(r)$.}
	\end{subfigure}%
	\hspace{0.5cm}
	\begin{subfigure}[t]{0.30\linewidth}
		\includegraphics[width=1.0\linewidth]{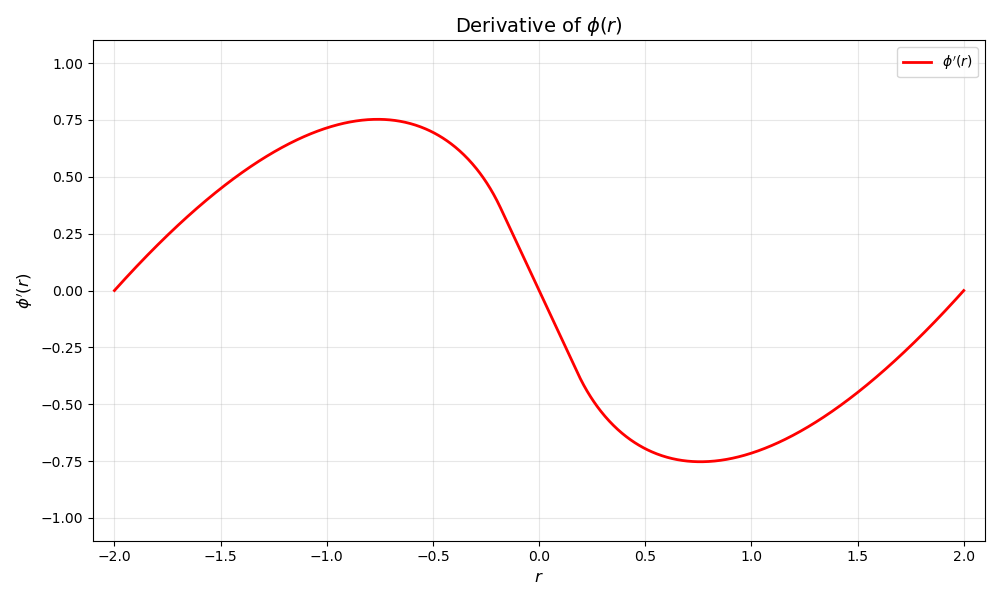}
		\subcaption{Graph of $\phi'(r)$.}
	\end{subfigure}%
	\hspace{0.5cm}
	\begin{subfigure}[t]{0.30\linewidth}
		\includegraphics[width=1.0\linewidth]{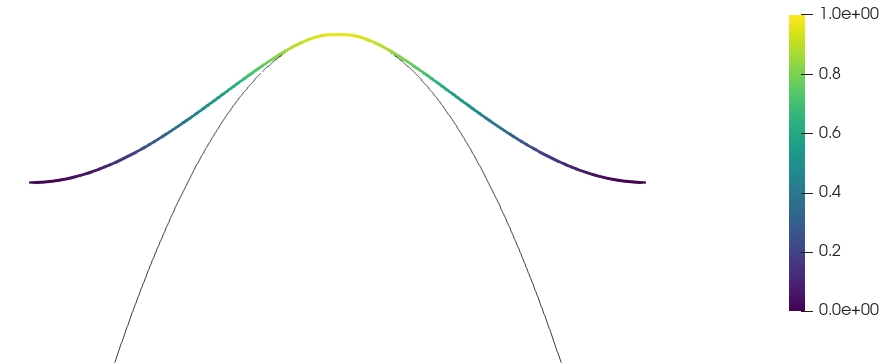}
		\subcaption{Cross section of the numerical solution.}
	\end{subfigure}%
	\caption{The manufactured solution $u_{\textup{exact}}$ enjoys radial symmetry and is of class $H^2_0(\omega)$. The numerical solution is very close to the manufactured solution.}
	\label{fig:uexact-2}
\end{figure}

In order to allow the non-linearity associated with the presence of the obstacle to effectively enter the numerical computations, we perturb the vector field $\vec{F}=-\nabla(\Delta u_{\textup{exact}})$ constituting the \emph{unique} forcing term associated with the manufactured solution $u_{\textup{exact}}$ by a term of order $\mathcal{O}(h)$. The activeness of the non-linearity in the model is observed by the Newton solver requiring more than one iteration to complete. The results of these experiments are reported in Figures~\ref{fig:2quater} below.

\begin{figure}[H]
	\centering
	\subfloat[Error convergence for $q=0.5$.]{
		\begin{tabular}{lcr}
\toprule
$h$ && Error\\
\midrule
5.67E-3&&1.03E-1\\
5.43E-3&&1.00E-1\\
5.22E-3&&9.82E-2\\
5.03E-3&&9.60E-2\\
4.84E-3&&9.41E-2\\
4.67E-3&&9.22E-2\\
4.51E-3&&9.04E-2\\
4.37E-3&&8.87E-2\\
4.23E-3&&8.71E-2\\
4.10E-3&&8.56E-2\\
\bottomrule
\end{tabular}
	}
	\hspace{0.5cm}
	\begin{subfigure}[t]{0.35\linewidth}
		\includegraphics[width=1.0\linewidth]{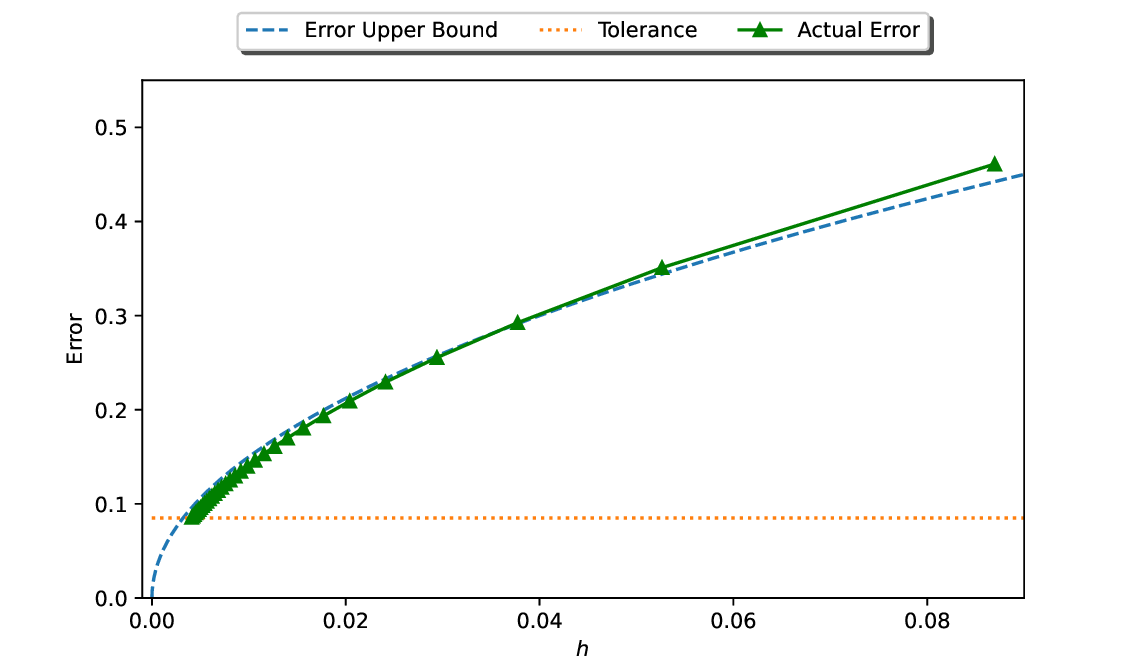}
		\subcaption{For $q=0.5$ the error is of the order $h^{1/2}$ up to a multiplicative constant.}
	\end{subfigure}%
	\hspace{0.5cm}
	\begin{subfigure}[t]{0.35\linewidth}
		\includegraphics[width=1.0\linewidth]{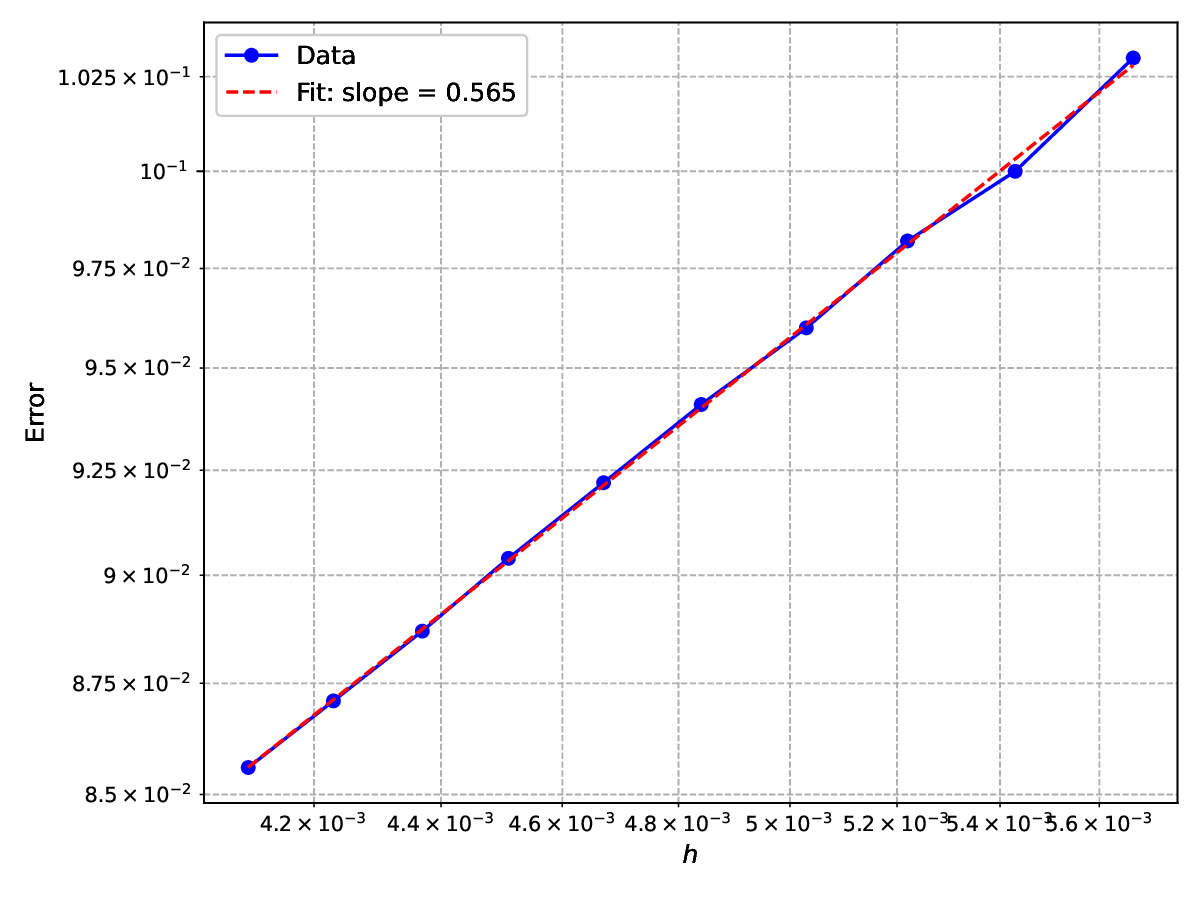}
		\subcaption{The $\log$-$\log$ version of the adjacent plot for $q=0.5$.}
	\end{subfigure}%
\end{figure}

\begin{figure}[H]
	\ContinuedFloat
	\centering
	\subfloat[Error convergence for $q=0.6$.]{
		\begin{tabular}{lcr}
\toprule
$h$ && Error\\
\midrule
2.94E-2&&1.79E-1\\
2.41E-2&&1.57E-1\\
2.04E-2&&1.40E-1\\
1.77E-2&&1.27E-1\\
1.56E-2&&1.17E-1\\
1.40E-2&&1.08E-1\\
1.27E-2&&1.01E-1\\
1.16E-2&&9.52E-2\\
1.06E-2&&8.99E-2\\
9.85E-3&&8.52E-2\\
\bottomrule
\end{tabular}
	}
	\hspace{0.5cm}
	\begin{subfigure}[t]{0.35\linewidth}
		\includegraphics[width=1.0\linewidth]{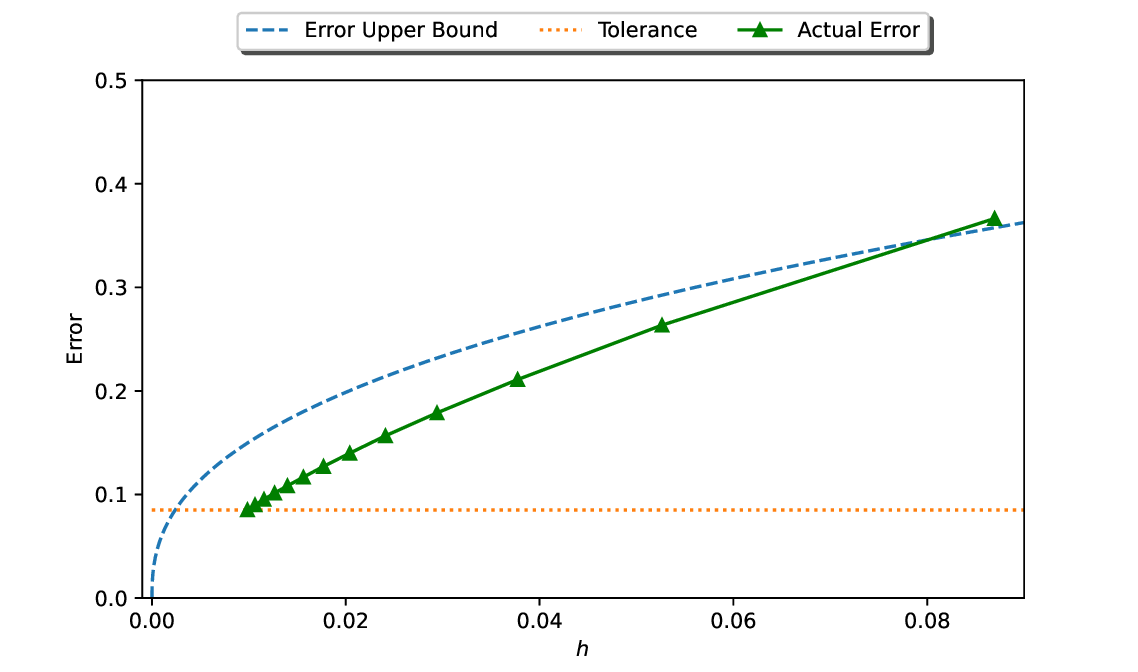}
		\subcaption{For $q=0.6$ the error is of the order $h^{2/5}$ up to a multiplicative constant.}
	\end{subfigure}%
	\hspace{0.5cm}
	\begin{subfigure}[t]{0.35\linewidth}
		\includegraphics[width=1.0\linewidth]{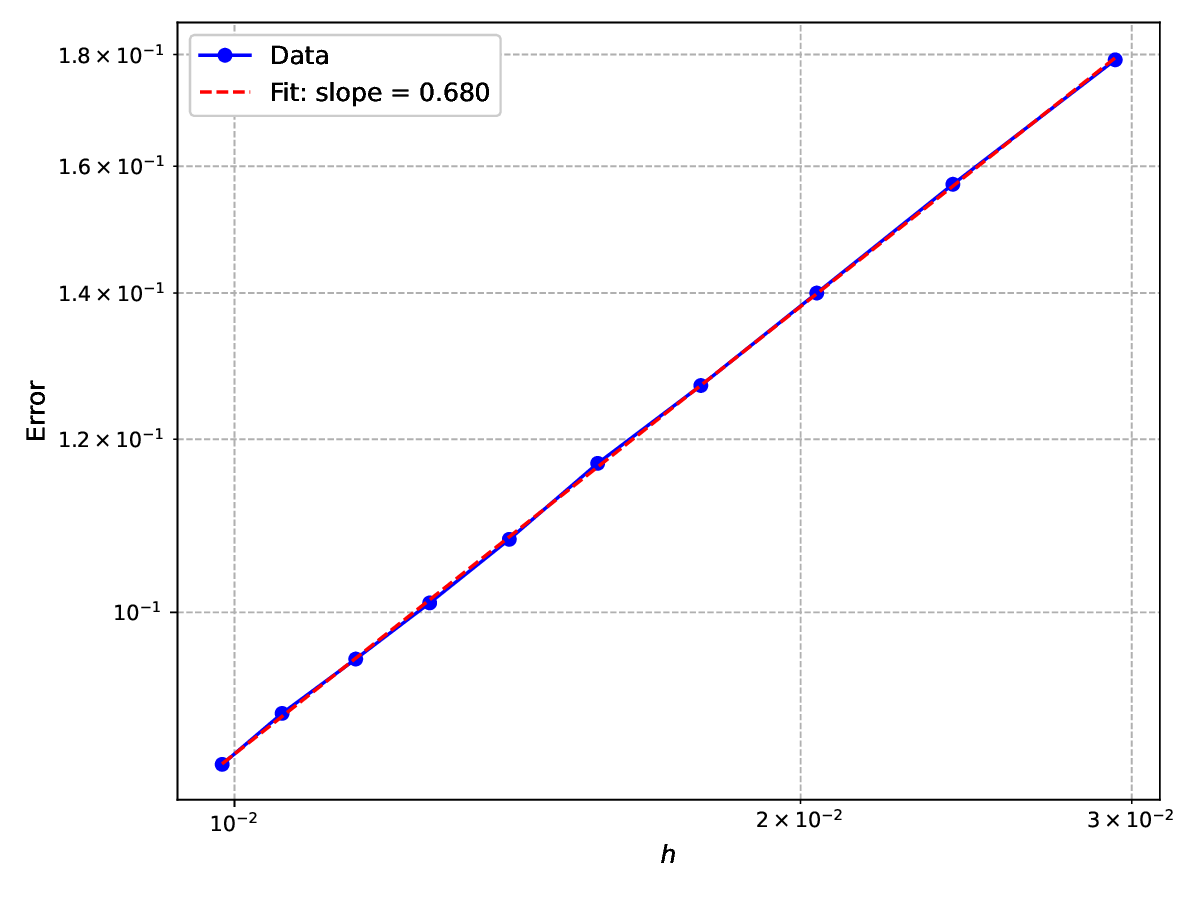}
		\subcaption{The $\log$-$\log$ version of the adjacent plot for $q=0.6$.}
	\end{subfigure}%
	\caption{Given $0<q<1$, the error $\|u_{h^{q}}^{h}-u_{\textup{exact}}\|_{H^1_0(\omega)}$ converges to zero as $h\to0$ with order of convergence $\mathcal{O}(h^{1-q})$ (viz. Theorem~\ref{th:3}).}
	\label{fig:2quater}
\end{figure}

The third batch of numerical experiments validates the genuineness of the model from the qualitative point of view.
We observe that the presented data exhibits the pattern that, for a fixed $0<h<<1$ and a fixed $0< q <1$, the contact area increases as the applied body force intensity increases.
For this experiment, we constrain solution $u$ of Problem~\ref{problem1} to remain confined above the function $\theta$ defined by two planes:

\begin{equation*}
	\theta(y):=\begin{cases}
		\dfrac{y_1-1}{2},&\textup{ if } y_1\ge 0,\\
		\\
		\dfrac{-y_1-1}{2},&\textup{ if } y_1\le 0.
	\end{cases}
\end{equation*}

The applied body force density $f$ entering the model is given by
$$
f(y):=
\begin{cases}
	(0.25y_1^2+0.25 y_2^2-0.0059 \ell), &\textup{ if } |y|^2< 0.0059 \ell,\\
	0, &\textup{otherwise},
\end{cases}
$$
where $\ell$ is a non-negative integer.
The results of these experiments are reported in Figure~\ref{fig:4bis} below.

\begin{figure}[H]
	\centering
	\begin{subfigure}[t]{0.45\linewidth}
		\includegraphics[width=\linewidth]{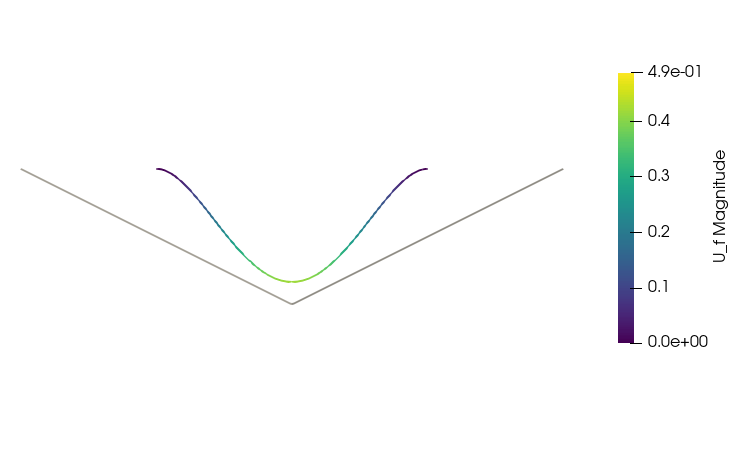}
		\subcaption{$\ell=600$}
	\end{subfigure}%
	\hspace{0.5cm}
	\begin{subfigure}[t]{0.45\linewidth}
		\includegraphics[width=1.0\linewidth]{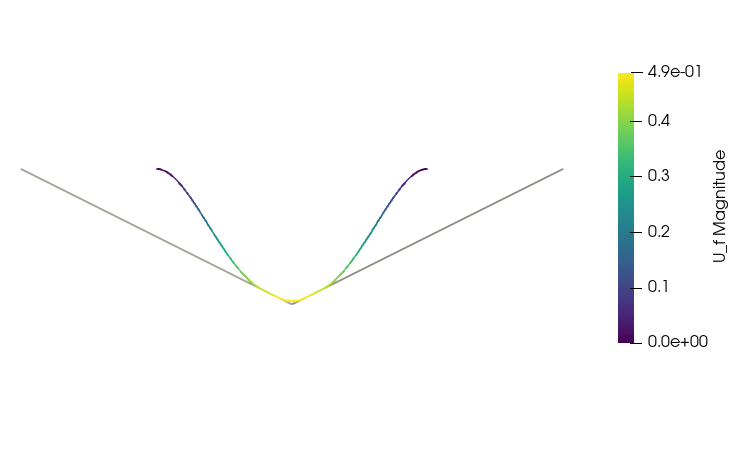}
		\subcaption{$\ell=750$}
	\end{subfigure}%
	\caption{Cross sections of a deformed membrane shell subjected not to cross a given convex obstacle.
		Given $0<h<<1$ and $0<q<1$ we observe that as the applied body force magnitude increases the contact area increases.}
	\label{fig:4bis}
\end{figure}

The first and third numerical experiment are carried out for the linearly elastic shallow shell model with flat middle surface in the supplementary material.
We refrain from carrying out the second experiment for the shallow shell model with flat surface because the execution time for the second experiment for the biharmonic model becomes extensively long when the mesh becomes finer.

\addtocontents{toc}{\protect\setcounter{tocdepth}{1}}

\section*{Conclusions and Commentary}

In this paper, we proposed a new mixed Finite Element Method for approximating the solution of obstacle problems governed by fourth order differential operators. First, we presented the method for a prototypical obstacle problem for the biharmonic operator, and we performed numerical experiments to corroborate the theoretical results we established. The method we proposed generalises the results in the papers~\cite{CiaRav74} and~\cite{CiaGlo75}, as the finite-dimensional space where the approximation is carried out is - in our case - a subspace of the space where the mixed formulation is posed. We showed that, if $\kappa=h^q$ for $0<q<1$, then the solution of Problem~\ref{problem4} converges to the solution of Problem~\ref{problem1} as $h\to 0$ (Theorem~\ref{th:3} and Theorem~\ref{th:biharmonic-strong}).

Second, we moved to the study of the numerical approximation of an obstacle problem for linearly elastic shallow shells constrained to remain confined in a prescribed half-space.
In order to establish the existence and uniqueness for the penalised mixed variational formulation related to the original problem, we had to establish a preliminary inequality of Korn's type in \emph{mixed coordinates}, whose validity hinges on the fact that the gradient matrix of the \emph{dual variable} is symmetric. In order to handle the latter constraint in the context of the numerical approximation of the solution by means of a Finite Element Method, we observed that there is a connection between the current analytical framework and the analytical framework of fluid mechanics; more precisely, the analytical framework of incompressible fluids.
As a result of this, we have shown that there are mixed Finite Element Methods for fourth order problems that still necessitate $\mathcal{C}^1$ finite elements to perform the discretisation of the solution. Moreover, we observe that the latter fact appears to be related to the lack of rigidity of the middle surface of linearly elastic shallow shells, for which a two-dimensional rigidity theorem is not available, unlike the linearly elastic shells discussed in~\cite{Ciarlet2000}.

Future directions of research might include the numerical analysis of contact problems by means of Discontinuous Galerkin Methods~\cite{Riv08}. In this case, the numerical simulations could potentially be carried out by resorting to the MATLAB library developed in~\cite{DPSY2023} for computing the elasticity coefficients entering the model.

\subsection*{Declarations}

\subsection*{Data Availability Statement}

Data sharing not applicable to this article as no datasets were generated or analysed during the current study.

\subsection*{Funding Statement}

The research of P.P. is currently being supported by the following agencies:
\begin{itemize}
	\item National Natural Science Foundation of China (NSFC), grant number W2533011;
	\item ZJ Talent Program of the Guangdong Province, grant number 2024QN11X057; 
	\item Peacock Talent Program - Type~C of the Shenzhen Municipality;
	\item University Development Fund of The Chinese University of Hong Kong, Shenzhen.
\end{itemize}

T. S. was partly supported by the Research Fund of Indiana University.

\subsection*{Conflict of Interest Disclosure}

All authors certify that they have no affiliations with or involvement in any organisation or entity with any competing interests in the subject matter or materials discussed in this manuscript.

\subsection*{Ethics Approval Statement}

Not applicable.

\subsection*{Patient Consent Statement}

Not applicable.

\subsection*{Permission to Reproduce Material from other Sources}

None of the presented material was drawn from other sources.

\subsection*{Clinical Trial Registration}

Not applicable.

\subsection*{Authors' Contribution}

All the authors equally contributed to the realisation of each part of this manuscript.

\subsection*{Acknowledgements}

This paper was started during the time P.P. held a position as a Zorn Postdoctoral Fellow at Indiana University Bloomington. P.P. and T.S. are greatly thankful to Professor Roger M. Temam for the insightful discussions and advice.

\bibliographystyle{abbrvnat} 
\bibliography{references}

\begin{thebibliography}{63}
\providecommand{\natexlab}[1]{#1}
\providecommand{\url}[1]{\texttt{#1}}
\expandafter\ifx\csname urlstyle\endcsname\relax
  \providecommand{\doi}[1]{doi: #1}\else
  \providecommand{\doi}{doi: \begingroup \urlstyle{rm}\Url}\fi

\bibitem[Agmon et~al.(1959)Agmon, Douglis, and Nirenberg]{AgmDouNir1959}
S.~Agmon, A.~Douglis, and L.~Nirenberg.
\newblock Estimates near the boundary for solutions of elliptic partial
  differential equations satisfying general boundary conditions. {I}.
\newblock \emph{Comm. Pure Appl. Math.}, 12:\penalty0 623--727, 1959.

\bibitem[Agmon et~al.(1964)Agmon, Douglis, and Nirenberg]{AgmDouNir1964}
S.~Agmon, A.~Douglis, and L.~Nirenberg.
\newblock Estimates near the boundary for solutions of elliptic partial
  differential equations satisfying general boundary conditions. {II}.
\newblock \emph{Comm. Pure Appl. Math.}, 17:\penalty0 35--92, 1964.

\bibitem[Ahrens et~al.(2005)Ahrens, Geveci, and Law]{Ahrens2005}
J.~Ahrens, B.~Geveci, and C.~Law.
\newblock \emph{ParaView: {A}n {E}nd-{U}ser {T}ool for {L}arge {D}ata
  {V}isualization}.
\newblock Visualization Handbook, Elsevier, 2005.
\newblock ISBN-13: 978-0123875822.

\bibitem[Arag\'{o}n and Duarte(2023)]{AD23}
A.~Arag\'{o}n and C.~A. Duarte.
\newblock \emph{Fundamentals of Enriched Finite Element Methods}.
\newblock 2023.

\bibitem[Bousquet et~al.(2016)Bousquet, Dragnea, Tayachi, and Temam]{BDTT16}
A.~Bousquet, B.~Dragnea, M.~Tayachi, and R.~Temam.
\newblock Towards the modeling of nanoindentation of virus shells: do substrate
  adhesion and geometry matter?
\newblock \emph{Phys. D}, 336:\penalty0 28--38, 2016.

\bibitem[Brenner et~al.(2013)Brenner, Sung, Zhang, and Zhang]{Brenner2013}
S.~Brenner, L.~Sung, H.~Zhang, and Y.~Zhang.
\newblock A {M}orley finite element method for the displacement obstacle
  problem of clamped {K}irchhoff plates.
\newblock \emph{J. {C}omput. {A}ppl. {M}ath.}, 254:\penalty0 31--42, 2013.

\bibitem[Brezis(2011)]{Brez11}
H.~Brezis.
\newblock \emph{Functional {A}nalysis, {S}obolev {S}paces and {P}artial
  {D}ifferential {E}quations}.
\newblock Springer, New York, 2011.

\bibitem[Caffarelli and Friedman(1979)]{Caffarelli1979}
L.~A. Caffarelli and A.~Friedman.
\newblock The obstacle problem for the biharmonic operator.
\newblock \emph{{A}nn. {S}cuola {N}orm. {S}up. {P}isa {C}l. {S}ci. (4)},
  6:\penalty0 151--184, 1979.

\bibitem[Caffarelli et~al.(1982)Caffarelli, Friedman, and
  Torelli]{CafFriTor1982}
L.~A. Caffarelli, A.~Friedman, and A.~Torelli.
\newblock The two-obstacle problem for the biharmonic operator.
\newblock \emph{{P}acific {J}. {M}ath.}, 103:\penalty0 325--335, 1982.

\bibitem[Cao-Rial and Rodr\'{\i}guez-Ar\'{o}s(2019)]{CR19}
M.~T. Cao-Rial and A.~Rodr\'{\i}guez-Ar\'{o}s.
\newblock Asymptotic analysis of unilateral contact problems for linearly
  elastic shells: error estimates in the membrane case.
\newblock \emph{Nonlinear Anal. Real World Appl.}, 48:\penalty0 40--53, 2019.

\bibitem[Cao-Rial et~al.(2021)Cao-Rial, Casti\~{n}eira,
  Rodr\'{\i}guez-Ar\'{o}s, and Roscani]{CCRR21}
M.~T. Cao-Rial, G.~Casti\~{n}eira, A.~Rodr\'{\i}guez-Ar\'{o}s, and S.~Roscani.
\newblock Mathematical and asymptotic analysis of thermoelastic shells in
  normal damped response contact.
\newblock \emph{Commun. Nonlinear Sci. Numer. Simul.}, 103:\penalty0 Paper No.
  105995, 22, 2021.

\bibitem[Carstensen and K{\"o}ler(2017)]{CarKol2017}
C.~Carstensen and K.~K{\"o}ler.
\newblock Nonconforming fem for the obstacle problem.
\newblock \emph{{IMA} {J}. {N}umer. {A}nal.}, 37\penalty0 (1):\penalty0 64--93,
  2017.

\bibitem[Carstensen et~al.(2021)Carstensen, Gaddam, Nataraj, Pani, and
  Shylaja]{CGN+2021}
C.~Carstensen, S.~Gaddam, N.~Nataraj, A.~K. Pani, and D.~Shylaja.
\newblock Morley finite element method for the von {K}\'{a}rm\'{a}n obstacle
  problem.
\newblock \emph{ESAIM Math. Model. Numer. Anal.}, 55\penalty0 (5):\penalty0
  1873--1894, 2021.

\bibitem[Ciarlet(1978)]{PGCFEM}
P.~G. Ciarlet.
\newblock \emph{The Finite Element Method for Elliptic Problems}.
\newblock North-Holland, Amsterdam, 1978.

\bibitem[Ciarlet(1988)]{Ciarlet1988}
P.~G. Ciarlet.
\newblock \emph{Mathematical Elasticity. Vol. I: Three-Dimensional Elasticity}.
\newblock North-Holland, Amsterdam, 1988.

\bibitem[Ciarlet(1997)]{Ciarlet1997}
P.~G. Ciarlet.
\newblock \emph{Mathematical Elasticity. Vol. II: {T}heory of Plates}.
\newblock North-Holland, Amsterdam, 1997.

\bibitem[Ciarlet(2000)]{Ciarlet2000}
P.~G. Ciarlet.
\newblock \emph{Mathematical Elasticity. Vol. III: {T}heory of Shells.}
\newblock North-Holland, Amsterdam, 2000.

\bibitem[Ciarlet(2005)]{Ciarlet2005}
P.~G. Ciarlet.
\newblock \emph{An {I}ntroduction to {D}ifferential {G}eometry with
  {A}pplications to {E}lasticity}.
\newblock Springer, Dordrecht, 2005.

\bibitem[Ciarlet(2025)]{Ciarlet2025}
P.~G. Ciarlet.
\newblock \emph{Linear and Nonlinear Functional Analysis with Applications}.
\newblock Society for Industrial and Applied Mathematics, Philadelphia, second
  edition, 2025.

\bibitem[Ciarlet and Glowinski(1975{\natexlab{a}})]{CG75}
P.~G. Ciarlet and R.~Glowinski.
\newblock Dual iterative techniques for solving a finite element approximation
  of the biharmonic equation.
\newblock \emph{Comput. Methods Appl. Mech. Engrg.}, 5:\penalty0 277--295,
  1975{\natexlab{a}}.

\bibitem[Ciarlet and Glowinski(1975{\natexlab{b}})]{CiaGlo75}
P.~G. Ciarlet and R.~Glowinski.
\newblock Dual iterative techniques for solving a finite element approximation
  of the biharmonic equation.
\newblock \emph{Comput. Methods Appl. Mech. Engrg.}, 5:\penalty0 277--295,
  1975{\natexlab{b}}.

\bibitem[Ciarlet and Lods(1996{\natexlab{a}})]{CiaLods1996a}
P.~G. Ciarlet and V.~Lods.
\newblock On the ellipticity of linear membrane shell equations.
\newblock \emph{{J}. {M}ath. {P}ures {A}ppl.}, 75:\penalty0 107--124,
  1996{\natexlab{a}}.

\bibitem[Ciarlet and Lods(1996{\natexlab{b}})]{CiaLods1996b}
P.~G. Ciarlet and V.~Lods.
\newblock Asymptotic analysis of linearly elastic shells. {I}. {J}ustification
  of membrane shell equations.
\newblock \emph{{A}rch. {R}ational {M}ech. {A}nal.}, 136\penalty0 (2):\penalty0
  119--161, 1996{\natexlab{b}}.

\bibitem[Ciarlet and Miara(1992)]{CiaMia1992}
P.~G. Ciarlet and B.~Miara.
\newblock Justification of the two-dimensional equations of a linearly elastic
  shallow shell.
\newblock \emph{Comm. Pure Appl. Math.}, 45\penalty0 (3):\penalty0 327--360,
  1992.

\bibitem[Ciarlet and Piersanti(2019)]{CiaPie2018b}
P.~G. Ciarlet and P.~Piersanti.
\newblock Obstacle problems for {K}oiter's shells.
\newblock \emph{{M}ath. {M}ech. {S}olids}, 24:\penalty0 3061--3079, 2019.

\bibitem[Ciarlet and Raviart(1974)]{CiaRav74}
P.~G. Ciarlet and P.-A. Raviart.
\newblock A mixed finite element method for the biharmonic equation.
\newblock In \emph{Mathematical aspects of finite elements in partial
  differential equations ({P}roc. {S}ympos., {M}ath. {R}es. {C}enter, {U}niv.
  {W}isconsin, {M}adison, {W}is., 1974)}, pages 125--145. Academic Press, New
  York-London, 1974.

\bibitem[Ciarlet et~al.(1996)Ciarlet, Lods, and Miara]{CiaLodsMia1996}
P.~G. Ciarlet, V.~Lods, and B.~Miara.
\newblock Asymptotic analysis of linearly elastic shells. {II}. {J}ustification
  of flexural shell equations.
\newblock \emph{Arch. {R}ational {M}ech. {A}nal.}, 136\penalty0 (2):\penalty0
  163--190, 1996.

\bibitem[Doyen and Ern(2009)]{DEP09}
D.~Doyen and A.~Ern.
\newblock Convergence of a space semi-discrete modified mass method for the
  dynamic {S}ignorini problem.
\newblock \emph{Commun. Math. Sci.}, 7\penalty0 (4):\penalty0 1063--1072, 2009.

\bibitem[Doyen and Ern(2011)]{DEP11-2}
D.~Doyen and A.~Ern.
\newblock Analysis of the modified mass method for the dynamic {S}ignorini
  problem with {C}oulomb friction.
\newblock \emph{SIAM J. Numer. Anal.}, 49\penalty0 (5):\penalty0 2039--2056,
  2011.

\bibitem[Doyen et~al.(2010)Doyen, Ern, and Piperno]{DEP10}
D.~Doyen, A.~Ern, and S.~Piperno.
\newblock A three-field augmented {L}agrangian formulation of unilateral
  contact problems with cohesive forces.
\newblock \emph{M2AN Math. Model. Numer. Anal.}, 44\penalty0 (2):\penalty0
  323--346, 2010.

\bibitem[Doyen et~al.(2011)Doyen, Ern, and Piperno]{DEP11}
D.~Doyen, A.~Ern, and S.~Piperno.
\newblock Time-integration schemes for the finite element dynamic {S}ignorini
  problem.
\newblock \emph{SIAM J. Sci. Comput.}, 33\penalty0 (1):\penalty0 223--249,
  2011.

\bibitem[Duan et~al.(2023)Duan, Piersanti, Shen, and Yang]{DPSY2023}
W.~Duan, P.~Piersanti, X.~Shen, and Q.~Yang.
\newblock Numerical corroboration of {K}oiter's model for all the main types of
  linearly elastic shells in the static case.
\newblock \emph{Math. Mech. Solids}, 28\penalty0 (11):\penalty0 2347--2369,
  2023.

\bibitem[Evans(2010)]{Evans2010}
L.~C. Evans.
\newblock \emph{Partial {D}ifferential {E}quations}.
\newblock American Mathematical Society, Providence, {S}econd edition, 2010.

\bibitem[Evans and Gariepy(2015)]{EvansGariepy2015}
L.~C. Evans and R.~F. Gariepy.
\newblock \emph{Measure theory and fine properties of functions}.
\newblock Textbooks in Mathematics. CRC Press, Boca Raton, FL, revised edition,
  2015.

\bibitem[Frehse(1971)]{Frehse1971}
J.~Frehse.
\newblock {Z}um {D}ifferenzierbarkeitsproblem bei {V}ariationsungleichungen
  höherer {O}rdnung. ({G}erman).
\newblock \emph{{A}bh. {M}ath. {S}em. {U}niv. {H}amburg}, 36:\penalty0
  140--149, 1971.

\bibitem[Gazzola et~al.(2024)Gazzola, Pata, and Patriarca]{GPP2024}
F.~Gazzola, V.~Pata, and C.~Patriarca.
\newblock Attractors for a fluid-structure interaction problem in a
  time-dependent phase space.
\newblock \emph{J. Funct. Anal.}, 286\penalty0 (2):\penalty0 Paper No. 110199,
  56, 2024.

\bibitem[Girault and Raviart(1986)]{GR86}
V.~Girault and P.-A. Raviart.
\newblock \emph{Finite element methods for {N}avier-{S}tokes equations},
  volume~5 of \emph{Springer Series in Computational Mathematics}.
\newblock Springer-Verlag, Berlin, 1986.
\newblock Theory and algorithms.

\bibitem[Gong et~al.(2019)Gong, Wu, and Xu]{GWX19}
S.~Gong, S.~Wu, and J.~Xu.
\newblock New hybridized mixed methods for linear elasticity and optimal
  multilevel solvers.
\newblock \emph{Numer. Math.}, 141\penalty0 (2):\penalty0 569--604, 2019.

\bibitem[Hou and Temam(2017)]{HT17}
Y.~Hou and R.~Temam.
\newblock About the modeling of the indentation of a virus shell: the role of
  the shape of the probe.
\newblock \emph{J. Sci. Comput.}, 73\penalty0 (2-3):\penalty0 783--796, 2017.

\bibitem[Koiter(1959)]{Koiter1959}
W.~T. Koiter.
\newblock A consistent first approximation in the general theory of thin
  elastic shells.
\newblock In \emph{Proc. Sympos. Thin Elastic Shells (Delft)}, pages 12--33,
  Amsterdam, 1959. North-Holland.

\bibitem[Koiter(1966)]{Koiter1966}
W.~T. Koiter.
\newblock On the nonlinear theory of thin elastic shells. {I}, {II}, {III}.
\newblock \emph{Nederl. Akad. Wetensch. Proc. Ser. B}, 69:\penalty0 1--17,
  18--32, 33--54, 1966.

\bibitem[Koiter(1970)]{Koiter1970}
W.~T. Koiter.
\newblock On the foundations of the linear theory of thin elastic shells. {I},
  {II}.
\newblock \emph{{N}ederl. {A}kad. {W}etensch. {P}roc. {S}er. {B} 73 (1970),
  169--182; ibid}, 73:\penalty0 183--195, 1970.

\bibitem[Langtangen and Logg(2016)]{Fenics2016}
H.~P. Langtangen and A.~Logg.
\newblock \emph{Solving {PDE}s in {P}ython}, volume~3 of \emph{Simula
  SpringerBriefs on Computing}.
\newblock Springer, Cham, 2016.
\newblock The FEniCS tutorial I.

\bibitem[L\'eger and Miara(2008)]{Leger2008}
A.~L\'eger and B.~Miara.
\newblock Mathematical justification of the obstacle problem in the case of a
  shallow shell.
\newblock \emph{J. {E}lasticity}, 90:\penalty0 241--257, 2008.

\bibitem[L\'eger and Miara(2010)]{Leger2010}
A.~L\'eger and B.~Miara.
\newblock Erratum to: {M}athematical justification of the obstacle problem in
  the case of a shallow shell.
\newblock \emph{J. {E}lasticity}, 98:\penalty0 115--116, 2010.

\bibitem[Lewicka(2023)]{Lewicka2023}
M.~Lewicka.
\newblock \emph{Calculus of variations on thin prestressed films---asymptotic
  methods in elasticity}, volume 101 of \emph{Progress in Nonlinear
  Differential Equations and their Applications}.
\newblock Birkh\"{a}user/Springer, Cham, 2023.

\bibitem[Li et~al.(2025)Li, Ye, and Chung]{LYC25}
Z.~Li, C.~Ye, and E.~T. Chung.
\newblock An iterative constraint energy minimizing generalized multiscale
  finite element method for contact problem.
\newblock \emph{J. Comput. Appl. Math.}, 461:\penalty0 Paper No. 116398, 19,
  2025.

\bibitem[Ma and Wang(2019)]{MW2019}
T.~Ma and S.~Wang.
\newblock \emph{Phase transition dynamics}.
\newblock Springer, Cham, 2019.
\newblock Second edition of [ MR3154868].

\bibitem[Meixner and Piersanti(2024)]{MeiPie2024}
A.~Meixner and P.~Piersanti.
\newblock Numerical approximation of the solution of an obstacle problem
  modelling the displacement of elliptic membrane shells via the penalty
  method.
\newblock \emph{Appl. Math. Optim.}, 89\penalty0 (2):\penalty0 Paper No. 45,
  60, 2024.

\bibitem[Miranville(2019)]{Miranville2019}
A.~Miranville.
\newblock \emph{The {C}ahn-{H}illiard equation}, volume~95 of \emph{CBMS-NSF
  Regional Conference Series in Applied Mathematics}.
\newblock Society for Industrial and Applied Mathematics (SIAM), Philadelphia,
  PA, 2019.
\newblock Recent advances and applications.

\bibitem[Ne\v{c}as(1967)]{Nec67}
J.~Ne\v{c}as.
\newblock \emph{Les m\'{e}thodes directes en th\'{e}orie des \'{e}quations
  elliptiques}.
\newblock Masson et Cie, \'{E}diteurs, Paris; Academia, \'{E}diteurs, Prague,
  1967.

\bibitem[Piersanti(2022{\natexlab{a}})]{Pie-2022-interior}
P.~Piersanti.
\newblock On the improved interior regularity of the solution of a second order
  elliptic boundary value problem modelling the displacement of a linearly
  elastic elliptic membrane shell subject to an obstacle.
\newblock \emph{Discrete Contin. Dyn. Syst.}, 42\penalty0 (2):\penalty0
  1011--1037, 2022{\natexlab{a}}.

\bibitem[Piersanti(2022{\natexlab{b}})]{Pie2020-1}
P.~Piersanti.
\newblock On the improved interior regularity of the solution of a fourth order
  elliptic problem modelling the displacement of a linearly elastic shallow
  shell subject to an obstacle.
\newblock \emph{{A}symptot. {A}nal.}, 127\penalty0 (1--2):\penalty0 35--55,
  2022{\natexlab{b}}.

\bibitem[Piersanti(2023)]{Pie2023}
P.~Piersanti.
\newblock Asymptotic analysis of linearly elastic flexural shells subjected to
  an obstacle in absence of friction.
\newblock \emph{J. Nonlinear Sci.}, 33\penalty0 (4):\penalty0 Paper No. 58, 39,
  2023.

\bibitem[Piersanti and Shen(2020)]{PS}
P.~Piersanti and X.~Shen.
\newblock Numerical methods for static shallow shells lying over an obstacle.
\newblock \emph{{N}umer. {A}lgorithms}, 85\penalty0 (2):\penalty0 623--652,
  2020.

\bibitem[Rivi\`ere(2008)]{Riv08}
B.~Rivi\`ere.
\newblock \emph{Discontinuous {G}alerkin methods for solving elliptic and
  parabolic equations}, volume~35 of \emph{Frontiers in Applied Mathematics}.
\newblock Society for Industrial and Applied Mathematics (SIAM), Philadelphia,
  PA, 2008.
\newblock Theory and implementation.

\bibitem[Rodr\'{\i}guez-Ar\'{o}s(2018)]{Rodri2018}
A.~Rodr\'{\i}guez-Ar\'{o}s.
\newblock Mathematical justification of the obstacle problem for elastic
  elliptic membrane shells.
\newblock \emph{{A}pplicable {A}nal.}, 97:\penalty0 1261--1280, 2018.

\bibitem[Rodr\'{\i}guez-Ar\'{o}s and Cao-Rial(2018{\natexlab{a}})]{RC18}
A.~Rodr\'{\i}guez-Ar\'{o}s and M.~T. Cao-Rial.
\newblock Asymptotic analysis of linearly elastic shells in normal compliance
  contact: convergence for the elliptic membrane case.
\newblock \emph{Z. Angew. Math. Phys.}, 69\penalty0 (5):\penalty0 Paper No.
  115, 22, 2018{\natexlab{a}}.

\bibitem[Rodr\'{\i}guez-Ar\'{o}s and Cao-Rial(2018{\natexlab{b}})]{RC18-2}
A.~Rodr\'{\i}guez-Ar\'{o}s and M.~T. Cao-Rial.
\newblock Mathematical and numerical analysis of an obstacle problem for
  elastic piezoelectric beams.
\newblock \emph{Math. Mech. Solids}, 23\penalty0 (3):\penalty0 262--278,
  2018{\natexlab{b}}.

\bibitem[Temam(2001)]{Temam2001}
R.~Temam.
\newblock \emph{Navier-{S}tokes equations}.
\newblock AMS Chelsea Publishing, Providence, RI, 2001.
\newblock Theory and numerical analysis, Reprint of the 1984 edition.

\bibitem[Wu et~al.(2024{\natexlab{a}})Wu, Shen, Shi, and Yu]{WSSY24}
R.~Wu, X.~Shen, D.~Shi, and J.~Yu.
\newblock Nonconforming finite element methods for two-dimensional linearly
  elastic shallow shell model.
\newblock \emph{Adv. Appl. Math. Mech.}, 16\penalty0 (2):\penalty0 493--518,
  2024{\natexlab{a}}.

\bibitem[Wu et~al.(2024{\natexlab{b}})Wu, Shen, Yang, and Zhu]{WSYZ24}
R.~Wu, X.~Shen, Q.~Yang, and S.~Zhu.
\newblock A new priori error estimation of nonconforming element for
  two-dimensional linearly elastic shallow shell equations.
\newblock \emph{Commun. Math. Sci.}, 22\penalty0 (1):\penalty0 167--179,
  2024{\natexlab{b}}.

\bibitem[Young(1912)]{Young1912}
W.~H. Young.
\newblock On {C}lasses of {S}ummable {F}unctions and their {F}ourier {S}eries.
\newblock \emph{{P}roc. {R}. {S}oc. {L}ond. {S}er. {A} {M}ath. {P}hys. {E}ng.
  {S}ci.}, 87:\penalty0 225--229, 1912.

\end{thebibliography}

\addresseshere

\clearpage

\appendix

\newpage%
\renewcommand{\thesection}{\Alph{section}}

\section{Numerical experiments simulating the displacement of linearly elastic shallow shells with flat middle surface subjected to an obstacle}
\label{sec4}

In this last section of the paper, we implement numerical simulations aiming to validate the theoretical results presented in section~\ref{sec2} in the case where the middle surface of the shallow shell under consideration is flat.
Consider as a domain a circle of radius $r_A:=0.5$, and denote one such domain by $\omega$:
\begin{equation*}
	\omega:=\left\{y=(y_\alpha)\in \mathbb{R}^2;\sqrt{y_1^2+y_2^2}<r_A\right\}.
\end{equation*}

Throughout this section, the values of $\varepsilon$, $\lambda$ and $\mu$ are fixed once and for all as follows:
\begin{equation*}
	\begin{aligned}
		\varepsilon&=0.001,\\
		\lambda&=0.4,\\
		\mu&=0.012.
	\end{aligned}
\end{equation*}

The parametrisation for the middle surface of the linearly elastic shallow shell under consideration is given by:
\begin{equation*}
	\label{middlesurf}
	\bm{\theta}^\varepsilon(y):=\left(y_1, y_2, 0.15\right),\quad\textup{ for all } y=(y_\alpha) \in \overline{\omega}.
\end{equation*}

We define the unit-norm vector $\bm{q}$ orthogonal to the plane $\{y_3 = 0\}$ constituting the obstacle by $q:=(0,0,1)$. The applied body force density $\bm{p}^\varepsilon=(p^{i,\varepsilon})$ entering the first two batches of experiments is given by $\bm{p}^\varepsilon=(0,0,g(y))$, where
$$
g(y):=
\begin{cases}
	-(-5.0 y_1^2-5.0 y_2^2+0.295), &\textup{ if } |y|^2< 0.060,\\
	0, &\textup{otherwise}.
\end{cases}
$$

The first batch of numerical experiments is meant to validate the convergences established in Theorem~\ref{th:shallow-strong}. After fixing the mesh size $0<h<<1$, we let $\kappa$ tend to zero in Problem~\ref{problem9}. Let $(\bm{\zeta}^{\varepsilon,h}_{\kappa_1},\vec{\xi}^{\varepsilon,h}_{\kappa_1})$ and $(\bm{\zeta}^{\varepsilon,h}_{\kappa_2},\vec{\xi}^{\varepsilon,h}_{\kappa_2})$ be two solutions of Problem~\ref{problem9}, where $\kappa_2:=2\kappa_1$ and $\kappa_1>0$.
The solution of Problem~\ref{problem8} is discretised component-wise by Courant triangles (cf., e.g., \cite{PGCFEM}) and homogeneous Dirichlet boundary conditions are imposed for all the components. 
At each iteration, Problem~\ref{problem9} is solved by Newton's method. The algorithm stops when the error residual with respect to the standard norm of $\bm{H}^1_0(\omega)$ is smaller than $1.0 \times 10^{-8}$. We verify that the error remains bounded below $C\sqrt{\kappa}$, where $C=4$.

\begin{figure}[H]
	\centering
	\subfloat[Error convergence.]{
		\begin{tabular}{lcr}
\toprule
$\kappa$ && Error\\
\midrule
4.77E-8&&4.24E-6\\
2.38E-8&&2.12E-6\\
1.19E-8&&1.06E-6\\
5.96E-9&&5.30E-7\\
2.98E-9&&2.65E-7\\
1.49E-9&&1.33E-7\\
7.45E-10&&6.63E-8\\
3.73E-10&&3.31E-8\\
1.86E-10&&1.66E-8\\
9.31E-11&&8.28E-9\\
\bottomrule
\end{tabular}

	}
	\hspace{1.5cm}
	\begin{subfigure}[t]{0.52\linewidth}
		\includegraphics[width=1.0\linewidth]{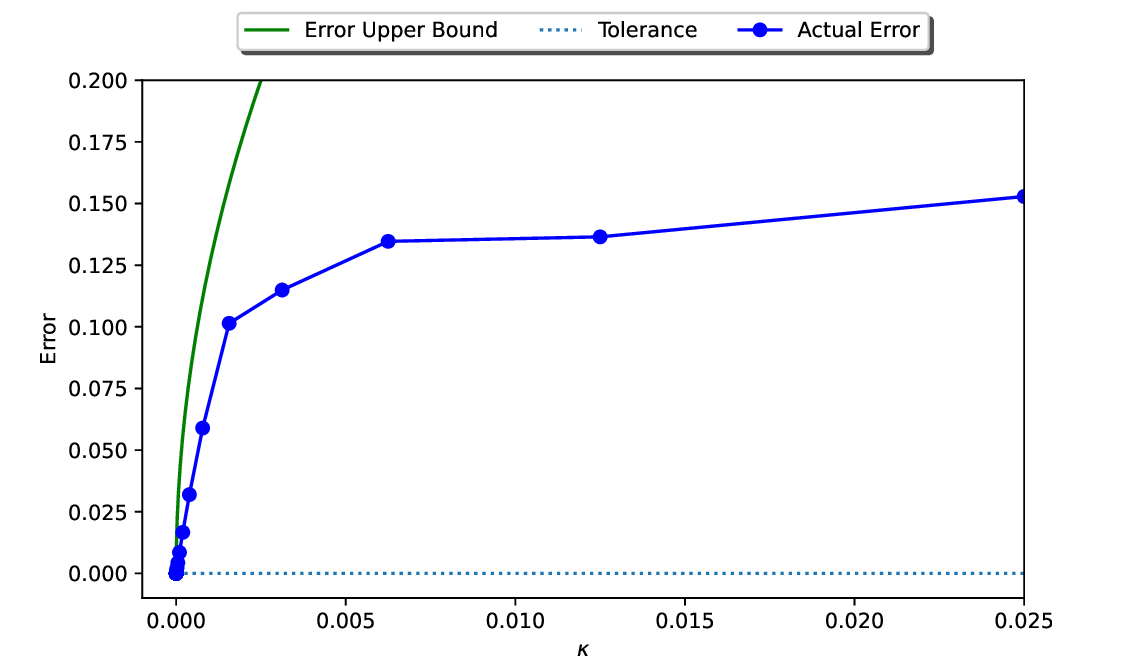}
		\subcaption{Error convergence as $\kappa \to 0$ for $h=0.0625$.}
	\end{subfigure}%
	
\end{figure}

\begin{figure}[H]
	\ContinuedFloat
	\centering
	\subfloat[Error convergence.]{
		\begin{tabular}{lcr}
\toprule
$\kappa$ && Error\\
\midrule
1.19E-8&&4.02E-6\\
5.96E-9&&2.01E-6\\
2.98E-9&&1.00E-6\\
1.49E-9&&5.02E-7\\
7.45E-10&&2.51E-7\\
3.73E-10&&1.25E-7\\
1.86E-10&&6.27E-8\\
9.31E-11&&3.14E-8\\
4.66E-11&&1.57E-8\\
2.33E-11&&7.84E-9\\
\bottomrule
\end{tabular}

	}
	\hspace{1.5cm}
	\begin{subfigure}[t]{0.52\linewidth}
		\includegraphics[width=1.0\linewidth]{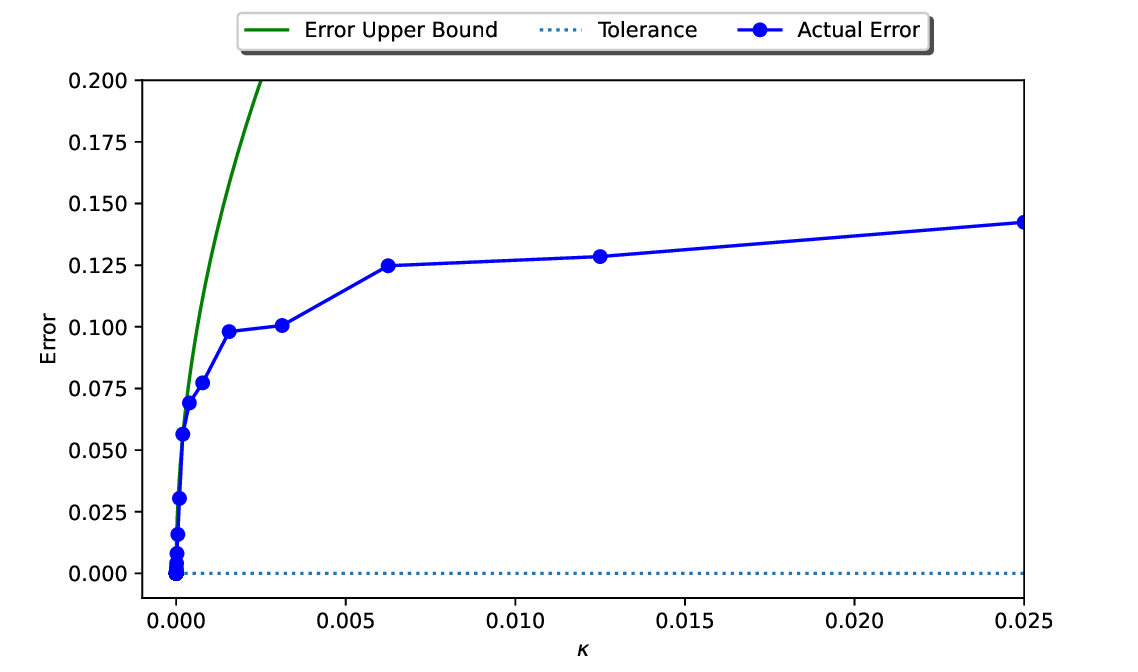}
		\subcaption{Error convergence as $\kappa \to 0$ for $h=0.03125$.}
	\end{subfigure}%
\end{figure}

\begin{figure}[H]
	\ContinuedFloat
	\centering
	\subfloat[Error convergence.]{
		\begin{tabular}{lcr}
\toprule
$\kappa$ && Error\\
\midrule
2.98E-9&&4.02E-6\\
1.49E-9&&2.01E-6\\
7.45E-10&&1.00E-6\\
3.73E-10&&5.02E-7\\
1.86E-10&&2.51E-7\\
9.31E-11&&1.25E-7\\
4.66E-11&&6.27E-8\\
2.33E-11&&3.14E-8\\
1.16E-11&&1.57E-8\\
5.82E-12&&7.84E-9\\
\bottomrule
\end{tabular}

	}
	\hspace{1.5cm}
	\begin{subfigure}[t]{0.52\linewidth}
		\includegraphics[width=1.0\linewidth]{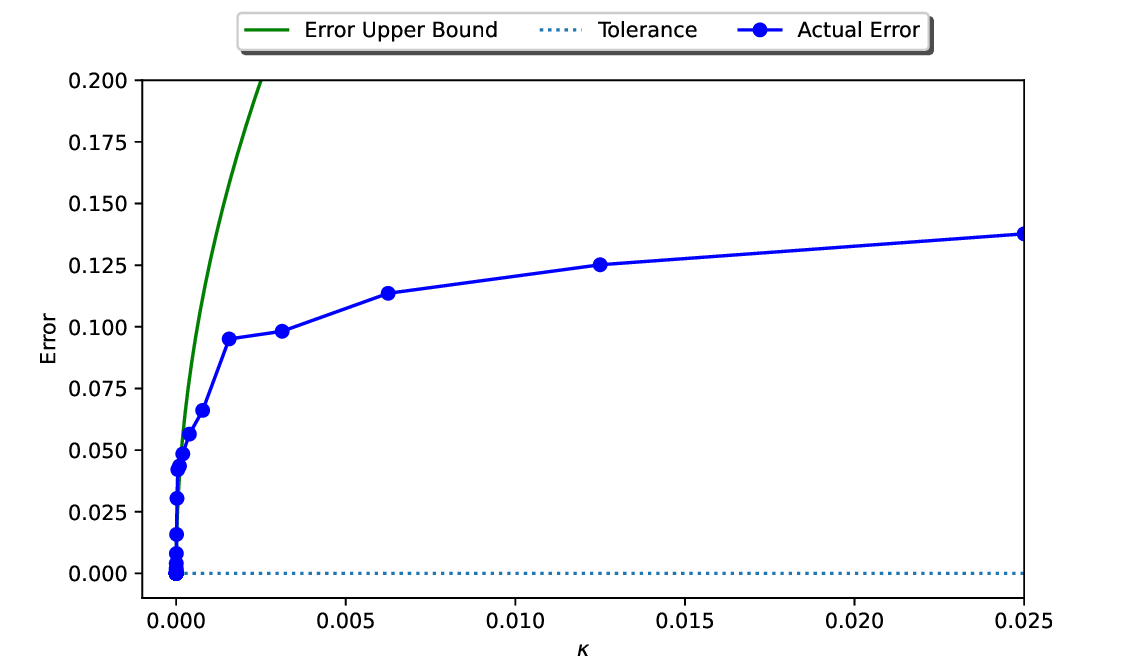}
		\subcaption{Error convergence as $\kappa \to 0$ for $h=0.015625$.}
	\end{subfigure}%
	
\end{figure}

\begin{figure}[H]
	\ContinuedFloat
	\centering
	\subfloat[Error convergence.]{
		\begin{tabular}{lcr}
\toprule
$\kappa$ && Error\\
\midrule
7.45E-10&&3.96E-6\\
3.73E-10&&1.98E-6\\
1.86E-10&&9.91E-7\\
9.31E-11&&4.96E-7\\
4.66E-11&&2.48E-7\\
2.33E-11&&1.24E-7\\
1.16E-11&&6.20E-8\\
5.82E-12&&3.10E-8\\
2.91E-12&&1.55E-8\\
1.46E-12&&7.74E-9\\
\bottomrule
\end{tabular}

	}
	\hspace{1.5cm}
	\begin{subfigure}[t]{0.52\linewidth}
		\includegraphics[width=1.0\linewidth]{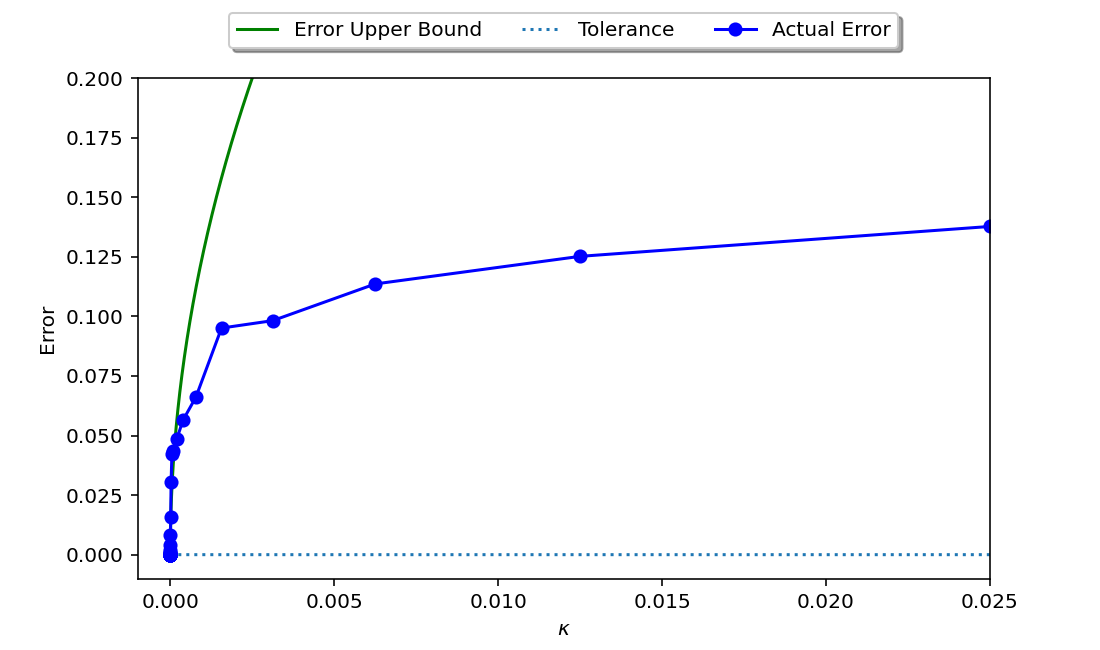}
		\subcaption{Error convergence as $\kappa \to 0$ for $h=0.0078125$.}
	\end{subfigure}%
	\caption{According to Theorem~\ref{th:shallow-strong}, given $0<h<<1$, the first component of the solution $(\bm{\zeta}^{\varepsilon,h}_\kappa,\vec{\xi}^{\varepsilon,h}_\kappa)$ of Problem~\ref{problem9} converges with respect to the standard norm of $\bm{H}^1_0(\omega)$ as $\kappa\to0$. In particular, we verify that $\{\bm{\zeta}^{\varepsilon,h}_\kappa\}_{\kappa>0}$ is a Cauchy sequence in $\bm{H}^1_0(\omega)$.}
	\label{fig:1}
\end{figure}

From the data patterns in Figure~\ref{fig:1} we observe that, for a given mesh size $h$, the solution $(\bm{\zeta}^{\varepsilon,h}_\kappa,\vec{\xi}^{\varepsilon,h}_\kappa)$ of Problem~\ref{problem9} converges as $\kappa \to 0$.

The second batch of numerical experiments validates the genuineness of the model from the qualitative point of view.
We observe that the presented data exhibits the pattern that, for a fixed $0<h<<1$ and a fixed $0<q<1$, the contact area increases as the applied body force intensity increases.

For this experiment, we constrain the shell to remain confined in the convex portion of Euclidean space $\mathbb{E}^3$ identified by the planes $z=\pm x/2$. The unit-normal vectors in the orthogonal complement of these two planes are given by
\begin{equation*}
	\bm{q}_1:=\left(-\dfrac{\sqrt{5}}{5},0,\dfrac{2\sqrt{5}}{5}\right) \quad \textup{ and }\quad \bm{q}_2:=\left(\dfrac{\sqrt{5}}{5},0,\dfrac{2\sqrt{5}}{5}\right),
\end{equation*}
respectively.

The applied body force density $\bm{p}^\varepsilon=(p^{i,\varepsilon})$ entering the model is given by $\bm{p}^\varepsilon=(0,0,\varepsilon^3 g_\ell(y))$, where $\ell$ is a non-negative integer and
$$
g_\ell(y):=
\begin{cases}
	(0.5y_1^2+ 0.5y_2^2-0.0059 \ell), &\textup{ if } |y|^2< 0.0059 \ell,\\
	0, &\textup{otherwise},
\end{cases}
$$
where $\ell$ is a non-negative integer.
The results of these experiments are reported in Figure~\ref{fig:4} below.

\begin{figure}[H]
	\centering
	\begin{subfigure}[t]{0.45\linewidth}
		\includegraphics[width=\linewidth]{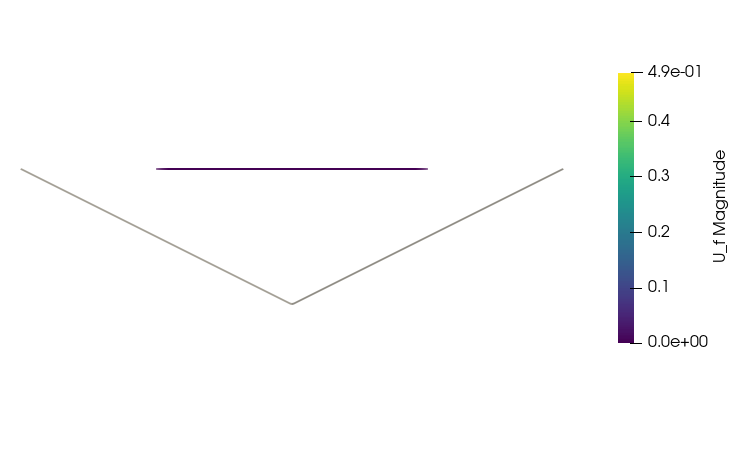}
		\subcaption{$\ell=0$}
	\end{subfigure}%
	\hspace{0.5cm}
	\begin{subfigure}[t]{0.45\linewidth}
		\includegraphics[width=1.0\linewidth]{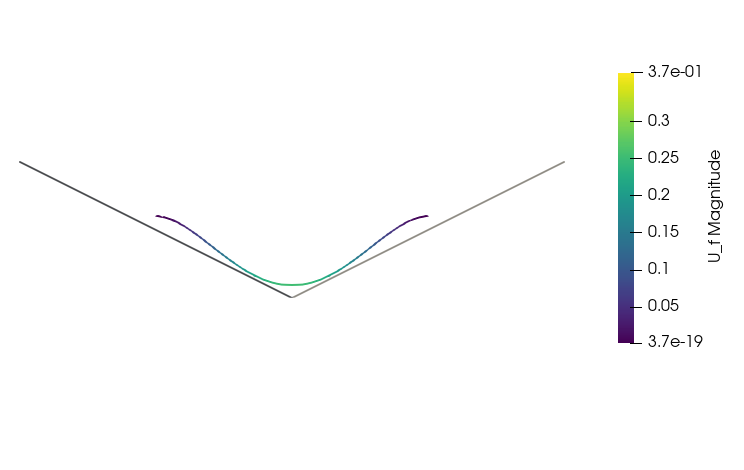}
		\subcaption{$\ell=2$}
	\end{subfigure}%
\end{figure}

\begin{figure}[H]
	\ContinuedFloat
	\centering
	\begin{subfigure}[t]{0.45\linewidth}
		\includegraphics[width=\linewidth]{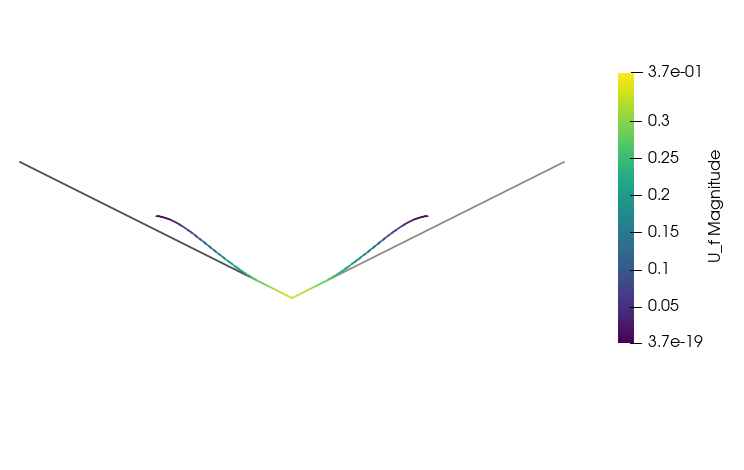}
		\subcaption{$\ell=20$}
	\end{subfigure}%
	\hspace{0.5cm}
	\begin{subfigure}[t]{0.45\linewidth}
		\includegraphics[width=1.0\linewidth]{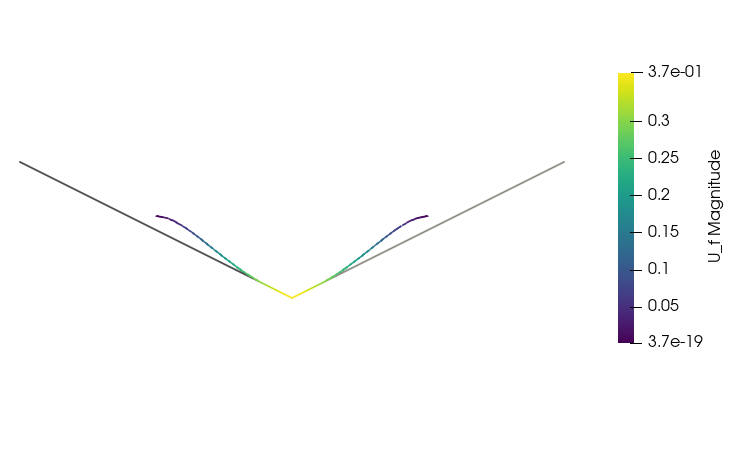}
		\subcaption{$\ell=39$}
	\end{subfigure}%
	\caption{Cross sections of a deformed membrane shell subjected not to cross a given convex obstacle.
		Given $0<h<<1$ and $0<q<1$ we observe that as the applied body force magnitude increases the contact area increases.}
	\label{fig:4}
\end{figure}

\end{document}